\documentclass{article}
\usepackage{tikz}
\usepackage{amsmath,amsfonts,amsthm,amssymb,latexsym,amscd,mathtools,enumerate}
\usepackage{xspace}
\input amssym.def
\input amssym.tex
\def\newline{\hfill\break}
\def\scong{{\scriptstyle\|}\lower.2ex\hbox{$\wr$}}
\def\Z{{\Bbb Z}}

\def\Q{{\Bbb Q}}

\def\tr{\mathop{\rm tr}\nolimits}
\def\End{\mathop{\rm End}\nolimits}

\def\Hom{\mathop{\rm Hom}\nolimits}
\def\Ind{\mathop{\rm Ind}\nolimits}

\def\Pic{\mathop{\rm Pic}\nolimits}

\def\Gal{\mathop{\rm Gal}\nolimits}
\def\Br{\mathop{\rm Br}\nolimits}

\def\Spec{\mathop{\rm Spec}\nolimits}

\def\Aut{\mathop{\rm Aut}\nolimits}

\def\Cor{\mathop{\rm Cor}\nolimits}

\def\adj{\mathop{\rm adj}\nolimits}

\def\rtimes{\mathop{\times\!\!{\raise.2ex\hbox{$\scriptscriptstyle|$}}}
	\nolimits} 

\outer\def\Demo #1. #2\par{\medbreak\noindent {\it#1.\enspace}
	{\rm#2}\par\ifdim\lastskip<\medskipamount\removelastskip
	\penalty55\medskip\fi}

\def\Br{\mathop{\rm Br}\nolimits}
\def\Cor{\mathop{\rm Cor}\nolimits}
\def\tr{\mathop{\rm tr}\nolimits}

\def\hangbox to #1 #2{\vskip1pt\hangindent #1\noindent \hbox to #1{#2}$\!\!$}
 
\title{Mixed Characterstic Cyclic Matters}
\author{David J. Saltman\\  
Center for Communications Research\\
805 Bunn Drive\\
Princeton, NJ 08540}

\newtheorem{theorem}{Theorem}[section]
\newtheorem{corollary}[theorem]{Corollary}
\newtheorem{lemma}[theorem]{Lemma}
\newtheorem{proposition}[theorem]{Proposition}

\begin{document}
\maketitle
\large

\begin{abstract}
The Artin-Schreier polynomial $Z^p - Z - a$ 
is very well known. Polynomials of this type describe 
all degree $p$ (cyclic) Galois extensions over 
any commutative ring of characteristic $p$. 
Equally attractive is the associated Galois action. 
If $\theta$ is a root then $\sigma(\theta) = \theta 
+ 1$ generates the Galois group. Less well known, 
but equally general, is the so called ``differential 
crossed product" Azumaya algebra generated 
by $x,y$ subject to the relations $xy - yx = 1$ and $a = x^p$, 
$b = y^p$ are central. In characteristic $p$ these 
algebras are always Azumaya and algebras of this sort 
generate the $p$ torsion subgroup of the Brauer 
group of any commutative ring (of characteristic 
$p$). 

It is not possible for there to be descriptions 
this general in mixed characteristic $0,p$ 
but we can come close. In Galois theory 
we define degree $p$ Galois extensions 
with Galois action $\sigma(\theta) = \rho\theta + 1$, 
where $\rho$ is a primitive $p$ root of one. 
We define an associated polynomial $Z^p + g(Z) - a$ 
and observe that the extension is Galois when 
$1 + a\eta^p$ is invertible, where 
$\eta = \rho - 1$. The Azumaya algebra analog is generated by $x,y$ subject to the relations 
$xy - \rho{y}x = 1$ 
and $x^p = a$, $y^b = b$ are central. 
These algebras are Azumaya when $1 + ab\eta^p$ 
is invertible. 

The strength of the above constructions can be 
codified by lifting results. We get 
characteristic $0$ to characteristic $p$ 
surjectivity for degree $p$ Galois extensions 
and exponent $p$ Brauer group elements in quite 
general circumstances. Obviously we want to 
get similar results for degree $p^n$ cyclic extensions 
and exponent $p^n$ Brauer group elements, 
and mostly we accomplish this though $p = 2$ 
is a special case. We also give results without 
assumptions about $p$ roots of one. Along the way 
we we introduce some new results and machinery. 
We prove a generalization to commutative rings 
of the classical Albert criterion for extension 
of cyclic Galois extensions. We analyze 
the degree $p$ Azumaya algebras described above 
in a new way and define generalizations called 
almost cyclic algebras. A recurring theme in this 
project is the additive Galois module structure 
of our extensions, and the use we make of a 
specific fiber product associated to cyclic 
group algebras of dimension $p$.    

\end{abstract}

\vfill\eject
\section{Introduction} 
 
This is an extended paper on two related themes. First of all, 
we are interested in a description of cyclic Galois extensions 
of commutative rings in mixed characteristic. More specifically, 
we are interested in cyclic Galois extensions of degree $p^n$ 
in characteristic $0$ and $p$. Now in characteristic $p$ 
there is such a theory which, starting from the Artin-Schreier 
polynomial $Z^p - Z - a$, gives a description of all cyclic Galois 
$S/R$ of degree $p^n$ where $pR = 0$. In fact there is such a theory 
for arbitrary $p$ group Galois extensions \cite{S1981} but here we restrict 
ourselves to cyclic $p$ groups. For reasons that will become clear 
in this paper, there really is no such theory possible in 
arbitrary characteristic, without restricting $S/R$ somehow. 
Our goal here is a general enough theory in mixed characteristic 
so that we can prove lifting. That is, we will prove: 
\bigskip

\noindent{\bf Theorem \ref{liftwithrho}} 
Let $p > 2$ be prime. 
Suppose $R$ is a commutative ring with $1 + pR \subset R^*$. 
Let $\hat R = R/pR$ and assume $\hat S/\hat R$ is a cyclic Galois 
extension of degree $p^n$. Then there is a cyclic Galois 
extension $S/R$ such that $S \otimes_R \hat R \cong \hat S$.

\bigskip

Note that there must be some restriction on $R$. Let 
$F_p = \Z/p\Z$ be the field of $p$ elements. 
If we consider the epimorphism $\Z[x] \to F_p[x]$, 
then 
$F_p[x]$ has plenty of cyclic Galois extensions 
and $\Z[x]$ has none. Also note that $p \not= 2$ 
is necessary. The key fact in \cite{W} (see also \cite{S1981}) 
was that 
if $R$ is the localization of $\Z$ at $2$, and 
$R \to F_2$ the canonical quotient map, then  
any cyclic Galois extension $\hat S/F_2$ of degree 
$2^n$, for $n \geq 3$, has no lift to $R$. 
The underlying fact is that adjoining the 
$2^n$ roots of one is not a cyclic extension. 

There is, however, a version of the above result for 
$p = 2$. 

\bigskip

\noindent{\bf Theorem \ref{2lift}}
Suppose $R$ is a commutative ring with $1 + 2R \subset R^*$. 
Let $\hat R = R/2R$ and assume $\hat S/\hat R$ is a cyclic Galois 
extension of degree $2^n$. Let 
$\hat T = \Ind_H^G(\hat S/\hat R)$ be the induced cyclic Galois 
extension of degree $2^{n+1}$ defined in \ref{induced} 
below. Then there is a cyclic Galois 
extension $T/R$ such that $T \otimes_R \hat R \cong \hat T$.

\bigskip

In \cite{S2022} we developed a theory of cyclic Galois extensions 
of degree $p$, which will be repeated here, but with a 
different slant. Here we will emphasize the Galois 
module structure. More specifically, if $S/R$ is $G$ Galois, 
we observe next that $S$ is a rank one projective module over 
over the group ring $R[G]$, which we will express as 
$S \in \Pic(R[G])$. The $R[G]$ structure of $S$ becomes 
crucial, and the description in \cite{S2022} is based on an assumption 
of trivial $R[G]$ structure. 

\begin{lemma}\label{proj}
Let $S/R$ be a $G$ Galois extension of commutative 
rings, for $G$ any finite group. 
Then $S$ is projective as an $R[G]$ module. 
\end{lemma}

\begin{proof} 
By faithful flatness we may assume $R$ is local. 
Thus it suffices to show $S \cong R[G]$ as a module when $R$ is local 
with maximal ideal $M$. By Nakayama's lemma it suffices to 
find an $s \in S$ such that $s + MS$ generates $S/MS$ 
as a module over $R/M[G]$. That is, we may assume 
$R = F$ is a field and we want $S \cong F[G]$. 
If $S$ were a field this is standard (e.g., \cite[p.~294]{J}) and usually 
called the normal basis theorem. If $S$ is not a field, it is well known that 
$S \cong \Ind_H^G(K)$ (\ref{induced}) where $H \subset G$ is a subgroup 
and $K$ is a field. The normal basis for $K$ then induces one for 
$S$. 
\end{proof}

Turning this around, the question of $R[G]$ module structure 
seems a key reason that characteristic $p$ is better behaved. 
The following lemma and theorem are easy and one would assume well known 
to many. 
Let $I \subset R[G]$ be the augmentation ideal. 
That is, $I$ is generated as an $R[G]$ ideal by all 
$g - 1$ for $g \in G$. We start with: 

\begin{lemma}
Let $G$ be a finite $p$ group and assume $pR = 0$. 
If $S/R$ is $G$ Galois then $S/IS \cong R$. 
$I$ is a nilpotent ideal. 
\end{lemma}

\begin{proof} 
There is a natural map $R \to S \to S/IS$ we need to prove 
an isomorphism. By descent we may assume $R$ is local 
and hence $S \cong R[G]$, and now it is clear. 
As for the second statement, it suffices to prove that there is a 
$N$ such that any product of $N$ elements of the form 
$g - 1$ is $0$. Thus it suffices to prove this 
when $R = \Z/p\Z$ a field and here it is extremely well known. 
To review the standard proof, $I$ must be the Jacobson radical 
and $F_p[G]$ is a finite dimensional algebra. 
\end{proof}

Now we turn to: 

\begin{theorem} 
Let $G$ be a finite $p$ group and $R$ a commutative 
ring with $pR = 0$. Suppose $S/R$ is Galois with group 
$G$. Then $S \cong R[G]$. 
\end{theorem}

\begin{proof} 
Since $I$ is nilpotent, we can apply Nakayama's lemma. 
We know $S/IS \cong R$. Thus $1 + IS$ generates $S$ 
over $R[G]$. Any $R[G]$ module surjection $R[G] \to S$ must be 
an isomorphism because the domain and range are projective 
of the same rank. 
\end{proof} 

After considering mixed characteristic cyclic Galois 
extensions, it is natural (for this author) to consider 
mixed characteristic cyclic Azumaya algebras. Once again, 
a beginning was accomplished in \cite{S2022}. In characteristic $p > 0$, 
the analogue (perhaps) of classical cyclic algebras are the 
so called differential crossed product algebras of \cite[p.~44]{KOS}. 
These are degree $p$ Azumaya algebras over, say, $R$ generated 
by $x,y$ with $x^p = a \in R$, $y^p = b \in R$, and $xy - yx = 1$ 
(remember $pR = 0$). We write this algebra as $(a,b)$. 
In contrast to the characteristic $0$ case, 
these algebras are Azumaya for any choice of $a$ and $b$. 
Furthermore, it was shown in \cite[p.~46]{KOS} that the $p$ torsion part of the 
Brauer group $\Br(R)$ is generated by Brauer classes $[(a,b)]$. 
It thus seems that these are the algebras to generalize to mixed 
characteristic. 

The appropriate ``generalized" algebra is easy to write down 
in a special case, which we do so in full in \ref{almostrho} but we 
summarize here. The special case is that $R$ is a commutative ring containing 
a primitive $p$ root of one. By this we mean the following. 
Let $\rho$ be a primitive $p$ root of one over the rational field 
$\Q$. Then $\Z[\rho]$ is the ring of integers of $\Q(\rho)$. 
To say ``$R$ contains a primitive $p$ root of one" really means 
that $R$ is an algebra over $\Z[\rho]$ or that there is a fixed 
homomorphism $\Z[\rho] \to R$. Note that if $pR = 0$ then $R$ 
is automatically a $\Z[\rho]$ algebra through the map 
$\Z[\rho] \to F_p \subset R$. 
Let $R$ contain $\rho$. Define the algebra $(a,b)_{\rho}$ 
as the algebra over $R$ generated by $x,y$ subject to $x^p = a$, 
$y^p = b$ and $xy - \rho{yx} = 1$. Note that we think of 
$\rho$ as an element of $R$ even though, if say $pR = 0$, then 
$\rho$ is $1$. Finally note that $\eta = \rho - 1$ is a prime 
element of $\Z[\rho]$ and $\Z[\rho]/\eta\Z[\rho] \cong F_p$. 
Given $a \in R$, write $\hat a$ to be the image of $a$ in 
$\hat R = R/\eta{R}$. 

We observed in \cite{S2022} that (see also \ref{almostrho}): 

\begin{theorem} 
$A = (x,y)_{\rho}$ is free as a module over $R$ of rank 
$p^2$. $A/R$ is Azumaya if and only if  $1 + ab\eta^p \in R^*$. 
$A/\eta{A} \cong (\hat a, \hat b)$ as an Azumaya algebra 
over $\hat R$. Suppose $1 + \eta{R} \subset R^*$. 
Then the natural map $\Br(R) \to \Br(\hat R)$ is surjective 
when restricted to elements of order $p$. 
\end{theorem} 

One of the motivating questions for this paper was to 
generalize the above result to all elements of order 
$p^r$ and to remove the root of unity assumption. 
This we accomplished. 

\bigskip

\noindent {\bf Theorem \ref{braueronto}} 
{\emph Suppose $R$ is a commutative ring and $1 + pR \subset R^*$. 
Then the natural map $\Br(R) \to \Br(R/pR)$ is surjective 
on elements of $p$ power order.} 

\bigskip

To prove the above result requires that we generalize 
the $(a,b)_{\rho}$ construction. To start with, 
one needs to generalize the characteristic $p$ 
$(a,b)$ construction. More specifically, 
we need the following. In \cite[p.~37]{KOS} it is proved that if 
$pR = 0$ then $\Br(R)$ is $p$ divisible. 
To lift arbitrary $p$ primary Brauer group 
elements in mixed characteristic we, first of all, 
need to make the above fact more concrete. 
That is, we need to construct, for any $p^r$, 
Azumaya algebras  $A/R$ such that $p^r[A] = [(a,b)]$ 
for any $a,b$. We then argue that to show 
$\Br(R) \to \Br(R/pR)$ surjective it suffices to 
lift the algebras $(a,b)$ and then the algebras 
$A$ as above for any $p^r$. 

In order to accomplish the construction of such $A$, we first 
had to understand the algebras $B = (a,b)$ and $B = (a,b)_{\rho}$ 
in a new way. 
Let $x,y \in B$ be as in the definition. As part of 
\ref{almostrho} we show that 

\bigskip

\noindent{\bf Lemma} {\emph If $x,y,B$ are as above, then $\alpha = xy$ 
is such that $S = R[\alpha]$ is cyclic over $R$. 
Moreover, $R[\alpha]$ splits $B$.}

\bigskip

The above result suggests there is a general category 
of ``almost cyclic" Azumaya algebras and one goal of this work 
is to define such algebras, construct useful examples 
of such algebras, and then use these algebras to prove the 
promised surjectivity theorem. 

It may be helpful to point out one more aspect of the arguments 
in this paper. Frequently we will perform our constructions 
over localized polynomials rings $R$, and then make our 
constructions general via specialization. This will provide 
us the convenience of working with regular rings $R$ 
whose Picard group we can get a handle on. 

We begin, in the next section, with some 
elementary facts about roots of unity. 
In this paper we denote by $R^*$ the group of units 
of a ring $R$ and $R[Z]/(f(Z)) = R[Z]/f(Z)R[Z]$ 
the formal adjunction of a root of the polynomial 
$f(Z)$ to $R$. We let $\Pic(R)$ be the Picard group, i.e. 
the group  of rank one projective $R$ modules up to isomorphism. 
We will write $J \in \Pic(R)$ to mean $J$ is a rank one 
projective $R$ module. If $A/R$ is Azumaya and 
of rank $n^2$ over $R$ we say $A/R$ has degree $n$. 

In this whole paper $p$ will be a prime, including the 
case $p = 2$. However, the $p = 2$ case is different 
from the others and this will be reflected in 
our choice of notation for roots of unity. 
We are trying to avoid separate (but not really different) 
arguments in the $p = 2$ case but this makes for 
other complications. In all this paper, when $p > 2$, 
we define $\rho = \mu$ to be a fixed choice of primitive 
$p$ root of one. 
In some sections we need the base ring in 
the $p = 2$ case to be $\Z[i]$ where $i^2 = -1$ 
and so when $p = 2$ we set $\mu = i$ . 
In other sections in the $p = 2$ case we need 
deal with the roots of unity group of order 2 and 
then $\rho = -1$. However, we feel that, without 
confusion, we can always set $\eta$ to be a choice of prime 
totally ramified over $p$ in $\Z[\rho]$, $\Z[i]$, or $\Z$. 
Specifically, $\eta = \rho - 1$, $i - 1$, and $-2$ respectively. 

We summarize the above cases and notation choices as: 

\bigskip 

{\bf Case A:}  The prime $p > 2$ and $\rho = \mu$ is a fixed choice of a primitive 
$p$ root of one. $\eta = \rho - 1$ is a prime of $\Z[\rho]$ 
totally ramified over $p$. 

\bigskip 

{\bf Case B:}  The prime $p = 2$ and $\mu$ is a square root of $-1$. 
When we are only talking about this case we will use $\mu = i$. 
Once again $\eta = i - 1$ is a choice of a prime in $\Z[i]$ totally ramified 
over 2. 

\bigskip 

{\bf Case C:} The prime $p = 2$ and $\rho$ is a primitive square root of one 
or $\rho = -1$. Obviously $\Z[\rho] = \Z$. Now $\eta = \rho - 1 = -2$. 

\bigskip

In all cases $R$ will be an flat algebra over the three base rings, 
$\Z[\rho]$, $\Z[i]$ or $\Z$. The point of these notation choices 
is that we can combine cases. In case A and B, $\mu$ is the 
root of unity and in case A and C $\rho$ is the root of unity.

\section{Computations} 

In this section we perform some computations 
involving roots of unity in commutative rings $R$. 
In this section we will be considering cases A and B, 
so $\mu$ is a primitive $p$ root of one when 
$p > 2$ and $\mu = i$ is a square root of $-1$ 
when $p = 2$.   
Of course $\Z[\mu] = \Z[Z]/(Z^{p-1} + \cdots + 1)$ 
or $\Z[\mu] = \Z[i] = \Z[Z]/(Z^2 + 1)$.

We next observe, concretely, that 
$\eta$ is totally ramified over $p$ or $2$ respectively. 

\begin{lemma}\label{total}
In case B, $\eta^2 = (i - 1)^2 = -2i$. In case A 
$\eta^{p-1} = up$ where $u \in \Z[\mu]^*$ and 
$u$ is congruent to $-1$ modulo $\eta$. If 
$\delta$ generates the Galois group of $\Q(\mu)/\Q$ 
or $\Q(i)/\Q$, 
then the norm $N_{\delta}(\eta) = p$ or $-2$. 
The map $\Z[\mu]^* \to F_p^*$ is surjective. 
\end{lemma}

\begin{proof} 
When $p = 2$ all these facts are trivial. 
If $p > 2$ then 
$1 = \mu^p = (\eta + 1)^p = \eta^p + p\eta(-u) + 1$ 
for $u \in \Z[\mu]$ so $\eta^{p-1} = pu$. 
We compute $u$ in more detail using the binomial 
theorem and get $-u =  1 + \eta{z}$ for some $z$. 
It follows that $\eta$ is the unique prime over $p$, 
and hence $u$ is divisible by no primes, implying 
$u \in \Z[\mu]^*$. Clearly 
$N_{\delta}(\eta) = \pm p$. Also, $\mu^i - 1 = \eta(1 + \cdots + \mu^{i-1})$ 
which we write as $\eta\epsilon_i$ 
and note that $\epsilon_i$ is also a unit and maps to 
$i$ modulo $\eta$. Of course $N_{\delta}(\eta) = 
\eta^{p-1}\prod_i \epsilon_i$ so we are done 
once we observe that $\prod_i \epsilon_i$ is congruent 
to $(p-1)!$ modulo $\eta$ which is congruent to 
$-1$ by Wilson's Theorem. 
\end{proof} 

We need to perform some computations involving traces 
and norms in algebras over $\Z[\mu]$. 
We choose $\mu_m$ such that $\mu_m^{p^m} = 
\mu$ and $\Q(\mu_m)/\Q(\mu)$ has degree 
$p^m$. Note that this means $\mu_m$ is a primitive 
$p^{m+1}$ root of one when $p > 2$ and a primitive 
$2^{m+2}$ root of one when $p = 2$. We have set $\eta = \mu - 1$ and 
we now set $\eta_m = \mu_m - 1$. 

In case A, the Galois group 
of $\Q(\mu_m)/\Q(\mu)$ is generated by 
$\tau$ where $\tau(\mu_m) = \mu_m^r$ and $r = p+1$. 
In case B, the Galois group of $\Q(\mu_m)/\Q(\mu) =
\Q(\mu_m)/\Q(i)$ is generated by $\tau$ where 
$\tau(\mu_m) = \mu_m^r$ and $r = 5$. 

In both cases, let $R$ be flat $\Z[\mu]$ algebra 
(so $\Z[\mu] \to R$ is injective) and a regular domain. 
We assume $R/\eta{R}$ is a domain and so $\eta$ 
is a prime element of $R$. Set 
$R_m = R \otimes_{\Z[\mu]} \Z[\mu_m] \supset R$ 
and assume $R_m$ is a regular domain. 
Then $\eta_mR_m = R \otimes_{\Z[\mu]} \eta_m\Z[\mu_m]$. 
By flatness 
$R_m/\eta_mR_m = R \otimes_{Z[\mu]} \Z[\mu_m]/\eta_m\Z[\mu_m] = 
R \otimes \Z[\mu]/\eta\Z[\mu] = R/\eta{R}$. 
Thus $\mu_m$ is a prime element of $R_m$. 
Since $R \to R_m \to R_m/\eta_mR_m \cong R/\eta{R}$ 
is the standard map, $\eta_mR_m \cap R = \eta{R}$. 

The automorphism $\tau$ acts 
on $R_m = R \otimes_{\Z[\mu]} \Z[\mu_m]$ by 
acting on $\Z[\mu_m]$. Clearly the fixed ring 
$R_m^{\tau} = R$. 
We set $N_{\tau}:R_m \to R$ to be the associated norm.  
Note that we have arranged our notation so that 
$\tau$ has order $p^m$ in both cases. 
Also, $\Z[\mu_m]$ is the ring of integers of $\Q(\mu_m)$ 
and $1,\mu_m,\ldots,\mu_m^{p^m - 1}$ is a basis 
of $\Z[\mu_m]/\Z[\mu]$.

For any $X$, $N_{\tau}(1 + X) = \sum_{i=0}^n s_i(X)$ 
where in the expansion of $N_{\tau}(1 + X) = 
\prod_i (1 + \tau^i(X))$, $s_i(X)$ is the sum of the degree 
$i$ monomials in $X$ and its $\tau^i$ conjugates. 
Thus $s_n(X) = N_{\tau}(X)$, $s_1 = \tr_{\tau}(X)$ etc. 
We will also make use of Newton's identity, which 
in the form we need is: 
$$ks_k(x) = \tr_{\tau}(x^k) + \sum_{i=1}^{k-1} (-1)^{i-1}s_i(x)\tr_{\tau}(x^{k-i})$$  
for $1 \leq k \leq n$. 
Since $Z^{p^m} - \mu$ is the minimal polynomial of $\mu_m$ over $\Z[\mu]$, it follows that $s_k(\mu_m) = 0$ for 
$1 \leq k \leq p^m-1$ and so $\tr_{\tau}(\mu_m^k) = 0$ 
for $1 \leq k \leq p^m-1$. From this we conclude that: 

\begin{lemma}\label{etaandp} 
$\tr_{\tau}(R_m) = p^mR$. 
Furthermore, $N_{\tau}(\eta_m) = \eta$ in case A 
and $N_{\tau}(\eta_m) = -i + 1 = -\eta$ in case B. 
For $p$ odd, $\eta_m^{p^m} = \eta(1 + \eta^{p-2}\eta_mv)$ for some 
$v \in \Z[\mu_m]$. When $p = 2$, $\eta_m^{2^m} = 
(i + 1)(1 + (i+1)\mu_mv)$. 
\end{lemma}

\begin{proof} 
The first two facts are above. 
From our description of the minimal polynomials of $\mu_m$ 
in either case, we have $N_{\tau}(\mu_m) = \mu$ or 
$N_{\tau}(\mu_m) = -i$ respectively. 
Thus $N_{\tau}(\eta_m) = N_{\tau}(\mu_m - 1) = 
N_{\tau}(\mu_m) + \sum_{i=1}^{n-1} (-1)^is_i(\mu_m) + (-1)^n = 
\mu - 1$ when $p$ is odd and $-i + 1$ when $p$ is even. For odd 
$p$, 
$(\mu_m - 1)^{p^m} = \mu  - 1 + \sum_1^{p^m-1} {p^m \choose i}(-1)^i\mu_m^{p^m-i} 
= \eta + pv'$. In the summation if we pair the $i$ and $p^m - i$ terms 
we see that $v'$ is divisible by $\eta_m$ and we have 
$\eta_m^{p^m} = \eta(1 + \eta^{p-2}\eta_mv)$ for some $v$. The $p = 2$ 
case is similar.
\end{proof} 

We require more information about the traces of 
the powers of $\eta_m$. 

\begin{lemma} 
For any $p$, and $0 \leq r < p^m$, then 
$\tr_{\tau}(\eta_m^{r + sp^m})$ 
is exactly divisible by $\eta^sp^m$. 
If $p$ is odd this is $\eta^{s+m(p-1)}$ and if 
$p = 2$ this is $(i+1)^{s + 2m}$. 
\end{lemma}

\begin{proof} 
$\eta_m^{r + sp^m} = \eta_m^r(\eta(1 + \eta^{p-2}\eta_mv))^s =   
\eta^sZ$ where $Z \in R_m$. This proves the divisbility. 
More precisely, $Z = \eta_m^r + \eta_m^{r+1}\eta^{p-2}v$ 
and $p^m$ divides $\tr_{\tau}(\eta_m^r)$ exactly while 
$\eta^{p-2}p^m$ divides $\tr_{\tau}(\eta^{p-2}\eta_m^{r+1}v)$. The $p = 2$ 
case is very similar. 
\end{proof} 

Another way to state the above result 
is the following. For any $b \in R$, let 
$v(b)$ be the highest power of $\eta$ 
dividing $b$. Note that $\eta$ is a prime element of 
$R$ and hence $v(b)$ is the valuation of $b$ in the 
the localization of $R$ at $\eta{R}$ which is a discrete 
valuation domain. Let $c$ be $v(p^m)$.   
Then $c = m(p-1)$ when $p$ is odd and 
and $c = 2m$ when $p$ is even. 
The above lemma says that $v(\tr_{\tau}(\eta_m^N)) = \lfloor N/n \rfloor + c$ 
where $\lfloor N/n \rfloor$ is the greatest integer 
less than or equal to $N/n$. 
Since $\lfloor a \rfloor + \lfloor b \rfloor \geq \lfloor a+b \rfloor - 1$, 
we have $(\lfloor a \rfloor + c) + (\lfloor b \rfloor + c) \geq (\lfloor a+b \rfloor+ c) + c - 1$. 
It follows that: 

\begin{lemma}\label{summingvs}
If  
$C = m(p-1) - 1$ for $p$ odd and $C = 2m - 1$ for 
$p = 2$ we claim,   
$$v(\tr_{\tau}(\eta_m^a)\tr_{\tau}(\eta_m^b)) = 
v(\tr_{\tau}(\eta_m^a)) + v(\tr_{\tau}(\eta_m^b)) 
\geq v(\tr_{\tau}(\eta_m^{a+b})) + C$$ 
\end{lemma} 

We use Newton's identities to compute 
$v(ks_k(\eta_m^r))$ and note that only $v(\tr_{\tau}(\eta_m^{ri}))$ matters. 

\begin{lemma}\label{leading} 
We have, 
$$v(ks_k(\eta_m^r)) = v(\tr_{\tau}(\eta_m^{kr}))$$ 
and 
$$v\left(\sum_{i=1}^{k-1} (-1)^{i-1}s_i\left(\eta_m^r\right)\tr_{\tau}\left(x^{k-i}\right)\right) > v\left(ks_k\left(\eta_m^r\right)\right).$$ 
\end{lemma} 

\begin{proof} 
We prove this by induction. The $k = 1$ case is trivial. 
It suffices to show that 
for all $i < k$, $v(s_i(\eta_m^r)) + v(\tr_{\tau}(\eta_m^{r(k-i)}) > 
v(\tr_{\tau}(\eta_m^{kr}))$. Let $d = v(p)$ 
and let $C$ be as in \ref{summingvs}. Since 
$i < k < p^m$, $v(k) \leq d(m-1)$ and so 
by induction 
$$v(s_i(\eta_m^r)) + v(\tr_{\tau}(\eta_m^{r(k-i)}) = 
v(\tr_{\tau}(\eta_m^{ri})) - v(i) + 
v(\tr_{\tau}(\eta_m^{r(k-i)}))$$ 
which by the above lemma is greater than or equal to 
$v(\tr_{\tau}(\eta_m^{rk})) + C - v(i)$. 
When $p = 2$, $C - v(i) \geq 2m - 1 - 2(m-1) = 1 > 0$ 
and when $p > 2$ $C - v(i) \geq m(p-1) - 1 - (p-1)(m-1) = 
p - 2 > 0$ as needed. 
\end{proof} 

Combining \ref{leading} and \ref{etaandp} 
we have: 

\begin{proposition}\label{formula} 
Let $k < p^m$ and write $k = ap^s$ where $p$ does not divide $a$. 
Assume $b \in R$ and $v(b) = 0$. Then 
$v(s_k(b\eta_m^r)) = \lfloor ra/p^{m-s} \rfloor + (m - s)c$ 
where $c = p - 1$ when $p > 2$ and $c = 2$ when $p = 2$. 
If $n = v(s_k(\eta_m^r))$, then 
$s_k(b\eta_m^r)$ is congruent to $\tr_{\tau}(\eta^{rk})/k$ 
modulo $\eta^{n + 1}$. 
\end{proposition}

\begin{proof}
The first fact is just the result of \ref{etaandp} 
and the second is by \ref{leading}.
\end{proof}

The purpose of the above calculations was to 
enable us to compute norms of the form 
$N_{\tau}(1 + b\eta_m^r)$ for $b \in R_m$ not divisible by $\eta_m$. 
We achieve this by writing 
$$N_{\tau}(1 + b\eta_m^r) = 1 + 
\sum_i^{n-1} s_i(b\eta_m^r) + N_{\tau}(b\eta_m^r)$$ 
and try and analyze the lowest 
values of $v(s_i(b\eta_m^r))$ and $v(N_{\tau}(b\eta_m^r))$. 
More precisely, if $v(bN_{\tau}(\eta_m^r))$ or $v(s_i(b\eta_m^r))$ is minimal among all the others we call $s_i(b\eta_m^r)$ 
or $N_{\tau}(b\eta_m^r)$ a low order 
term. We need to compute which are these low order terms. 
The first step is to compute some $v$ values, so we know that 
the low order terms, $X$, have $v(X)$ less or equal to that these values. For convenience, if $b \in R_m$ define $v_m(b)$ 
to be the highest power of $\eta_m$ dividing $b$. 
As in the previous case, $v_m(b)$ is the valuation of 
$b$ in the discrete valuation domain which is the localization 
of $R_m$ at $\eta_mR_m$. 

The norm is the easiest part. It is obvious 
that: 

\begin{lemma}\label{normvalue}
Suppose $b \in R_m$ satisfies $v_m(b) = s$. 
$N_{\tau}(b\eta_m^r) = N_{\tau}(b)\eta^r$ and 
$v(N_{\tau}(b\eta_m^r)) = r + s$. 
\end{lemma} 

One  consequence of the above result is that 
any low order term $X$ when $v_m(b) = 0$ has to have $v(X) \leq r$. 

We need to understand the difference between 
$v(s_k(\eta_m^r))$, about which we have exact information, 
and the more general $v(s_k(b\eta_m^r))$ where $b \in R_m$. 
Of course if $b \in R$ and $v(b) = 0$  
then $s_k(b\eta_m^r) = b^ks_k(\eta_m^r)$ so 
$v(s_k(b\eta_m^r)) = v(s_k(\eta_m^r))$. 
We claim: 

\begin{proposition}
For any $b \in R_m$, $v(\tr_{\tau}(b\eta_m^a)) \geq 
v(\tr_{\tau}(\eta_m^a))$ and $v(s_k(b\eta_m^r)) \geq 
v(s_k(\eta_m^r))$. 
\end{proposition}

\begin{proof}
We use the first statement to prove the second. 
Starting with the first, write $a = r + sp^m$ where $r < p^m$, 
so $\eta_m^a = \eta^sZ$ for $Z \in R_m$. Thus 
$b\eta_m^a = \eta^s(bZ)$ and $v(\tr_{\tau}(bZ)) \geq v(p^m)$ 
which proves the first statement. 

As for the second, assume by induction the result for 
$i < k$. It suffices to show that 
$v(ks_k(b\eta_m^r)) \geq v(ks_k(\eta_m^r))$.  
Now:  
$$ks_k(b\eta_m^r) = \tr_{\tau}(b^k\eta_m^{kr}) + 
\Sigma_{i=1}^{k-1} (-1)^{i-1}s_i(b\eta_m^r)\tr_{\tau}(b^{k-i}\eta_m^{r(k-i)})$$ 
and $v(\tr_{\tau}(b^k\eta_m^{kr})) \geq v(\tr_{\tau}(\eta_m^{kr})) = 
v(ks_k(\eta_m^r))$ while $v(s_i(b\eta_m^r)) + 
v(\tr_{\tau})b^{k-i}\eta_m^{r(k-i)} \geq v(s_i(\eta_m^r)) + 
v(\tr_{\tau}(\eta_m^{k-i}))$ which we showed in \ref{leading} was strictly 
greater than $v(ks_k(\eta_m^r))$. 
\end{proof}

When $r$ is a bit bigger than $p$ we get a different 
low $v(X)$ when $b \in R$. 

\begin{lemma}\label{tracevalue}
Suppose $p < r < 2p$, $p > 2$, $k = p^{m-1}$ and $b \in R$. 
Then 
$v(s_k(b\eta_m^r)) = p + (p^{m-1})v(b)$. If $p = 2$ then 
$v(s_{2^{m-1}}(b\eta_m^r)) = (2^{m-1})v(b) + \lfloor r/2 \rfloor + 2 > r$ 
if $r < 3$ and equals $r$ if $r = 3$. 
\end{lemma}

\begin{proof}
By \ref{formula}, 
$v(s_{p^{m-1}}(b\eta_m^r)) = p^{m-1}v(b) + \lfloor r/p \rfloor + (p-1)$ 
which equals $p$ if $v(b) = 0$. If $p = 2$, 
$v(s_{2^{m-1}}(b\eta_m^r)) = \lfloor r/2 \rfloor + 2 + 2^{m-1}v(b)$. 
\end{proof} 

Now we can pare down the possible low order terms. 

\begin{proposition}\label{highp}
Let $b \in R_m$. 
Suppose $k = ap^s$, $a$ is prime to $p$, 
and $s < m - 1$.    
Then $v(s_k(b\eta_m^r)) \geq 2(p-1)$ when $p > 2$ 
and $4$ when $p = 2$. If $v_m(b) = 0$ and $r \leq p$ 
for $p > 2$ and $r < 4$ when $p = 2$ 
then $v(s_k(b\eta_m^r))$ is not a low order term.  
If $b \in R$, $p < r < 2p$ and $p > 2$, then 
$s_k(b\eta_m^r) \geq kv(b) + 2(p-1)$ and if 
$v(b) = 0$ this is not a low order term. If 
$p = 2$, $b \in R_m$ satisfies $v_m(b) = 0$, 
and $r < 4$  
then $s_k(b\eta_m^r)$ is not a low order term.
\end{proposition}

\begin{proof} 
We have $v(s_k(b\eta_m^r)) \geq v(s_k(\eta_m^r)) = 
\lfloor kr/p^m \rfloor + 2(p - 1) \geq 2(p-1)$ when 
$p > 2$ and greater than or equal to $2(m-2) \geq 4$ 
when $p = 2$. Now \ref{normvalue} shows 
that $s_k(b\eta_m^r)$ is not a low order term when 
$v_m(b) = 0$. 

If $b \in R$, $p < r < 2p$ and $p > 2$ 
then $v(s_k(b\eta_m^r)) \geq kv(b) + v(s_k(\eta_m^r)) \geq 
kv(b) + \lfloor kr/p^m \rfloor + 2(p - 1) 
\geq kv(b) + 2(p - 1) > kv(b) + p$.  
If $v(b) = 0$ then by \ref{tracevalue}
$s_k(b\eta_m^r)$ is not a low order term.  
\end{proof} 

\begin{proposition}\label{higha}
Suppose $p > 2$ and $k = ap^{m-1}$ 
where $a > 1$ and $b \in R_m$. Then $v(s_k(b\eta_m^r)) 
\geq p - 1$. If $r \geq p/2$, 
$v(s_k(b\eta_m^r)) \geq p$. 
If $r \geq p$, $v(s_k(b\eta_m^r)) \geq p + 1$. 
Suppose $v_m(b) = 0$. 
If $r \leq p$, then  
$s_k(b\eta_m^r)$ is not a low order term. 
If $b \in R$ has $v(b) = 0$, then  
$s_k(b\eta_m^r)$ is not a low order term. 
\end{proposition}

\begin{proof} 
$v(s_k(b\eta_m^r)) \geq v(s_k(\eta_m^r)) 
\geq \lfloor 2r/p \rfloor + p - 1$. This 
shows the inequalities. Assume $v_m(b) = 0$. 
By \ref{normvalue} if $s_k(b\eta_m^r)$ is a low order 
term, then $\lfloor 2r/p \rfloor + (p - 1) \leq r$. 
In particular $r \geq p - 1$. But since $p > 2$, 
$p - 1 > p/2$ so $r = p - 1$ is impossible. 
If $r = p$,  $2 + (p - 1) = p + 1 > r = p$ 
making this impossible also.  
If $b \in R$, $v(b) = 0$, $r > p$, and 
$s_k(b\eta_m^r)$ is a low order term then  
then $\lfloor 2r/p \rfloor + p - 1 \leq p$ by \ref{tracevalue}, 
and this is impossible also. 
\end{proof}

We need a specific norm computation enabled 
by the above machinery. 
But before that we state the following result. 

\begin{proposition}\label{valuesmodeta} 
Let $b \in R_m$ and $b_0 \in R$ 
be such that $v_m(b - b_0) > 0$. 
\begin{enumerate}
\item $N_{\tau}(b) - b_0^{p^m}$ is divisible by 
$\eta$. \label{part:1}

\item Assume $b_0 \notin \eta{R}$. Let $r = p - 1$, 
or allow $r = 2$ when $p = 2$. 
$v(s_{p^{m-1}}(b_0\eta_m^r)) = v(s_{p^{m-1}}(\eta_m^r)) = 
r$. Modulo $\eta^{r+1}$, $s_{p^{m-1}}(b\eta_m^r)$ 
is congruent to $-b_0^{p^{m-1}}\eta^r$ 
modulo $\eta^{r + 1}$ when $p > 2$, and is divisible by $\eta^{r + 1}$ when 
$p = 2$. \label{part:2}
\end{enumerate}
\end{proposition}

\begin{proof}
To prove part \ref{part:1}, write $b = b_0 + z\eta_m$ for $z \in R_m$. 
Writing $N_{\tau}(b_0 + z\eta_m)$ out as a product, it is clear that 
it equals $b_0^{p^{m-1}} + \eta_mX$ for some $X$. 
But $\eta_mX \in \eta_mR_m \cap R = \eta{R}$ proving \ref{part:1}. 

The argument for \ref{part:2} will be more involved. 
$v(s_{p^{m-1}})(\eta_m^r) = \break 
\lfloor p^{m-1}r/p^m \rfloor + p - 1 = 
\lfloor r/p \rfloor + p - 1 = r$. 
If $i < p^{m-1}$ then $p^{m-1}$ does not divide 
$i$ and $v(s_i(b\eta_m^{p-1})) \geq v(s_i(\eta_m^{p-1}))$ 
which is strictly greater that $p$ by \ref{highp}. 
Also $v(\tr_{\tau}((b\eta_m^{p - 1})^{k - i})) \geq v(p^m)$. 
Using Newton's formula we have 
\begin{multline*}
s_{p^{m-1}}(b\eta_m^{p-1}) = \tr_{\tau}((b\eta_m^{p-1})^{p^{m-1}})/p^{m-1} \\ 
+ \sum_{i = 1}^{k-1} s_i(b\eta_m^{p-1})\tr_{\tau}((b\eta_m^{p - 1})^{k-i}))/p^{m-1}.
\end{multline*}

If $X = s_i(b\eta_m^{p-1})\tr_{\tau}(b\eta_m^{p - 1}))/p^{m-1}$ 
then $v(X) > p + v(p^m) - v(p^{m-1}) > p$. 
Thus $s_{p^{m-1}}(b\eta_m^{p-1})$ is congruent to 
$\tr_{\tau}((b\eta_m^{p-1})^{p^{m-1}})$ modulo $\eta^{r+1}$. 
Now writing $b = b_0 + \eta_mX$ 
we observe that $((b_0 + \eta_mX)\eta_m^r)^{p^{m-1}} = 
(b_0\eta_m^r + X\eta_m^{r+1})^{p^{m-1}}$ and using the binomial
expansion this has the form $b_0^{p^{m-1}}\eta_m^{rp^{m-1}} + 
X^{p^{m-1}}\eta^{(r + 1)p^{m-1}} + pX$. 
Since $v_m(p) = p^mv(p)$, when $p > 2$, $pX$ is a sum of terms 
of the form $b'\eta_m^t$ where $t \geq 2p^m$. 
When $p = 2$, $2^mv(2)$ and $pX$ 
is again the sum of terms $b\eta_m^t$ for $t \geq 2(2^m)$. 
But $v(\tr_{\tau}(b'\eta_m^t)) > p$ for such $t$. 
If $r = p - 1$, 
$v(\tr_{\tau}(X^{p^{m-1}}\eta_m^{p^m})) \geq p$ 
and so $s_{p^{m-1}}(b\eta_m^{p-1})$ is congruent to 
$\tr_{\tau}(b_0\eta_m^{p-1})^{p^{m-1}} = b_0^{p^{m-1}}\tr_{\tau}(\eta_m^{p-1})/p^{m-1}$ modulo $\eta^p$. 
Since $\eta_m^{p-1} = (-1)^{p - 1} + X$ where 
$X$ is sums of $\mu_m^r$ for $r < p$ we have 
$\tr_{\tau}(\eta_m^{p - 1})/p^{m-1} = (-1)^{p-1}p$. 
When $p$ is odd this is $(-1)^p\eta^{p - 1}$ modulo $\eta^p$  
and when $p = 2$, $v(p) = 2$ so this is divisible 
by $\eta^2$. Finally if $p = 2$ and $r = 2$, 
$v(\tr_{\tau}(X^{p^{m-1}}\eta_m^{32^{m-1}})) \geq 
\lfloor 3/2 \rfloor + 2 = 3$.  
\end{proof}

The next theorem gives norm computations 
of $1 + b\eta_m^r$ for the 
the values of $r$ we require. 

\begin{theorem}\label{normcomputations} 
Suppose $b \in R$ and write $b = b_0 + \eta_mX$ 
for some $b_0 \in R$ and $X \in R_m$. 
\begin{enumerate}
\item $N_{\tau}(1 + b\eta_m^p)$ is congruent to $1$ modulo 
$\eta^p$. When $p = 2$ $N_{\tau}(1 + b\eta_m^2)$ is congruent 
to $1 + b_0^{2^m}\eta^2$ modulo $\eta^3$. \label{itm:1}

\item $N_{\tau}(1 + b\eta_m^{p-1})$ is congruent 
to $1 + (b_0^{p^m} - b_0^{p^{m-1}})\eta^{p-1}$ modulo $\eta^p$ 
when $p > 2$ and is congruent to $1 + b_0^{2^m}\eta$ 
modulo $\eta^2$ when $p = 2$. \label{itm:2}

\item If $b$ is not divisible by $\eta_m$ and 
$r < p - 1$, then $N_{\tau}(1 + b\eta_m^r)$ is congruent to 
$1 + b_0^r\eta^r$ modulo $\eta^{r+1}$. \label{itm:3}

\item Suppose $s \leq p - 1$, $v_m(b) = 0$ and 
$v(a) = 0$. If $N_{\tau}(1 + b\eta_m^r) = 
1 + a\eta^s$ then $s = r$. \label{itm:4}

\item Suppose $p > 2$ and $r = p + 1$ and $b \in R$ with $v(b) = 0$. 
Then $N_{\tau}(1 + b\eta_m^{p+1})$ 
is congruent to $1 + -b^{p^{m-1}}\eta^{p}$ modulo $\eta^{p+1}$. \label{itm:5}
\end{enumerate}
\end{theorem} 

\begin{proof} 
We start with \ref{itm:1}. 
$N_{\tau}(1 + b\eta_m^p) = 1 + \sum_i s_i(b\eta_m^p) + 
N_{\tau}(b\eta_m^p)$. Thus to prove the first statement 
we need to show all terms $X$ except $X = 1$ satisfy $v(X) \geq p$.  It is clear from \ref{normvalue} that $v(N_{\tau}(b\eta_m^p)) 
\geq p$. Suppose $p > 2$. 
If $v(s_k(b\eta_m^p)) < p$ then by \ref{highp} $k$ 
is divisible by $p^{m-1}$. By \ref{higha} $k = p^{m-1}$. 
Finally, $v(s_{p^{m-1}}(b\eta_m^p)) \geq 
v(s_{p^{m-1}}(\eta_m^p)) = 1 + ( p - 1) = p$ finishing this 
case. 

In \ref{itm:1}, if $p = 2$ we need $v(s_k(b\eta_m^2)) > 2$ 
when $k < 2^m$.  
By \ref{highp} we can assume $k = 2^{m-1}$. 
We are done by \ref{valuesmodeta}. 

In \ref{itm:2}, we can obviously assume $v_m(b) = 0$ 
and note that\linebreak $v(N_{\tau}(b\eta_m^{p-1})) = p - 1$. 
First we show no nontrivial terms $X$ have 
$v(X) < p - 1$ and the only other term with 
$v(s_k(b\eta_m^{p-1})) = p - 1$ is $k = p^{m-1}$. 
Now $v(s_k(b\eta_m^{p-1})) \geq  v(s_k(\eta_m^{p-1})$. 
From \ref{highp} we see that $v(s_k(\eta_m^{p-1})) > p - 1$ 
unless $p^{m-1}$ divides $k$. If $k = ap^{m-1}$ 
for $a > 1$ we see that $v(s_k(b\eta_m^{p-1})) \geq p$ 
if $a > 1$. Thus $k = p^{m-1}$ is the only remaining 
case and we are done by \ref{valuesmodeta}. 

Turning to \ref{itm:3}, we need that no $s_k(b\eta_m^r)$ 
is a low order term when $r < p - 1$ and $v_m(b) = 0$. 
Note that $v(N_{\tau}(b\eta_m^r)) = r$. 
By \ref{highp} we can assume $p^{m-1}$ divides $k$. 
By \ref{higha} we can assume $k = p^{m-1}$. 
But $v(s_{p^{m-1}}(b\eta_m^r)) \geq v(s_{p^{m-1}}(\eta_m^r)) = 
\lfloor r/p \rfloor  + p - 1 > r$.  

As for \ref{itm:4}, this is clear if $s < p - 1$ or $p = 2$. If 
$s = p - 1$ then $r \geq p - 1$ by 2) but $r \geq p$ 
is excluded by 1). 

For \ref{itm:5}, $v(N_{\tau}(b\eta_m^{p+1})) = p + 1$ of course. 
However, $v(s_{p^{m-1}}(b\eta_m^{p+1})) = v(s_{p^{m-1}}(\eta_m^{p+1}) = 
\lfloor (p+1)p^{m-1}/p^m \rfloor + p - 1 = 
\lfloor 1 + 1/p \rfloor + p - 1 = p$ 
so the 
norm term is not low order. By \ref{highp} 
any low order term $s_k(b\eta_m^{r+1})$ must 
have $k$ divisible by $p^{m-1}$. 
By \ref{higha} $s_{p^{m-1}}(b\eta_m^{p+1})$ 
is the only low order term. 
Now $s_{p^{m-1}}(b\eta_m^{p+1}) = 
b^{p^{m-1}}s_{p^{m-1}}(\eta_m^{p+1})$. 
By \ref{leading}, 
$s_{p^{m-1}}(\eta_m^{p+1})$ is congruent to 
$\tr_{\tau}(\eta_m^{(p+1)p^{m-1}})/p^{m-1}$ modulo 
$\eta^{p+1}$. 
But $\eta_m^{(p+1)p^{m-1}} = \eta_m^{p^m}\eta_m^{p^{m-1}} = 
\eta(1 + \eta^{p-2}\eta_mv)\eta_m^{p^{m-1}} = 
\eta(\eta_m^{p^{m-1}} + \eta^{p-2}\eta_m^{p^{m-1} + 1}$ and so 
$\tr_{\tau}(\eta_m^{(p+1)p^{m-1}})/p^{m-1} = \eta\tr_{\tau}(\eta_m^{p^{m-1}}) + 
\eta^{p-1}\tr(X)/p^{m-1}$ for some $X$. 
This is congruent modulo $\eta^{p+1}$ to 
$\eta\tr_{\tau}(\eta_m^{p^{m-1}})/p^{m-1} = \eta{p}$ which is 
congruent to $-1$ modulo $\eta^{p+1}$. 
\end{proof}

There is a small detail we need to consider: 

\begin{lemma} 
Suppose $M \subset R$ is a prime ideal not containing 
$\eta$.  
If $b \in M$ and $N(1 + b\eta_m^r) = 
1 + z\eta^s$ where $\eta$ does not divide $z$ then 
$z \in M$. 
\end{lemma}

\begin{proof} 
Since $R_m = R[Z]/(Z^{p^m} - \mu)$, $R$ is a direct summand 
of $R_m$ and $MR_m \cap R = M$. Clearly $\tau: R_m \to R_m$ 
induces $\tau: R_m/MR_m \to R_m/MR_m$. Since 
$N_{\tau} = \prod_i \tau(X)$, we have the commutative diagram 
$$\begin{matrix}
R_m&\buildrel{N_{\tau}}\over\longrightarrow&R_m\cr
\downarrow&&\downarrow\cr
R_m/MR_m&\buildrel{N_{\tau}}\over\longrightarrow&R_m/MR_m.\cr
\end{matrix}$$ 
This implies that $N_{\tau}(1 + b\eta_m^r) \in 1 + MR_m$ 
and so $z\eta^s \in MR_m \cap R = M$. Since $M$ is prime 
and does not contain $\eta$, $z \in M$.
\end{proof}

\section{Operations}

In the rest of this paper, we will need to perform operations on cyclic 
extensions. That is, we need to operate in the group 
of degree $n$ cyclic extensions of a commutative ring. 
This is isomorphic to a cohomology group of the etale fundamental group. 
It will be more useful, however, to represent the elements of this 
group as explicit ring extensions. In this group  
are elements whose order is not $n$. 
In that case, we will be dealing with induced Galois extensions. 
If we are considering extensions of degree $n$ and 
$S/R$ is cyclic of degree $m < n$ with $m$ dividing 
$n$, then our ``group" element is an induced extension 
(defined below) 
$\Ind_H^G(S/R)$ and the distinction is crucial. 
One cannot, for example, lift degree $8$ cyclics but 
viewed as an element of the degree 16 cyclic extensions 
(via $\Ind_H^G$) they can be lifted. 
Thus we conceive 
of our operations as directly concerning cyclic extensions 
presented as algebras. This approach will also allow us to keep 
track of the Galois module structure of our extensions. 

We start with some basic observations about Galois extensions 
of commutative rings. 
Recall (e.g., \cite{S1999} or \cite{DI}) that if $T/R$ is an extension of commutative 
rings where $T$ is finitely generated projective over $R$, 
and $G$ is a finite group acting on $T$ with $R \subset T^G$, 
let $\Delta(T/R,G)$ be the twisted group ring 
$T_*[G]$ where $g \in G$ and $t \in T$ satisfy 
$gt = g(t)g$. Define $\phi: \Delta(T/R,G) \to \End_R(T)$ 
by setting $\phi(t)$ to be left multiplication by $t$ and 
$\phi(g)$ to be the given action of $g$ on $T$. 
Then $T/R$ is Galois with group $G$ if and only if $\phi$ is an 
isomorphism. We refer to \cite{S1999} and \cite{DI} 
for the basic theory of these extensions. 
One result we will need is that $T/R$ is $G$ Galois 
if and only if $T/R$ is finitely generated projective as a module 
and $T \otimes_R T \cong \oplus_{g \in G} T$ where the 
$g$ component of the map is $t \otimes t' \to tg(t')$. 
It is also easily seen to be true that if $R'/R$ 
is faithfully flat, and if $G$ acts on $T$, then 
$T/R$ is $G$ Galois if and only if $(T \otimes_R R')/R'$ 
is $G$ Galois.

We record the well known: 

\begin{lemma}\label{galoistower}
Let $G'$ be a finite group with normal subgroup $N'$. 
Suppose $G'$ acts on $T \supset S \supset R$ such that $T/S$ is $N'$ 
Galois and $S/R$ is $G'/N'$ Galois. Then $T/R$ is $G'$ Galois. 
\end{lemma} 

\begin{proof} 
$T/R$ is etale and $T$ is a projective as an $R$ module. 
For each $\bar g \in G/N$, choose a preimage $g \in G$. 
There is an idempotent 
$e_g$ such that $(g(s) \otimes 1 - 1 \otimes s)e_{g} = 
0$ for all $s \in S$. Then $(T \otimes_R T)e_{g} \cong T \otimes_{g,S} T$ 
where the $S$ action on the right $T$ is via $g$. 
As rings $T \otimes_{g,S} T \cong \oplus_N T$ where the 
projection on the $n \in N$ component is $t \otimes t' \to tgn(t')$. 
Since $T \otimes_R T$ is the direct sum of the $(T \otimes_R T)e_g$ 
we are done.
\end{proof}

If $T'/R$ is $H$ Galois and $H \subset G$ is a subgroup, define 
$\Ind_H^G(T'/S) = \Hom_{\Z[H]}(\Z[G],T')$ where the $\Hom$ 
refers to left module homomorphisms. Note that if $T = \Ind_H^G(T'/S)$, 
there is a 
``coinduced" action of $G$ on $T$. That is, if $t \in T$ we define 
$gt$ by setting $gt(x) = t(xg)$. The product 
in $T'$ makes $T$ a commutative $R$ algebra. 

\begin{proposition}\label{induced}
Suppose $T/R$ is Galois with group $G$ and $e \in T$ 
is an idempotent for which $eg(e) = 0$ when $g(e) \not= e$ and $re = 0$ implies $r = 0$ 
for all $r \in R$. Let $H = \{g \in G | g(e) = e \}$. 
Let $eT$ be the ring with $e$ as the identity. 
Then $eT/R$ is Galois with group $H$. 

Conversely, suppose $T'/R$ is $H$ Galois and 
$H \subset G$. Then\linebreak $\Ind_H^G(T')/R$ is $G$ Galois. 

Finally, if $R$ is a field then there is an $H \subset G$ such that 
$T'$ is a field, $T'/R$ is $H$ Galois, and $T \cong \Ind_H^G(T')$. 
\end{proposition}

\begin{proof} 
Let $g_1 = 1,g_2,\ldots,g_r$ be a set of right 
coset representatives of $H$ in $G$, so 
$G$ is the disjoint union of the $Hg_i$. 
Then $\Z[G] \cong \oplus \Z[H]g_i$ as 
$\Z[H]$ modules. 

$eT$ is faithfully projective as an $R$ module. 
Let $A = \Delta(T/R,G)$ which is Azumaya over $R$. 
Then $B = eAe$ is Azumaya over $R$ and equals 
$\End_R(eT)$. Clearly 
$eT \subset eAe$. If $h \in H$, $he = eh$ and so 
$eh \in B$. On the other hand $ege = eg(e)g = 0$ if 
$g \notin H$ and $e(1-e)T = 0$. Thus 
$B$ is spanned by $eT$ and the $h$'s and the result is clear. 

We consider the converse. The $1 \in T$ is 
defined by $g \to 1$ for all $g$, and so 
$R \subset T$ via left multiplication on $T'$. 
$T$ is clearly 
faithfully projective. Let $e \in T$ be the idempotent 
such that $e(H) = 1$ and $e(Hg) = 0$ if $Hg \not= H$ 
where $1 \in T$ is the identity. 
Clearly $Te \cong T'$. The idempotent $e_i = g_i^{-1}(e)$ satisfies 
$e_i(Hg_i) = 1$ 
and $e_i(Hg) = 0$ if $Hg \not= Hg_i$. 
Thus $1 = \sum e_i$ in $T$. We can consider $A = \Delta(T/R,G)$ 
and note that $eAe \cong \Delta(T'/R,H) \cong \End_R(T')$. 
The action given above makes $T$ a module over $A$. 
Now $T = \oplus_i Te_i$ and it is clear that $e_iAe_j \cong 
\Hom_R(Te_j,Te_i)$ and so $A \to \End_R(T)$ is surjective. 
Similarly if $a \in A$ maps to 0 then $e_iae_j = 0$ and 
$a = (\sum_i e_i)a(\sum e_j) = 0$. 

As for the last statement, since $T$ is a separable algebra 
over the field $R$, $T$ is a direct sum of fields. 
Let $e \in T$ be an idempotent such that $eT$ is a field 
and set $H = \{g \in G | g(e) = e\}$. 
\end{proof}

We will need to work out how the induced Galois 
operation combines with the Galois correspondence. 
Let $H \subset G$ be groups and $N \subset G$ 
a normal subgroup. Assume $S/R$ is $H$ Galois and form 
the $G$ Galois extension $T = \Ind_H^G(S/R)$. 
We require a result about the structure of the 
$G/N$ Galois extension $T^N/R$. 

\begin{proposition}\label{inducedfixed}
$T^N$ is 
$\Ind_{H/N \cap H}^{G/N}(S^{N \cap H})$.  
\end{proposition} 

\begin{proof}
$T$ has an $H$ invariant idempotent $e$ and the $G$ orbit 
of $e$ is a set of idempotents, $E$, such that 
$\sum_{f \in E} = \sum_{gH \in G/H} g(e) = 1$. 
Under the action by $N$, $E$ is the union of orbits of $N$, 
one for each double coset $NgH$. Since $N$ is normal, 
$NgH = gNH$. More concretely, let $g_i$, with $g_1 = 1$, be a set of 
left coset representatives of $NH$ in $G$ so $G$ is the disjoint 
union $\cup_i g_iNH$.  If $E_i \subset E$ is the set 
$g_iNH(e) = g_iN(e)$, then the $E_i$ are the orbits of $N$ in $E$ 
because $g_iN(e) = Ng_i(e)$. 

Let $e_i = \sum_{f \in E_i} f$. Then $e_i$ is $N$ fixed, and 
$T_i = Te_i = \sum_{f \in E_i} Tf$ satisfies 
$N(T_i) = T_i$ and $T_i = g_i(T_1)$. As an $R[N]$ module $T$ is the direct 
sum of the $T_i$. It follows that $T^N$ 
is the direct sum of the 
$T_i^N$. Let $\{n_j\}$ be a set of left coset representatives 
of $H \cap N$ in $N$ and hence of $H$ in $NH$. 
Then $E_1 = \{ n_j(e) \}$. We identify $S$ with $Te$, 
and it is clear that $T_1^N = \{ \sum_j n_j(s) | s \in S^{H \cap N} \}$. 
This implies $T_1^N \cong S^{H \cap N}$. Since $T^N$ 
is the direct sum of the $T_i = g_i(T_1)$, the result is clear. 
\end{proof}

We will have need of some basic operations that can be applied to cyclic 
extension of commutative rings. 

We begin with the easiest case. Suppose 
$S/R$ and $T/R$ are Galois with cyclic groups 
$G$,$H$,  
of order $n$ generated by $\sigma$ and $\tau$ 
respectively. Then 
$S \otimes T$ is Galois over $R$ with group 
$G \oplus H$ and we can set $N = <(\sigma,\tau^{-1})>$ 
to be the subgroup generated by that pair. Of course,  $(S \otimes_R T)^N/R$ 
is cyclic Galois of degree $n$ and this is the product of $S/R$ and 
$T/R$. It is clear that $N$ is not 
unique but unique if we require that $\sigma$ and $\tau$ are the specified generators. 
To be more precise, then, we have defined above an operation 
on pairs $(S,\sigma)$ where $S/R$ is Galois of degree $n$ and 
$\sigma$ generates the Galois group. We usually ignore 
this. 

We observe that isomorphism classes of degree $n$ cyclic Galois extensions 
of degree $n$ form a group. The identity element is 
the split extension $\Ind_{\{1\}}^G (R)$ because: 

\begin{lemma}
The product of $S/R$ and $\Ind_{\{1\}}^G(R)/R$ is $S$ as Galois extensions of $R$.
\end{lemma}

\begin{proof}
If $T = S \otimes_R \Ind_{\{1\}}^G(R)$ then 
$T/R$ is Galois with group $<\sigma> \oplus <\sigma>$ 
and $T$ contains idempotents $e_i$ such that 
$e_iT \cong S$ and $(1,\sigma)(e_i) = e_{i+1}$. 
Writing $T = \oplus_i e_iT \cong \oplus S$ we have that 
$T^{(\sigma,\sigma^{-1})} = (s,\sigma(s), \ldots) \cong 
S$ preserving the action of $(\sigma,1)$.
\end{proof}
 
Now we can define
the inverse operation. If $T$ is the pair 
$(S,\sigma)$, let $S'$ be $(S,\sigma^{-1})$. 
Then $SS'$ is the fixed ring $(S \otimes_R S)^{\tau}$ 
where $\tau = \sigma \otimes \sigma$.  

\begin{lemma}
$TT'$ above is the split extension of $R$. 
\end{lemma}

\begin{proof}
Of course $T \otimes T \cong \oplus_{g \in G} T$
where the $g$ projection is $t \otimes t' \to tg(t')$. 
In other language, $T \otimes T$ has idempotents $e_g$ 
defined by the relations $e_g(T \otimes T) = T$ and 
 $(t \otimes 1 - 1 \otimes g^{-1}(t))e_g  
= 0$. Acting by $\sigma \otimes \sigma$ preserves the relation 
and so $\sigma \otimes \sigma$ fixes $e_g$. Thus 
$(T \otimes T)^{\sigma \otimes \sigma} = \sum_g Re_g$ 
is split.
\end{proof} 

Finally the group of cyclic extensions $T/R$ of degree $n$ has exponent 
dividing $n$. 

\begin{proposition}\label{torsion}
If $T/R$ has degree $n$ and $m$ divides $n$ then $T^m$ 
has the form $\Ind_K^G(S/R)$ where $S/R$ has degree $n/m$. 
As a Galois extension of $R$, $S$ can be identified with 
$T^{\sigma^{n/m}}$.  
In particular, $T^n$ is split. 
\end{proposition}

\begin{proof} 
It is pretty clear that $T^m$ can be described as follows. 
Let $G = <\sigma>$ be the Galois group of $T/R$ and form 
the $m$ fold tensor power $W = T \otimes_R \ldots \otimes_R T$ 
which has Galois group the $m$ fold direct sum $H = G \oplus \cdots \oplus G$. 
$T^m = W^N$ where $N$ is generated by all elements of $H$ of the 
form $(g_1,\ldots,g_m)$ where $\prod_i g_i = 1$. 

It is also clear that $W$ is the direct sum of $n(m-1)$ copies 
of $T$ with idempotents $e(i_2,\ldots,i_m)$ defined as follows. 
If $t \in T$ let 
$t(i) = 1 \otimes \cdots \otimes 1 \otimes t \otimes 1 \cdots \otimes 1$ 
where the $t$ appears in the $i$ place. Then $e(i_2,\ldots,i_m)$ 
is the idempotent where $(t(1) - (\sigma^{i_j}(t))(j))e(i_2,\ldots,i_n) = 0$ 
for all $t \in T$ and $2 \leq j \leq m$. If $\tau \in G$ 
let $\tau_i = (1,\ldots,1,\tau,1,\ldots,1)$ where $\tau$ appears 
in the $i$ place. Then $\sigma_1(e(i_2,\ldots,i_m)) = e(i_2-1,\ldots,i_m-1)$ 
and 
$\sigma_j(e(i_2,\ldots,i_m)) = e(i_2,\ldots,i_{j-1},i_j+1,i_{j+1},\ldots,i_m)$. 
Now $N$ is generated by elements of the form 
$\epsilon_i = \sigma_1^{-1}\sigma_i$ and if $\epsilon_i(e(i_2,\ldots,i_m)) = 
e(j_2,\ldots,j_m)$ then $\sum_k i_k = \sum_k j_k$ modulo $m$. 
On the other hand $\sigma_1(e(i_2,\ldots,i_m)) = 
e(j_2,\ldots, j_m)$ where $\sum_j j_k = (\sum_k i_k) + 1$ 
modulo $m$. 
Thus if we set $e(i) = \sum e(i_2,\ldots,i_m)$, the sum being over 
all $(i_2,\ldots,i_m)$ with $\sum_k i_k = i$ modulo $m$, 
then $e(i)$ is an idempotent, 
is $N$ invariant, and $e(0) + \cdots + e(m-1)$ is the sum of all 
the $e(i_2,\ldots,i_m)$ and hence equals $1$. Moreover, 
$H/N$ permutes the $e(i)$'s cyclically. 

$e(0)W$ is the direct sum of the $e(i_2,\ldots,i_m)W$'s 
where the $i_j$'s sum to $0$ modulo $m$. Thus there is a projection $e(0)W \to 
e(0,\ldots,0)W$. However, if $T_1 \subset W$ is 
$T \otimes 1 \otimes \cdots \otimes 1$ clearly $e(0,\ldots, 0)W = 
e(0,\ldots,0)T_1 \cong T$. The stabilizer in $H$ 
of $e(0,\ldots,0)$ is generated by $\sigma_1\ldots\sigma_m$ 
and so 
the stabilizer in $N$ is generated 
by $\sigma_1^{n/m}\ldots\sigma_m^{n/m}$. 
The composition $S = e(0)W^N \to e(0)W \to T$ then identifies 
$S$ with the the $\sigma^{n/m}$ invariant subring of $T$. 
\end{proof}  

Next suppose $S/R$ and $T/R$ are as above except that 
$G = <\sigma>$ has order $n$ a multiple of $m$, the 
order of $H = <\tau>$. Then $S \otimes_R T$ 
has group $G \oplus H$. The largest cyclic image 
of this group has order $n$ and we can set 
$N = <(\sigma^{n/m},\tau^{-1})>$ so that 
$(G \oplus H)/N$ is cyclic of order $n$ 
generated by the image of $\sigma$. 
This time we have chosen $N$ so that in the 
cyclic image, $\tau$ is the $n/m$ power of $\sigma$. 
We write $ST = (S \otimes_R T)^N$ again. 
Note that we are defining an operation
of degree $m$ cyclics on the group of degree 
$n$ cyclics. 

Of course we can view $H \subset G$ so that $\tau = \sigma^{n/m}$, 
and let $T' = \Ind_{H}^G(T)$ be the induced $G$ 
Galois extension with specified generator $\sigma$. 
Then $T'/R$ is cyclic of degree $n$ and we can form 
$ST'/R$. This is the same operation as detailed 
above because: 

\begin{proposition}
$ST$ defined above is $ST'$. 
\end{proposition}

\begin{proof}
We are claiming that 
 $$(S \otimes T')^{(\sigma \otimes \sigma^{-1})} \cong  
 (S \otimes T)^{\sigma^{n/m} \otimes \tau^{-1}}.$$
$T'$ has an idempotent $e$ such that $\tau(e) = e$, $T = eT'$ 
and the induced action of $H$ on $eT' = T$ is the given 
$H$ action. Moreover, $\sum_{i=0}^{n/m - 1}\sigma^i(e) = 1$ 
and $\sigma^i(e)\sigma^j(e) = 0$ if 
$0 \leq i < j < n/m$. Thus $S \otimes T' = 
\sum_i (S \otimes T)\sigma^i(e)$ for $i$ in the same 
range. It follows that the $\sigma \otimes \sigma^{-1}$ 
fixed elements of $S \otimes T'$ have the form 
$\sum_0^{n/m-1} \sigma^{-i} \otimes \sigma^i(\beta{e})$ 
where $\beta \in S \otimes T$ is $\sigma^{n/m} \otimes \tau^{-1}$ 
fixed. Projecting on the $e$ component proves the result.
\end{proof} 

It will be important to understand how this operation 
on unequal degree cyclic extensions affects the 
rings. Suppose $S/R$ is cyclic of degree $n$ and 
$T/R$ is cyclic of degree $m$ dividing $n$. 
Let $U \subset S$ be such that  
$U \subset S$ is the fixed ring of the 
subgroup of order $m$ so $U/R$ is cyclic of degree $n/m$. 
In the rest of this 
paper we will describe this situation by saying $S/U/R$ is cyclic 
Galois. 

\begin{lemma}\label{topproduct}
The extension $ST/R$ decomposes as $ST/U/R$ 
where, as extensions of $U$, 
$ST$ is the product of $S/U$ and $(T \otimes_R U)/U$.  
\end{lemma} 

\begin{proof} 
Let $\sigma$ generate the Galois group, $G$, of $S/R$ 
and let $H = <\tau>$ be the Galois group of $T/R$. Then $U$ is the fixed ring 
$S^{\sigma^{n/m}}$. Recall that 
the image, $\bar \sigma$, of $\sigma$ in 
$(G \oplus H)/(\sigma^{n/m},\tau^{-1})$ 
is the designated generator. If $R \subset U' \subset ST$ 
is such that $U'/R$ has degree $n/m$, then 
$U' = (ST)^{\bar \sigma^{n/m}}$. If $K \subset G \oplus H$ 
is generated by $(\sigma^{n/m},1)$ and $(\sigma^{n/m},\tau^{-1})$, 
then $U' = (S \otimes_R T)^K$. But $K = <\sigma^{n/m}> \oplus <\tau>$ 
so $U' = U \otimes 1 \cong U$. Thus $K$ is the Galois group of $ST/U$.  
Identifying $S \otimes_R T$ with $S \otimes_{U} (U \otimes_R T)$ 
we have $ST = (S \otimes_{U} (U \otimes_R T))^{(\sigma^{n/m},\tau^{-1})} = 
(S/U)((T \otimes_R U)/U)$.
\end{proof}
 
The way we like to think about the above result is that 
when $m$ divides $n$ but $m < n$ then $T/R$ multiples 
at the ``top" of $S/R$. 

We will be studying the extension problem 
for cyclic extensions and in that context we will need the 
reverse of the above remark. 

\begin{proposition}\label{difference}
Suppose $T/S/R$ and $T'/S/R$ are cyclic of the same degree 
and $d$ is the degree of $T/S$ and (of course) $T'/S$. 
Then $T' = DT$ where $D/R$ is cyclic of degree $d$. 
\end{proposition} 

Of course $T'/S$ and $T/S$ differ by a degree $d$ 
cyclic. The point of the above remark is that 
the cyclic comes from $R$. This will be repeatedly 
important because we will be in situations where 
we know something about the arithmetic of $R$ 
and much less about the arithmetic of $S$. 

\begin{proof}
For convenience Let $G$ denote the 
Galois group of $T$ and $T'$, so 
$T \otimes_R T'$ has Galois group $G \oplus G$. 
If $\bar G$ is the Galois group of $S/R$, 
we can make identifications (note $\Aut(G) \to \Aut(\bar G)$ 
is surjective) such that 
$\phi: G \to \bar G$ is defined by $T/S/R$ and $T'/S/R$. 
$S \otimes_R S \subset T \otimes_R T'$ 
and $S \otimes_R S$ has an idempotent $e$ 
where $e(S \otimes_R S) = S$ and $(s \otimes 1 - 1 \otimes s)e = 0$. 
Then $e(T \otimes_R T') = T \otimes_S T'$ is Galois 
over $R$ with group $H = \{g,g' \in G \oplus G | \phi(g) = \phi(g')\}$. 
In $H$ we have $N = \{(g,g) | g \in G\}$ and $D = (T \otimes_S T')^N$ 
with the appropriate generator. 
\end{proof} 

We will have need to understand cyclic $S/R$ 
which are split modulo an ideal. To this end 
let $N \subset R$ be an ideal. Let $T/S/R$ and $T'/S/R$ 
be cyclic Galois of degree $n$ as above. 
Write $T' = TD$ where $D/R$ is cyclic of the appropriate 
degree. 

\begin{lemma}\label{splitdifference}
If $T/NT$ and $T'/NT'$ are both split over $R/N$, 
then $D/ND$ is split over $R/N$.
\end{lemma}

\begin{proof}
Given the proof of the existence of $D$, it suffices to observe 
and if $T/R$ is $G$ Galois and split, and $N \subset G$ 
is a normal subgroup, then $T^N/R$ is split. This is clear from 
\ref{inducedfixed}. 
\end{proof}

Next we make a simple observation about 
how the operation above relates to the 
$R[G]$ structure of the cyclic extensions. 

\begin{proposition}\label{moduleproduct}
Suppose $G$ is the cyclic group
$S/R$ is a $G$ Galois extension, and $S'/R$ is an 
$H$ Galois extension where $H \subset G$. 
Let $T = SS'/R$. Then as $R[H]$ modules 
$T \cong S \otimes_{R[H]} S'$. 
\end{proposition} 

\begin{proof} 
As usual, let $G$ be generated by $\sigma$, 
and $H$ generated by $\sigma^{n/m} = \tau$.  
$SS'$ is defined as $(S \otimes_R S')^{\gamma}$ 
where $\gamma = \sigma^{n/m} \otimes \tau^{-1}$ and where we have 
identified $R[G \oplus H] = R[G] \otimes_R R[H]$. 
Define the surjection $\Psi: S \otimes_R S' \to SS'$ 
by $\Psi(s \otimes s') = \tr_{\gamma}(s \otimes s')$. 
Since $\tr_{\gamma}(\gamma - 1) = 0$ we have 
$\Psi(\sigma^{n/m}(s) \otimes \tau^{-1}(s')) = \Psi(s \otimes s')$, 
or changing variables, $\Psi(\sigma^{n/m}(s) \otimes s') = s \otimes \tau(s')$. 
Thus $\Psi$ induces a surjection $\Psi: S \otimes_{R[H]} S' \to SS'$ 
which is an isomorphism as the domain and range are 
projective of the same rank over $R$.
\end{proof} 

We require a corestriction operation on cyclic extensions 
in order to eliminate the assumption about $p$ roots 
of unity and understand $p^n$ degree cyclic extensions. 
In this discussion we confine ourselves to cases A and B, so  
$\mu$ will be a primitive $p$ roots of one and 
when $p = 2$ our base ring will be $\Z[i]$. 
Since $\Z[\mu]/\Z$ or $\Z[i]/\Z$ are not etale extensions, 
we have to frame our construction more generally and then 
apply special arguments when we need to corestrict in those cases. 
With this in mind, we consider the 
following set up. 
Suppose $D/R$ is an extension 
of rings and $C$ is a group acting on $D$ with fixed ring $D^C = R$. Let $C$ have order $m$. 
We also suppose $D$ is a finite generated projective 
as a 
module over $R$. Let $T/D$ be a cyclic $G$ Galois extension 
of degree $p^n$. For $\eta \in C$ let $\eta(T)$ be $T/D$ twisted by 
$\eta$ and $W = \otimes_{D,\eta \in C} \eta(T)$. 
Then the action of $C$ on $D$ extends to an 
action on $W$ by permuting the $\eta(T)$'s using the  
regular action of $C$. 
Clearly $W/D$ has Galois group the $m$-fold direct 
sum $H = G \oplus \l]cdots \oplus G$. There is an induced 
action of $C$ on $H$ which is also a permutation action on the 
$G$'s. 
Moreover $C$ and 
$H$ generate, in the $R$ automorphism group of $W$, 
the wreath product $$\hat H = H \rtimes C.$$ Let $N \subset H$ 
be the kernel of the product map $G \oplus \cdots \oplus G 
\to G$. Clearly $C$ preserves $N$ and so we have a homomorphism 
$\hat H \to G$ with kernel $\hat N = N \rtimes C$. 
If $W^N = T'$, then $T'/D$ 
is $G$ Galois. Set 
$S = \Cor_{D/R}(T/D) = (T')^C$, so $S$ is an extension 
of $R$. If $D/R$ is not Galois 
we can make no general statement about $S$ 
but with extra assumptions we have: 

\begin{theorem}\label{corestriction}
Suppose $S$ is a projective module over $R$ of rank 
$|G|$ and $SD = T'$. Then $S/R$ is $G$ Galois. 
In particular, if $D/R$ is $C$ Galois then $S/R$ 
is $G$ Galois. 
\end{theorem}

\begin{proof}
Clearly $T'/D$ is $G$ Galois and there is a natural 
map $\phi: S \otimes_R D \to T'$. By assumption $\phi$ 
is surjective and its domain and range are projective 
over $R$ of equal ranks. Thus $\phi$ is an isomorphism. 
Since $D/R$ is faithfully flat, $S/R$ is $G$ Galois. 

When $D/R$ is $C$ Galois we have by \ref{galoistower} 
that $W/R$ is $\hat H$ Galois and $S = W^{\hat N}$ 
so $S/R$ is $G$ Galois. 
\end{proof}

In the case $D/R$ is $C$ Galois, we can easily 
observe that the corestriction is functorial. 
The proof is trivial. 

\begin{proposition}\label{functorial}
Let $f: R \to R'$ be a homomorphism of commutative 
rings and $D/R$ $C$ Galois. Set 
$D' = D \otimes_f R'$ and let $f$ also denote 
the induced $f: D \to D'$. If $T/D$ 
is cyclic Galois, then $\Cor_{D/R}(T/D) \otimes_f R' 
\cong \Cor_{D'/R'}((T \otimes_f D')/R')$. 
\end{proposition}

We need a version of the restriction-corestriction result 
for the construction $\Cor_{D/R}(T/D)$ detailed above. 
Suppose $T = T' \otimes_R D$ where $T'/R$ is $G$ Galois. 
Then $W = W' \otimes_R D$ where $W' = 
T' \otimes_R \ldots \otimes_R T'$ the $m$ fold tensor power. 
Of course, $W'/R$ also has Galois group $H$. Since 
$T'D = T$, we also have $W'D = W$. Since $W'/R$ is projective, 
$W \cong W' \otimes_R D$ and $W^{\tau} = W'$. 
Furthermore, $W'^N$ is just $T'^m$ because 
$\tau$ acts trivially on $T'$. 

\begin{proposition}\label{rescor} 
Let $T'/R$ be cyclic Galois with Galois group 
$G$. Let $m$ be the degree of $D/R$ and $T = T' \otimes_R D$. Then 
$\Cor_{D/R}(T/D)$ is defined and  $\Cor_{D/R}(T/D)\cong T'^m/R$. If $m$ divides 
$|G|$, then $\Cor_{D/R}(T'D/D) \cong \Ind_H^G(T'')$ 
where $T'/T''/R$ is cyclic and $T''/R$ has degree $|G|/m$. 
\end{proposition}

\begin{proof}
Since $W^{\tau} = W'$, we take $N$ invariants of both sides 
and get $\Cor_{D/R}(T/D) = T'^m$. 
Thus $\Cor_{D/R}(T/D)$ is projective. Since the trace 
$\tr_N: W \to T$ is surjective, 
$T = \tr_N(W) = \tr_N(DW')= D\tr_N(W') = DT'^m$ 
so the Theorem \ref{corestriction} applies. 
The second statement follows from \ref{torsion}. 
\end{proof}

We want to apply \ref{corestriction} to the following  
two special cases. 
For the first, let $R$ be a regular 
dimension 2 domain of characteristic 0 and 
$D = R[Z]/(Z^2 + 1) = R \otimes_{\Z} \Z[i]$. 
Assume $T/D$ is $G$ Galois where $G$ is 
a cyclic group of order $2^m$. Let $i \in D$ 
be the image of $Z$. Assume $2 \in R$ is a prime element. 

\begin{theorem}\label{degree2cor}
$\Cor_{D/R}(T/D)$ is defined and is a $G$ Galois 
extension of $R$. 
\end{theorem}

\begin{proof}
Let $\sigma$ be the automorphism of 
$D$ such $\sigma(i) = -i$. Clearly $D^{\sigma} = R$. 
Proceeding with the definition of the corestriction, 
let $W = T \otimes_D \sigma(T)$ which is 
Galois with group $G \oplus G$ over $D$. 
Let $N = \{(g,g^{-1}) | g \in G\} \subset G \oplus G$. 
Then $W$ and $W^N$ have an 
extension of $\sigma$ of order 2. Let 
$S' = W^N$ and $S = (S')^{\sigma}$. 
In order to apply \ref{corestriction} 
we need to show $S/R$ projective and $SD = S'$. 

Since $D(1/2)/R(1/2)$ is $<\sigma>$ Galois, 
$S(1/2)/R(1/2)$ is projective. If 
$s \in S(1/2)$ is integral over $S$ it is 
integral over $R$ and hence $R'$, implying 
$s \in S'$. Since $s$ is $\sigma$ fixed, $s \in S$. 
Thus $S$ is integrally closed, hence reflexive, 
hence projective by e.g., \cite[p.~73]{OS}. 

Next we consider the case $T/D$ is split. 
That is, if $\delta$ generates $G$, then 
$T = \oplus_{i=0}^{2^m-1} De_i$ where 
$e_i$ is idempotent and $\delta(e_i) = e_{i+1}$ 
where we interpret the indices modulo $2^m$. 
Thus 
$W$ has as a $D$ basis the elements 
$e_i \otimes e_j$. Let $f_k = \sum_i e_i \otimes e_{k-i}$ 
so $(\delta,\delta^{-1})(f_k) = f_k$. it is 
clear that the $f_k$ for $k = 0,\ldots,2^m - 1$ 
form a $D$ basis of $W^N = S'$. Moreover, 
$\sigma(f_k) = f_k$ so $S'^{\sigma} = 
\sum_k f_kR$ and $S/R$ is split $G$ Galois. 
The fact that $DS = S'$ is clear. 

Returning to the general case, 
note that we need to prove that the canonical 
map $\phi: S \otimes_R D \to S'$ is an isomorphism. 
Thus by localization we may assume $R$ is regular local 
of dimension $2$. Since $S \otimes_R D$ and $S'$ are both 
free over $R$ of the same rank, 
the determinant of $\phi$ as an $R$ morphism is an element 
$a \in R$. 
Since $D(1/2)/R(1/2)$ is Galois, $\phi$ becomes 
an isomorphism after inverting $2$. 
Thus up to a unit $a = u2^n$ where $u$ is a unit of $R$.

Let $R_2$ be the localization of $R$ at the prime ideal $2R$, 
and $\hat R$ the maximal unramified extension of the completion 
of $R_2$. Since $2$ ramifies in $\Z[i]$, $D \otimes_R \hat R = 
\hat D$ is totally ramified of degree $2$ over $\hat R$ 
(with prime $\eta = i - 1$). Let $\hat T = T \otimes_D \hat D$ 
which is 
cyclic Galois over $\hat R$. Since the residue field 
of $\hat D$ is separably closed, and $\hat D$ is complete, 
$\hat T/\hat D$ is split. 

Let $\hat S' = (\hat T \otimes \hat T)^N = S' \otimes_D \hat D$, 
and $\hat S = (\hat S')^{\sigma}$. Since $\hat S' = S' \otimes_D \hat D = 
S' \otimes_R \hat R$, it follows that $(\hat S')^{\sigma} = 
S'^{\sigma} \otimes_R \hat R$. Since $\hat T$ is split, 
$\phi \otimes 1: (S \otimes_R D) \otimes_D \hat D \to 
S' \otimes_D \hat D$ is an isomorphism. 
Thus $a = u2^n$ must have $n = 0$. 
\end{proof}

Our next application of \ref{corestriction} is the following. 
Let $p$ be an odd prime (case A) and $\eta = \rho - 1$ the prime element 
of $\Z[\rho]$. Suppose $R$ is a dimension $2$ 
regular faithful domain over 
$\Z$ and $D = R \otimes_{\Z} \Z[\rho]$. 

$D$ is clearly free over $R$. 
Let $\tau$ be a generator of the Galois group 
of $\Q(\rho)/\Q$. Clearly $\tau$ induces 
an action on $D$ by letting $\tau$ act as the 
identity on $R$ and $D^{\tau} = R$.

Now suppose $T_1/D$ is cyclic Galois of degree 
$p^n$. We want to show that the conditions of 
\ref{corestriction} are met and $\Cor_{D/R}(T_1/D)$ 
is defined. To that end, let $W$ be $T_1 \otimes \tau(T_1) 
\otimes_D \cdots \otimes_D \tau^{m-1}(T_1)$, Then $W/D$ is Galois with group 
$H = G \oplus \cdots \oplus G$ and we can let 
$N$ be the kernel of the ``product" map $H \to G$. 
Let $T' = W^N$ and $T = T'^{\tau}$. Clearly $T'/D$ 
is Galois with group $G$.   

\begin{proposition}\label{corrho} 
\hfill\break
\begin{enumerate}
\item  $T$ is projective of rank $p^n$ as a module over $R$. \label{corrho:1}

\item $\tr(T') = T$. \label{corrho:2}

\item $T \otimes_R D = T'$.  \label{corrho:3}

\item If $T = \Cor_{D/R}(T_1)$ then $
T/R$ is Galois with group $G$.  \label{corrho:4}
\end{enumerate}
\end{proposition}

\begin{proof} 
Since $D(1/p)/R(1/p)$ is $C$ Galois, 
it is clear that $T(1/p)/R(1/p)$ is cyclic Galois. 
Just as above $T$ is the integral closure of 
$R$ in $T(1/p)$. Just as above, $T$ is reflexive and 
hence projective as an $R$ module (e.g., \cite[p.~73]{OS}) 
and its rank is clearly $p^n$. 

Let $\tr = 1 + \tau + \cdots + \tau^{m-1}$. 
Of course, $\tr(T') \subset T$. Let 
$V$ be the $R$ module quotient $T/\tr(T')$. 
Since $D(1/p)/R(1/p)$ is Galois, 
$p^rV = 0$ for some $r$. 
If $y \in T$, 
$\tr(y) = (p-1)y$ so $(p-1)V = 0$ also 
and $V = 0$ or $\tr(T') = T$.  

We next claim that $TD = T'$. If we 
invert by $p$ then the equality clearly holds as above. 
Thus if $U$ is the $D$ module quotient 
$T'/TD$, then 
$p^rU = 0$ for some $r$. 
However, $T' = \eta{T'} + TD$, because 
modulo $\eta$, $\tau$ is the identity, 
$\tr$ is multiplication by $p-1$, and hence 
everything is a trace. That is, 
$pU = \eta^{p-1}U = U$ proving $U = 0$ which is our claim. 
We are done by \ref{corestriction}
\end{proof} 

Let $D/R$ be as in \ref{corrho} 
or \ref{degree2cor}. In addition, assume 
$\hat R = R/pR$ is a Dedekind domain. 
In both cases $D$ has a prime element $\eta$ 
which is totally ramified over the prime element 
$p$ in $R$. Thus $\hat D = D/\eta{D} = 
\hat R$. 
Let $m$ be the degree of $D/R$ and let $T_1/D$ be cyclic 
of a degree where $T = \Cor_{D/R}(T_1/D)$ is defined. 
We are interested in comparing $\hat T_1 = T_1/\eta{T_1}$ 
and $\hat T = T/pT$ which are both cyclic extensions of $\hat R = R/pR$. 
Let $C$ be the group of automorphisms of $D$ with $D^C = R$ 
as in \ref{corrho} or \ref{degree2cor}. 

\begin{proposition}\label{goodmodp}
Let $T_1/D$ and $T/R$ be as above.   
Then $(\hat T_1)^m \cong \hat T$ as cyclic Galois extensions of 
$\hat R$.
\end{proposition}

\begin{proof}
Let $W/D$ be the Galois extension in the 
definition of $T$. Since $\tau^i(T_1)/\eta\tau^i(T_1) 
\cong \hat T_1$, $W/\eta{W} \cong 
\hat T_1 \otimes \cdots \otimes \hat T_1$ with Galois 
group $H = G \oplus \cdots \oplus G$. 
The kernel, $N$, of the summation map $H \to G$ 
is generated by the elements $(1,\ldots,\sigma,1,\ldots,1,\sigma^{-1},1,\ldots,1)$. Set $T_2 = W^N$ so $T = T_2^C$. 
It is clear that $T_2/\eta{T_2} = W/\eta{W}^N = (\hat T_1)\ldots(\hat T_1) = \hat T_1^m$ 
as Galois extensions of $D$. 
Also $T = T_2^{C} \subset T_2$ which induces 
$\phi: T/pT \to T_2/\eta{T_2}$ and $\phi$ extends 
the identity $R/pR = D/\eta{D}$. The map $\phi$ clearly 
preserves the Galois actions, and the domain and range are 
Galois extensions of $\hat R$. Since the kernel of this map 
is $G$ invariant it must be $0$. Since $\hat R$ is a Dedekind domain, 
$T/pT$ must be integrally closed and so 
$T/p{T} = T_{2}/\eta{T_{2}}$. 
\end{proof}

\section{Picard Groups of Fiber Products}

In this paper a prominent role will be played by 
the Galois module structure of the $G$ Galois extension 
$S/R$ where $G$ is cyclic of degree $p$ and $R$ is an algebra 
over $\Z[\rho]$. We will observe that $R[G]$ 
is an iterated fiber product. Happily, it will turn 
out that we really only care about the first fiber product. 
This short section deals with fiber products in general. 

We begin with a result which appeared in \cite[p.~55]{C}. 
\begin{theorem}\label{fiber} 
Suppose we have the fiber product diagram 
$$\begin{matrix}
R&\longrightarrow&R_1\cr
\downarrow&&\downarrow\cr
R_2&\buildrel{f}\over\longrightarrow&R_3\cr
\end{matrix}$$ 
with $f$ surjective. Further assume that for each 
maximal ideal, $M$,of $R$, $R_M \otimes_R R_i$ 
has trivial Picard group. Then there is a long exact 
sequence
$$0 \to R^* \to R_1^* \oplus R_2^* \to R_3^* \to $$
$$\to \Pic(R) \to \Pic(R_1) \oplus \Pic(R_2) \to \Pic(R_3) \to$$
$$\to \Br(R) \to \Br(R_1) \oplus \Br(R_2)$$
\end{theorem} 

In the paper \cite{C} the author assumes that $R_M \otimes R_i$ 
is semilocal for all $i$, but the proof only needs the assumption that all 
the Picard groups $\Pic(R_i)$ are trivial Zariski locally 
on $\Spec(R)$, which is what we are assuming.  

Fiber diagrams of the above sort are also called pullback diagrams 
or equalizer diagrams. 
We will frequently use the well known fact: 

\begin{proposition}\label{idealfiber}
Let $R$ be a commutative ring with ideals $I,J$. 
Then 
$$\begin{matrix}
R/(I \cap J)&\longrightarrow&R/I\cr
\downarrow&&\downarrow\cr
R/J&\buildrel{f}\over\longrightarrow&R/(I+J)\cr
\end{matrix}$$ 
is a fiber diagram. 
\end{proposition}

It will be useful to understand the above diagram 
in a concrete way. If $P$ is an element of $\Pic(R)$, 
set $P_i = P \otimes_R R_i$. Then $P$ is defined by the induced pull-back diagram:
$$\begin{matrix}
P&\longrightarrow&P_1\cr
\downarrow&&\downarrow\cr
P_2&\buildrel{f}\over\longrightarrow&P_3\cr
\end{matrix}$$ 
If $\bar u \in R_3^*$ we can form the new 
pull-back diagram defining $P^*$:
$$\begin{matrix}
P^*&\longrightarrow&P_1\cr
\downarrow&&\downarrow\cr
P_2&\buildrel{\bar u{f}}\over\longrightarrow&P_3\cr
\end{matrix}$$ 
where we alter $f$ by composing with 
$\bar u: P_3 \cong P_3$. If however $\bar u$ 
is an image $\bar u_1\bar u_2$ where $u_i^* \in R_i^*$ 
are units, then the using these isomorphisms on $P_1$ and $P_2$ makes the above diagram isomorphic to the original one for $P$. In this way we have an action of 
$R_3^*/M$ on $\Pic(R)$ where $M \subset R_3^*$ is generated 
by the images of $R_1^*$ and $R_2^*$. 

We need a piece of the above result expressed in a 
concrete way. Let $\phi_i: R_i \to R_3$, $i = 1,2$ 
be the labels for those maps in the diagram. 

\begin{proposition}
Suppose $P$ is a rank one projective 
$R$ module such that $P \otimes_R R_i \cong R_i$ 
for $i = 1,2,3$. Then there is an 
exact sequence 
$$0 \to P \to R_1 \oplus R_2 \buildrel{h}\over\longrightarrow R_3 \to 0$$ 
and an $u \in R_3^*$ where $h(r_1,r_2) = 
\psi_1(r_1) - u\psi_2(r_2)$. Here $u$ is uniquely 
determined up to multiplication by $\phi_1(R_1^*)\phi_2(R_2^*)$. 
\end{proposition} 

\begin{proof}
Set $P_i = P \otimes_R R_i$. 
Let $\theta_i: P_i \cong R_i$ 
be the assumed isomorphisms. 
Tensoring by $P$ yields the exact sequence 
$0 \to P \to P_1 \oplus P_2 \to P_3 \to 0$. 
Composing with $\theta_1 \oplus \theta_2$ yields 
the injection $P \to R_1 \oplus R_2$ whose image 
is the kernel of $R_1 \oplus R_2 \to P_1 \oplus P_2 
\to P_3$ which we can then follow with $\eta_3$. 
The induced map $R_1 \oplus R_2 \to R_3$ differs 
from the given one by an automorphism of $R_3$ 
which is given by $u$. 
\end{proof} 

\section{Degree \protect{\boldmath{$p$}} Case} 

In this section, we really care that the multiplicative 
groups of our ground rings 
contain a cyclic group of order $p$. That is, in this 
section we confine ourselves to cases A and C.  
In detail, $\rho$ is a primitive $p$ root of one, 
including $\rho = -1$ when $p = 2$ (so $\Z[\rho] = \Z$). 
As usual, 
$\eta$ is a choice of prime element totally ramified over 
$p$, specifically $\eta = \rho - 1$ so $\eta = -2$ when 
$p = 2$. 
We consider degree $p$ cyclic Galois extensions $S/R$ 
where $R$ is an algebra over $\Z[\rho]$. 
This is really the core machinery of this paper. 

Now we consider the polynomial ring $R_0 = 
\Z[\rho][a]$ and the polynomial 
$X^p - (1 + a\eta^p)$. For reasons that will soon be clear, 
we set $R = R_0(1/(1 + a\eta^p))$. If we substitute 
$X = 1 + \eta{Z}$ we get

$$(1 + \eta{Z})^p - (1 + a\eta^p) = 
\eta^p{Z^p} - a\eta^p + 
p\eta\left(\sum_1^{p-1} b_iZ^i\eta^{i-1}\right)$$ 
$$= \eta^p{Z^p} - a\eta^p + 
u\eta^p\left(\sum_1^{p-1} b_iZ^i\eta^{i-1}\right) = 
\eta^p\left(Z^p + u\sum_1^{p-1} b_iZ^i\eta^{i-1} - a\right)$$
and we define $g(Z)$ so that this last expression 
is $$\eta^p(Z^p + g(Z) - a).$$ 
Equivalently, we have defined 
the polynomial $g(Z) \in R_0$ so that 
 
$$(1 + Z\eta)^p = 1 + (g(Z) + Z^p)\eta^{p}$$

where the coefficients of $g(Z)$ are in $\Z[\rho]$. 
Since $b_1 = 1$ we note that $g(Z)$ is congruent
to $-Z$ modulo $\eta$. That is, our new 
polynomial $Z^p + g(Z) - a$ equals the Artin-Schreier 
polynomial $Z^p - Z - a$ modulo $\eta$. 

Suppose $\alpha$ is the image of $X$ 
in $S_0 = R_0[X]/(X^p - (1 + a\eta^p))$. 
Set $S = S_0(1/(1 + a\eta^p)) = R[X]/(X^p - (1 + a\eta^p))$. 
Let $S_1 = S(1/\eta)$. 
Then $\theta = (\alpha - 1)/\eta \in S(1/\eta)$ is a root of $Z^p + g(Z) - a$. Of course, 
$Z^p + g(Z) - a$ has coefficients in $R_0$ and 
so $\theta$ is integral over $R_0$. 
We view $T_0= R_0[Z]/(Z^p + g(Z) - a) \subset S_1$ 
by identifying the image of $Z$ with $\theta$. 
Set $T = T_0(1/(1 + a\eta^p))$.  
Since $\Z[\rho]/(\eta) = F_p$, the field of 
$p$ elements, $T_0/\eta{T_0} = T/\eta{T} = F_p[a][Z]/(Z^p - Z - a)$ 
and so $\eta$ is a prime element of $T_0$ and $T$. 

$S$ has an automorphism $\sigma: S \cong S$ 
defined by $\sigma(\alpha) = \rho\alpha$ and 
$\sigma$ extends to an automorphism of $S_1$. We have 

$$\sigma(\theta) = (\rho\alpha - 1)/\eta = \rho((\alpha - 1)/\eta + \eta/\eta = 
\rho\theta + 1.$$ 

Thus $\sigma$ satisfies $\sigma(T_0) = T_0$ 
and extends to an automorphism of $T$. 
Of course, modulo $\eta$, this is just the 
Artin-Schreier action.  
By induction $\sigma^i(\theta) = \rho^i\theta + 
\rho^{i-1} + \cdots + 1$. Also, 
$Z^p + g(Z) - a = \prod_i (Z - \sigma^i(\theta))$.

\begin{proposition}\label{prop:1.1} Suppose $f: \Z[\rho] \to R'$ defines 
$R'$ as a $\Z[\rho]$ algebra. Let $g: R_0 \to R'$ be any 
$\Z[\rho]$ algebra homomorphism (i.e. extension of $f$).  
Set $b = g(a)$. 
Then $g(Z^p + g(Z) - a) = Z^p + g(Z) - b$ is separable if and only 
if $1 + b\eta^p$ is invertible. In particular, the extension $T/R$ above 
is Galois with group generated by $\sigma$. 
If $1 + b\eta^p$ is invertible 
in $R'$ then $g$ extends to $g: R \to R'$ 
and $T' = T \otimes_f R'$ is $G$ Galois over $R'$.  
If $\eta$ is invertible in $R'$ then 
$T' = R[Z]/(Z^p - (1 + b\eta^p))$. 
\end{proposition}

\begin{proof} Let $\theta$ be a root of $Z^p + g(Z) - a$ 
and $\Delta$ the discriminant. Then $\Delta$ 
the product of expressions of the form 
$\sigma^j(\sigma^i\theta - \theta) = 
\sigma^j((\rho^i - 1)\theta + \rho^{i-1} + \cdots + 1) = (\rho^{i-1} + \cdots + 1)(\eta\theta + 1)$ 
where $i \not= 0$ and $j = 0, \ldots, p-1$. 
Fixing $i$ and taking the product for all $j$ 
yields $(\rho^{i-1} + \cdots + 1)(1 + a\eta^p)$ 
since $\eta\theta + 1$ is a root of 
$X^p - (1 + a\eta^p)$. We saw in the proof of 
\ref{total} that $\epsilon_i =  \rho^{i-1} + \cdots + 1$ 
is a unit in $\Z[\rho]$ proving the first statement. 

As for the second statement, we use the 
criterion of \cite[p.~81]{DI}. Since 
$\sigma^i(\theta) - \theta$ is a unit, 
we only need to verify that the fixed ring $T^{\sigma} = 
R$. If $K$ is the field of fractions 
of $R$ and $L = T \otimes_{R} K$ 
then $L^{\sigma} = K$ because $L/K$ is the 
Kummer field extension $K((1 + u\eta^p)^{1/p})/K$. 
Since $T/R$ is integral, and 
$R$ is integrally closed, $T^{\sigma} = 
R$ is clear.
\end{proof}

The following discussion is excerpted from [S2022]. 
Since [S2022] will not appear in a journal, we include 
it here. In the generic case $\eta$ is a non--zero divisor and one can pass easily from 
$\theta$ to $\alpha = \eta\theta + 1$. 
In general we cannot, and so we must develop 
relations in the generic case and then 
specialize them. Furthermore, the notation 
we develop will be quite suggestive. In the following discussion 
we will operate in two domains. If we work ``generically'' we will 
be making our definitions and calculations in $R_0$ or $R$ or an extension 
ring $R[x,y]$ etc. The key point will be that $\eta$ is a 
non--zero divisor. When we work generally, we mean we have an arbitrary 
$\Z[\rho]$ algebra $R'$ and a homomorphism $g:R_0 \to R'$ or $R[x,y]\to R'$ etc. 

Let $A$ be the monoid $1 +\eta{R}$ under 
multiplication, and $A^* \subset R^*$ the subgroup of invertible elements.  
Define $\phi: R \to A$ via $\phi(x) = 
1 + x\eta$. Since $\eta$ is a non--zero divisor, 
this is a bijection.  
To make $\phi$ some sort of homomorphism we define 
$x \oplus y$ by writing $(1 + x\eta)(1 + y\eta) = 
 1 + (x \oplus y)\eta$ or $\phi(x \oplus y) = 
\phi(x)\phi(y)$. 
Note that we are just using helpful notation. 
Clearly, $x \oplus y = x + y + xy\eta$. 
If we look at general $g: R \to R'$ (which includes $R \to R[x,y]$) 
we define 
$\eta' = f(\eta)$, $\rho' = f(\rho)$, $A' = 1 + \eta'R'$, $\phi: R' \to A'$, 
$\phi(x \oplus y) = \phi(x)\phi(y)$ etc. In the following 
if we make an observation about $R$ or $R[x,y]$, we will not restate the obvious 
inference about $R'$. 

We can define $x \ominus y = z$ to mean $x = y \oplus z$. 
Of course $z$ 
exists if $\phi(y) \in A'^*$. Note that since 
$\rho = 1 + \eta$, $\phi(1) = \rho$. 
Also, the basic relation between the Kummer 
element $\alpha$ and $\theta$ described above 
is just that $\phi(\theta) = \alpha$. 
Thus $\sigma(\alpha) = \rho\alpha$ translates 
into $\sigma(\phi(\theta)) = \phi(1)\phi(\alpha)$ 
or, suggestively, 
$$\sigma(\theta) = 1 \oplus \theta$$ 
which makes sense because $1 \oplus \theta = 
1 + \theta + \theta\eta = 1 + \rho\theta$. 

We are interested in the case $Z^p + g(Z) - a$ defines a 
Galois extension so let $g: R \to R'$ be as above. 
Let $S' = R'[Z]/(Z^p + g(Z) - a)$ and let $\theta$ be the image of 
$Z$, so $S'/R'$ is $C_p$ Galois and $\sigma(\theta) = \theta \oplus 1$. 
The next result describes how we can change 
the choice of $\theta$ and then track the corresponding 
change in $a$. To describe the result, 
let $A_p$ be the monoid $1 + R'\eta'^p$ 
and $\phi_p: R' \to A_p$ the map 
$r \to 1 + r\eta^p$. In a similar way as before define $x \oplus_p y$ by 
$\phi_p(x)\phi_p(y) = \phi_p(x \oplus_p y)$. 
The point of this and subsequent definitions 
is to distinguish $z$ from it's image 
$1 + z\eta$ or $1 + z\eta^p$. 
This is crucial when $\eta$ is a zero 
divisor.

\begin{proposition}\label{prop:1.2}  If $\theta$ is as 
above and $\theta'$ also satisfies 
$\sigma(\theta') = 1 \oplus \theta'$ then 
$\theta' = \theta \oplus z$ for some $\sigma$ 
fixed $z$. If $\theta$ satisfies $Z^p + g(Z) - a = 0$ then $\theta'$ satisfies 
$Z^p + g(Z) - a' = 0$ 
where $a' = a \oplus_p (g(z) + z^p)$. 
Conversely, if $a' = a \oplus_p (g(z) + z^p)$ for some $z$ with $\phi(z) \in A^*$ then 
$Z^p + g(Z) - a$ and $Z^p + g(Z) - a'$ define
isomorphic $C_p$ Galois extensions. 
\end{proposition}

\begin{proof} $(1 + \eta\theta)^p = 1 + a\eta^p$ and so 
$1 + \eta\theta$ is invertible.  We can 
set $z = \theta' \ominus \theta$ and the first 
statement is clear. To show the second statement, 
we can work in the generic case and hence 
assume $\eta$ is a non--zero divisor. We compute that 
$(1 + \theta'\eta)^p = 
(1 + \theta\eta)^p(1 + z\eta)^p = 
(1 + \theta\eta)^p(1 + \eta^p(g(z) + z^p)) = 
(1 + u\eta^p)(1 + (g(z) + z^p)\eta^p) = 1 + (u \oplus_p (g(z) + z^p))\eta^p)$. 
Thus $\theta'$ behaves as claimed. 
For the last statement we observe that 
$\tau(\theta') = \theta \oplus z$ defines the 
isomorphism.
\end{proof} 

Of course, $Z^p - Z - a$ and $Z^p - Z - a^p$ 
define isomorphic Galois extensions in characteristic $p$, 
but $Z^p + g(Z) - a$ and $Z^p + g(Z) - a^p$ 
do not necessarily define isomorphic extensions. 
That is, the extensions defined by $Z^p - Z - a$ and $Z^p - Z - a^p$ 
are the same modulo $p$ but $Z^p + g(Z) - a$ 
and $Z^p + g(Z) - a^p$ define possibly distinct 
lifts. 

For a moment, just assume $G$ is a finite abelian group. 
Recall from \ref{proj} that $S$ is projective over 
$R[G]$, and checking ranks shows that $S \in \Pic(R[G])$. 
We can say a bit more about $S$ as an element of $\Pic(R[G])$. 
Let $n$ be the order of $G$.  

\begin{lemma}\label{picorder}
Let $G$ be a finite abelian group of order $n$. 
As an element of $\Pic(R[G])$, $S$ 
has order dividing $n$. 
\end{lemma} 

\begin{proof} 
Recall that $S \otimes_R S 
\cong \sum_{g \in G} S$ where the map onto the $g$ component is 
$s \otimes t \to sg(t)$. Thus 
$S \otimes_R S \cong S[G]$ as $S[G]$ 
modules, and then restriction-
corestriction on the Picard group shows that $S^n  
\cong R[G]$. Alternatively, 
$\Cor_{S/R}(S \otimes_R S) \cong 
\Cor_{S/R}(\Ind_1^G(S)) \cong \Ind_1^G(R)$ 
and the underlying module of $\Cor_{S/R}(S \otimes_R S)$ is the Picard group norm. 
\end{proof} 

To push our results further, we need to employ 
the pull back results of the previous section. 
We again assume $G$ is the cyclic group of order $p$. 
Assume that $R$ is a $\Z[\rho]$ algebra, and so 
$T^p - 1 = \prod_{i=0}^{p-1} (T - \rho^i)$. 
Also assume $\eta = \rho - 1$ is a non zero divisor. In particular, $\rho^i \not= \rho^j$ 
when $(i-j)$ is not divisible by $p$. 
We view $R[G]$ as $R[T]/(T^p - 1)$. In particular, if 
$G = <\sigma>$  and $f(T) \in R[T]$ it makes sense 
to write $f(\sigma)$. 
We start with: 

\begin{proposition}\label{exact}
Suppose $f(T)$, $g(T) \in R[T]$ are monic polynomials 
such that $f(T)g(T) = T^p - 1$. Then 
$$R[G] \buildrel{f(\sigma)}\over\longrightarrow 
R[G] \buildrel{g(\sigma)}\over\longrightarrow R[G]$$ 
is exact. $f(\sigma)R [G]$ is 
the set of all $v \in R[G]$ with $g(\sigma)(v) = 0$ and 
$R[G]/f(\sigma)R[G] \cong g(\sigma)R[G]$. 
If $f(T)$ and $g(T)$ divide $T^p - 1$ with no common factors, then 
$f(\sigma)R[G] \cap g(\sigma)R[G] = f(\sigma)g(\sigma)R[G]$
\end{proposition}

Because the lifting to $R[T]$ can get tiresome, we will write 
$f(\sigma)/g(\sigma)$ to mean $h(\sigma)$ where  
$g(T)|f(T)|(T^p - 1)$ and $h(T) = f(T)/g(T)$. In particular, we define 
$g^*(\sigma) = (\sigma^p - 1)/g(\sigma)$.

\begin{proof}
Clearly $f(\sigma)g(\sigma) = 0$. If $g(\sigma)h(\sigma) = 0$, 
then $g(T)h(T) = (T^p - 1)k(T) = g(T)f(T)k(T)$. 
Since $g(T)$ is monic, it is a non zero divisor. All but the 
last statement are now easy. 

As for the last statement, define  
$c(\sigma) = (\sigma^p - 1)/(f(\sigma)g(\sigma)) = f^*(\sigma)/g(\sigma) = g^*(\sigma)/f(\sigma)$ so $g^*(\sigma) = f(\sigma)c(\sigma)$ 
and $f^*(\sigma) = g(\sigma)c(\sigma)$. 
Then $f(\sigma)R[G] \cap g(\sigma)RG] = 
\{x \in R[G] | g^*(\sigma)x = f(\sigma)c(\sigma)x = 0 = 
g(\sigma)c(\sigma)x\}$ while $f(\sigma)g(\sigma)R[G] = 
\{x \in R[G] | c(\sigma)x = 0\}$. Thus we are done once we prove the 
lemma: 

\begin{lemma}
Suppose $f(t)$, $g(t)$ are monic divisors of $t^p - 1$ 
with no common factors. If $x \in R[G]$ and $f(\sigma)x = g(\sigma)x = 0$ then $x = 0$. 
\end{lemma}

\begin{proof}
We induct on the degree of $f$. If $f$ has degree one, 
it has the form $t - \rho^i$. Note that 
$g(t) = (t - \rho^i)q(t) + g(\rho^i)$ and by assumption  
$t - \rho^i$ does not appear in $g(t)$ so  
$g(\rho^i) = u\eta^d$ where $u$ is a unit of $\Z[\rho]$ 
and $d$ is the degree of $g(t)$. 
If $(\sigma - \rho^i)x = g(\sigma)x = 0$ then 
$\eta^dx = 0$ and $x = 0$ since $\eta$ is a non zero divisor. 

For the general case, write $f(t) = (x - \rho^i)f'(t)$ 
and assume $f(\sigma)x = 0 = g(\sigma)x$. By the above, 
$f'(\sigma)x = 0$ and we are done by induction.
\end{proof} 
\end{proof}

The above results allow us to analyze the group ring 
$R[G]$ as a series of pullbacks. Identify $R[G] = R[T]/(T^p - 1)$. 
Now $T^p - 1 = (T - \rho^{p-1})\ldots(T - 1)$ and note $\tr = 1 + T + \cdots + T^{p-1} = 
(T - 1)^*$. Define
$R[G](p-1) = R[G]$ and $R[G](i-1) = (\sigma - \rho^{i})R[G](i)$. 
Thus $R[G](0) = \tr{R[G]} = R$ and if 
$f_i(t) = (t - \rho^{p-1})(t - \rho^{p-2})\ldots(t - \rho^{i+1})$ then 
$R[G](i) = f_i(\sigma)R[G]$.   
Equivalently, $R[G](i) = \{x \in R[G] | f_i(\sigma)^*x = 0\}$ which 
implies $R[G](i) \cong R[G]/(f_i(\sigma)R[G])$ and this allows us to 
view $R[G](i)$ as a ring.  
Also, 
$R[G](i) = \{x \in R[G] | (\sigma - \rho^i)x \in R[G](i-1)\}$. 
We have particular interest in $R[G](1) = [(\sigma - \rho)(\sigma - 1)]^*R[G] = 
\{x \in R[G] | (\sigma - \rho)(x) \in R[G](0)\} = \{x \in R[G] | 
(\sigma - \rho)(\sigma - 1)(x) = 0\}$.

It will be useful to us to note that the successive 
$R[G](i)$ can be built up using pullback diagrams. 
If $g_i = f_i^* = (\sigma - 1)\ldots(\sigma - \rho^{i})$ and $J_i= R[G]g_i$, 
then $R[G](i) = R[G]/J_i$ and $J_i = (\sigma - \rho^{i})R[G] \cap 
J_{i-1}$. For any $R$ module $M$ we write $M^{(i)}$ 
to mean the $R[G]$ module where $\sigma$ acts via $\rho^i$. 
Thus $R[G]/(\sigma - \rho^{i})R[G] = R^{(i)}$. 
Moreover, if $I = (\sigma - \rho^{i})R[G] + J_{i-1}$ 
then the image of $I$ in $R^{(i)}$ is $\eta^{i}R^{(i)}$ 
and we have the fiber diagram of $R[G]$ modules and commutative 
rings (ignoring the $G$ action):

$$\begin{matrix}
R[G](i)&\longrightarrow&R[G](i-1)\cr
\downarrow&&\downarrow\cr
R^{(i-1)}&\longrightarrow&(R/\eta^{i-1})^{(i-1)}\cr
\end{matrix}$$ 

Using the above diagram and induction we show: 

\begin{proposition}\label{local} 
If $R$ is a local ring, then  every rank one projective 
$R[G]$ module is $R[G]$.
\end{proposition} 

\begin{proof} 
We will show, by induction, that $\Pic(R[G](i)) = 0$. 
If $i = 0$ this is clear. Let $M$ be the maximal ideal of 
$R$. If $\eta \notin M$, then $\eta$ is invertible and 
$R[G](i) \cong R \oplus R[G](i-1)$. If $\eta \in M$, 
then $R(i-1)$ and $(R/\eta^{i-1})$ are local 
and $\Pic(R[G](i-1)) = 0$ by induction. 
Furthermore, since $R$ is local, $R^* \to (R/\eta^{i-1})$ 
is surjective. 
Thus \ref{fiber} implies that $\Pic(R[G](i))$ embeds in 
$\Pic(R[G](i-1)) \oplus \Pic(R) = 0$.  
\end{proof}

As a corollary we have: 

\begin{corollary}
If $V$ is a rank one projective $R[G]$ module then 
there is a faithfully flat extension $R'$ of $R$ 
such that $V \otimes_{R[G]} R'[G] \cong R'[G]$. 
\end{corollary}

We continue with our study of degree $p$ Galois extensions. 
Let $S/R$ be such an extension with Galois group, $G$, generated by 
$\sigma$. We saw that $S$ is a rank one projective $R[G]$ module. 
For this reason we look at the general subject of 
rank one projective $R[G]$ modules. 

\begin{proposition}\label{projexact} 
Let $V \in \Pic(R[G])$. 
Suppose $f(t)$ is a polynomial dividing $t^p - 1$. Then 
$$V \buildrel{f(\sigma)}\over\longrightarrow 
V \buildrel{f(\sigma)^*}\over\longrightarrow V$$ 
is exact. $f(\sigma)V$ is 
the set of all $v \in V$ with $f(\sigma)^*(v) = 0$ and $V/f(\sigma)V \cong f(\sigma)^*V$. 
If $f(t)$ and $g(t)$ divide $x^p - 1$ with no common factors, then 
$f(\sigma)V \cap g(\sigma)V = f(\sigma)g(\sigma)V$
\end{proposition}

\begin{proof} 
Since $V \cong R[G]$ after a faithfully flat extension of $R$, 
this follows from the facts about $R[G]$. 
\end{proof} 

If $V$ is as above, then it is immediate 
that $V$ has trivial Tate cohomology. That is, 
the fixed module $V^G = \tr{V} = (1 + \sigma + \cdots + \sigma^{p-1})V$ 
and if $v \in V$ satisfies $\tr(v) = 0$ then $v = (\sigma - 1)v'$ 
for some $v'$. Thus the above proposition seems to be a strengthening 
of the the triviality of the cohomology of $V$.  

The above proposition allows us to derive a corollary of 
Lemma~\ref{picorder}. 
Let $S_{\rho} = \{ s \in S | \sigma(s) = \rho{s}\}$. 

\begin{corollary}\label{invertibleu} 
Let $S/R$ be $G$ Galois. 
Then $S_{\rho} \cong S/(\sigma - \rho)S \cong S \otimes_{R[G]} 
R[G]/(\sigma - \rho)R[G]$. Thus $S_{\rho}$ is a rank one projective 
$R$ module of order dividing $p$ in $\Pic(R)$. If 
$S_{\rho} = Ru \subset S$ then $u$ is a unit. 
The multiplication on $S$ induces an isomorphism $S_{\rho}^p \cong R$. 
\end{corollary} 

\begin{proof}
The isomorphisms are immediate from Proposition~\ref{projexact} 
and the second statement follows from Lemma~\ref{picorder}. 
Suppose $S_{\rho} = Su$. We need to show that 
for any maximal ideal $M$ of $S$, $u \notin S$. 
Since $\sigma^i(u) = \rho^iu$, it suffices to show 
$u \notin N = \bigcap_{i=0}^{p-1} \sigma^i(M)$. Since 
$N$ is $G$ invariant, it has the form $N = N'S$ 
where $N' \subset R$ is a maximal ideal. But $\bar S = S/N'S$ is Galois 
with group $G$ over the field 
$\bar R = R/N'$. If $\eta \notin N'$ we can apply Kummer 
theory  and deduce that $\bar S_{\rho} \not= 0$. 
If $\eta \in N'$ then $\sigma - \rho$ is the same as $\sigma - 1$ 
on $\bar S$ and $\bar S_{\rho} = (\bar S)^{\sigma} = R \not= 0$. 
On the other hand, $\bar S_{\rho}$  is the image of 
$(\sigma - \rho)^*S = Ru$. It follows that $u \notin N'S$. 

To show the multiplication on $S$ induces $S_{\rho}^p \cong R$ 
we can assume $R$ is local and hence $S_{\rho} = Ru$ 
for $u \in R^*$. The result is now clear. 
\end{proof}

Note that in characteristic $p$ these Picard group elements are all trivial, 
since $S \cong R[G]$. For convenience, if $M$ 
is any $R[G]$ module, we simplify $M/f(\sigma)M$ 
and write it as $M/(f(\sigma))$.

The above facts suggest a very general description 
of degree $p$ Galois extensions, and this description 
will be used in the section on the embedding problem. 

If $R$ contains $1/\eta$ then there is a classical way to 
construct cyclic extensions associated to a rank one projective 
$R$ module $P$ with $\phi: P^p \cong R$.  
We form the tensor algebra $T'' = R[P]$ with $G$ action 
by letting $\sigma$ act on $P^i$ as $\rho^i$. 
We identify $P^{rp}$ with $R$ using $\phi$. 
In $T''$ let $\theta = \phi^{-1}(1) \in P^p$ and form 
$T' = T''(1/\theta)$.  
This yields a sort of universal cyclic extension 
but $R$ is not the base ring. 
In fact, $T''^G$ is the localized Laurent polynomial ring 
$R'' = R[\theta](1/\theta)$ and $T''/R''$ 
is Galois with group $G$. 

In the mixed characteristic case we start with something 
a bit more complicated. 
Assume we are given a $R[G](1)$ rank one projective 
$P(1)$, and we set $P_1 = P(1)/(\sigma - \rho)$ and 
$P_0 = P(1)/(\sigma - 1)$. Then $P(1)$ is defined by the 
fiber diagram:
$$\begin{matrix}
P(1)&\longrightarrow&P_0\cr
\downarrow&&\downarrow\cr
P_1&\longrightarrow&P_2\cr
\end{matrix}$$
where 
$P_2 = P_1/(\rho - 1)P_1 = P_0/(1 - \rho)P_0 = P_0/\eta{P_0}$. 

We assume $P(1)$ is {\bf normalized} 
by which we mean that $P_0 \to P_2$ is $R \to R/\eta{R}$ 
and so $P(1)$ is defined by the isomorphism 
$P_1/\eta{P_1} \cong R/\eta{R}$ induced by 
$\phi: P_1 \to R/\eta$. It will be our habit to use 
the same symbol $\phi$ for the map $P_1 \to R/\eta{R}$ 
and the induced isomorphism $\phi:P_1/\eta{P_1} \cong 
R/\eta{R}$. The above fiber diagram becomes, in the normalized case,: 
$$\begin{matrix}
P(1)&\longrightarrow&R\cr
\downarrow&&\downarrow\cr
P_1&\buildrel{\phi}\over\longrightarrow&R/\eta{R}.\cr
\end{matrix}$$ 

Assume $P(1)$ above is normalized. Since $P(1)$ is projective over $R[G](1)$, 
$P(1)^G = (\sigma - \rho)(P(1)) \cong R$ and $P_1 \cong (\sigma - 1)P(1)$ 
and so it can sometimes be convenient to view $P_1$ and $R$ as $R[G](1)$ 
submodules of $P(1)$. We have: 

\begin{lemma}\label{generation}
If $P(1)$ is normalized as above, $P(1)$ is generated 
as an $R$ module by $P_1$, $R$ and any $x \in P(1)$ with 
$\sigma(x) - \rho{x} = 1 \in R$. $P(1)/(R \oplus P_1) \cong 
R/\eta{R}$. 
\end{lemma}

\begin{proof} 
Let $P' \subset P(1)$ be the $R$ submodule with generators as above. 
$R = P_0 \subset P(1) \to R = P(1)/(\sigma - 1)P(1)$ is multiplication by $\eta$, 
and $P_1 \subset P(1) \to P_1 = P(1)/(\sigma - \rho)P(1)$ is multiplication 
by $-\eta$. Since $P(1) \to P(1)/(\sigma - 1)P(1) = R \subset P(1)$ 
can be identified with multiplication by $\sigma - \rho$, $x \in P(1)$ 
maps to $1 \in R$. Thus $P' \subset P(1) \to R$ is surjective 
and the kernel of $P(1) \to R$ is $P_1 \subset P'$.
\end{proof}

Suppose $Q(1)$ is another normalized rank one projective $R[G](1)$ 
module with $Q_1 = Q(1)/(\sigma - \rho) = P_1$. 
Then $Q(1)$ differs from $P(1)$ because of a different 
isomorphism $\phi': P_1/\eta{P_1} \cong R/\eta{R}$. 
It follows that $\phi = u\phi'$ where $u \in (R/\eta{R})^*$. 
From \ref{fiber} we have 

\begin{lemma}
$Q(1) \cong P(1)$ over $R[G](1)$ if and only 
if $u$ is in the image of $R^*$.
\end{lemma}

Set $R' = R(1/\eta)$ and $M' = M \otimes_R R'$ for any $R[G]$ module $M$. 
It follows that $P(1)' \cong P_1' \oplus R'$ and so 
$P(1) \subset P_1' \oplus R'$. 

We can be more precise. 
We saw in \ref{generation} that we can view 
$R \subset P(1)$ and $P_1 \subset P(1)$ 
as $(\sigma - 1)P(1)$ and $(\sigma - \rho)P(1)$ 
respectively. Moreover, $P(1)/(P_1 \oplus R) \cong R/\eta$. 
Thus $R \oplus P_1 \subset P(1) \subset 
(1/\eta)R \oplus (1/\eta)P_1$ and $\eta{P(1)} \subset R \oplus P_1$. 

We form the tensor algebra $T' = R'[P_1']$ and view 
$P(1) \subset T'$ using the inclusion $P(1) \subset (1/\eta)R \oplus 
(1/\eta)P_1$. We set $R[P(1)] \subset T'$ 
to be the subring generated over $R$ by $P(1)$. 
Note that $P(1)^G = (\sigma - \rho)P(1)$ is being identified 
with $R \subset R[P(1)]$. 

Set $T'' =  
R[P_1] = R \oplus P_1 \oplus P_1^2 \oplus \cdots$ 
and we view $T'' \subset R[P(1)] \subset T'$ using the inclusion 
$R \oplus P_1 \subset P(1)$.    
The definition of $R[P(1)]$ 
implies that there is an action by $\sigma$ on $R[P(1)]$. 
Note that by \ref{generation} $R[P(1)]$ is generated as an algebra over $T''$ by 
any $x$ whose image generates $P(1)/(R \oplus P_1)$ and in particular for 
any such $x$ which can be written $(\theta - 1)/\eta$ for some $\theta \in P_1$. 
Define $P(r) \subset R[P(1)]$ to be the submodule generated by 
all $r$ fold products of elements in $P(1)$. 
Thus $R[P(1)]$ is the sum of all the $P(r)$'s, 
$P(r) \subset (1/\eta^r)R \oplus \cdots \oplus (1/\eta^r)P_1^r$, 
and $P(r)/(R \oplus P_1 \oplus \cdots \oplus P_1^r)$ is annihilated 
by $\eta^r$. 

We are particularly interested in $P(p)$ which is preserved by 
$\sigma$ and more precisely we are interested in the invariant module $P(p)^G$. Note that $\sigma$ acts trivially on $P_1^p$.  
Since $T'^G = R'[{P'_1}^p]$, it is clear that 
$$R \oplus P_1^p \subset P(p)^G \subset (1/\eta^p)R \oplus (1/\eta^p)P_1^p.$$ 
Also, $P(p)/(R \oplus \cdots \oplus P_1^p)$ is spanned by 
the images of $1,x,\ldots,x^p$ as an $R/\eta^p{R}$ module 
where $x = (\theta - 1)/\eta$ is as above. 
Said differently, we can set $M = R + Rx + \cdots + Rx^p$ 
and note that $M + (R \oplus \cdots \oplus P_1^p) = P(p)$. 
Clearly $M$ is the direct sum of the $Rx^i \cong R$. 
Since $\sigma(x) = \rho{x} + 1$, it is also clear that $\sigma(M) = M$. 
Furthermore, since $\sigma$ is the identity on $R \oplus P_1^p$, $P(p)^G$ 
is generated by $R \oplus P_1^p$ and the image of $M^G$. 

In fact, let $R[z]$ be the polynomial ring and define 
$\sigma$ on $R[z]$ by $\sigma(z) = \rho{z} + 1$. 
Then $M$ is isomorphic to the degree less than or equal to 
$p$ part of $R[z]$. Now $R[z]^G = R[z^p + g(z)]$ so 
$M^G$ is generated over $R$ by $R$ and $N_G(x) = x^p + g(x)$. 
Thus: 

\begin{lemma}\label{pgeneration}
Suppose $x \in P(1)$ satisfies $\sigma(x) = \rho{x} + 1$. 
Then $P(p)^G$ is generated as an $R$ module by 
$P_1^p$, $R$, and $N_G(x) = x^p + g(x)$. 
Furthermore, 
$P(p)^G/(R + P_1^p) \cong R/\eta^pR$ where $N_G(x)$ 
maps to $1$. 
\end{lemma}

Said another way, 
we have the fiber diagram (of abelian groups so far):
$$\begin{matrix}
P(p)^G&\longrightarrow&R\cr
\downarrow&&\downarrow\cr
P_1^p&\buildrel{\phi^p}\over\longrightarrow&R/\eta^p{R}\cr
\end{matrix}$$ 
where we explain $\phi^p$ and this diagram further below. 

We need to explore how $P(p)^G$ is derived from 
$P(1)$. Restricting ourselves to normalized $P(1)$'s, 
$P(p)^G$ is determined by a choice of $P_1 \in \Pic(R)$ 
and a $\phi: P_1 \to R/\eta{R}$ inducing an 
isomorphism $P_1/\eta{P_1} \cong R/\eta{R}$, but we need to 
explore the details. To do this, we next note some facts particular to our mixed characteristic 
situation. 

\begin{proposition}
Suppose $P$ is a projective module over $R$ with 
symmetric power $S^p(P)$ and let $\Phi_P: P^p \to 
S^p(P)$ be the canonical surjection. Let $D \subset S^p(P)$ 
be the $R$ submodule generated by all elements 
of the form $\Phi(x \otimes \cdots \otimes x)$. 
Then $S^p(P)/D$ is annihilated by a power of 
$(p-1)!$. 
\end{proposition}

\begin{proof} 
If $Q$ is another projective $R$ module and 
$Q = Q' \oplus P$ for some $Q'$ then there is a split 
surjection $\phi: S^p(Q) \to S^p(P)$ since 
$S^p(Q) \cong S^p(P) \oplus (Q' \otimes S^{p-1}(P)) 
\oplus (S^2(Q') \otimes S^{p-2}(P)) \oplus \cdots 
\oplus S^p(Q')$. Moreover, if $D_Q \subset S^p(Q)$ 
is generated by all $\Phi_Q(x \otimes \cdots \otimes x)$, 
then $D_Q$ maps onto $D$. It follows that it suffices to 
prove the proposition for $Q$, and hence we can assume 
$P$ is a free module over $R$. 

Secondly, let $(p-1)!$ also designate the image 
of this integer in $R$. If this element in nilpotent 
in $R$, the result is automatic. If not, 
set $R'$ be the localization of $R$ at the multiplicative set 
generated by $(p-1)!$. It 
suffices to prove the result for $R'$. 
That is, we may assume $(p-1)!$ is invertible in 
$R$ and then show that $D = S^p(P)$. 

If $v_i \in P$ and $\sum_i a_i = p$ 
then we set $v_1^{a_1}\ldots{v_r^{a_r}}$ 
to be $\Phi_P(v_1 \otimes \cdots \otimes v_r)$ 
where $v_i$ appears $a_i$ times. Clearly 
the monomials $v_1^{a_1}\ldots{v_r^{a_r}}$ 
span $S^p(P)$ and we must show they all 
belong to $D$. Obviously $v^p \in D$ for any 
$v \in P$. We assume by induction that 
$v_1^{a_1}\ldots{v_r}^{a_r} \in D$ for any $a_i$ and 
any $v_i \in D$. 
Now $v_1^{a_1}\ldots(v_r + tv_{r+1})^{a_r}) = 
\sum_i c_it^iv_1^{a_1}\ldots{v_r}^{a_r-i}v_{r+1}^i \allowbreak \in D$ where 
$c_i$ is the appropriate binomial coefficient 
and clearly invertible in $R$. 
We set $w_i = c_iv_1^{a_1}\ldots{v_r}^{a_r-i}v_{r+1}^i$.  
Then $\sum_i t^iw_i \in D$ for any choice of $t \in R$. 
If we choose $t = 0,\ldots,p-1$ successively 
then the associated matrix is the Vandermonde 
matrix with determinant $\pm \prod (i-j)$ 
which is invertible in $R$. Thus $w_i \in 
D$ and hence $(1/c_i)w_i \in D$ which proves the 
result by induction. 
\end{proof}

\begin{theorem}\label{ppowerhomo} 
Suppose $I \in \Pic(R)$ is provided with a surjection  
$\phi: I \to R/\eta$. 
Then there is a unique $\phi_p: I^p \to R/\eta^p$ 
such that $\phi(x)^p \in R /\eta^p$ is well defined and 
$\phi_p(x \otimes \cdots \otimes x) = 
\phi(x)^p$ modulo $\eta^p$.  
\end{theorem}

\begin{proof} 
$I$ embeds in $\prod_M I_M$ 
where the product is over all maximal ideals of $R$. 
Since $I_M = R_M$, the symmetric group $S_p$ 
acts trivially on $I^p$ and we can identify 
$I^p$ with $S^p(I)$. If $x = x' + y\eta$ for 
$y \in I$ then $x \otimes \cdots \otimes x = 
(x' \otimes \cdots \otimes x') + z\eta^p$ and so 
$\phi(x)^p \in R/\eta^p$ is well defined and the 
definition of $\phi_p$ makes sense. 
Thus $\phi^p$ is well defined 
on the image of all the $x \otimes \cdots \otimes x$ 
in $S^p(I)$. Since $p^2(R/\eta^p) = 0$, 
all the $1,\ldots,p-1$ act invertibly 
on $R/\eta^p$. Thus by the above result, 
$\phi^p$ extends uniquely to $I^p$. 
\end{proof} 

This is 
the $\phi^p$ that appears in the pull back diagram for 
$P(p)^G$. Our construction of $P(p)^G$ assumed $P(1)$ 
was normalized, which implied that 
the right side of the diagram defining $P(p)^G$ can be 
fixed to be $R \to \eta^p{R}$. That is, we will say 
$P(p)^G$ is {\bf normalized}. 
It is clear that $P(p)^G$ is a rank one projective module over 
$R *_p R$ which we define by the fiber diagram of rings: 
$$\begin{matrix}
R *_p R&\longrightarrow&R\cr
\downarrow&&\downarrow\cr
R&\longrightarrow&R/\eta^p{R}\cr
\end{matrix}$$ 
Since $P(p)^G$ is automatically 
normalized when $P(1)$ is, it is completely defined as an $R *_p R$ 
module by $\phi^p: P_1/\eta^pP_1 \cong R/\eta^pR$.

Thus in the normalized case the map $P(1) \to P(p)^G$ can be boiled 
down to $\phi \to \phi^p$ where 
$\phi: P_1/\eta{P_1} \cong R/\eta{R}$ and 
$\phi^p: P_1^p/\eta^p{P_1^p} \cong R/\eta^pR$. 

We phrase the above remark in slightly different 
language. Looking at inverses, $\phi$ is determined 
by a choice of generator $\bar \theta \in P_1/\eta{P_1}$ 
and $\phi^p$ by a generator of $P_1^p/\eta^pP_1^p$. 
The map $\phi \to \phi^p$ translates into 
the observation that if $\bar \theta$ generates $P_1/\eta{P_1}$ 
then $\bar \theta^p$ is well defined and generates 
$P_1^P/\eta^pP_1^p$. 
 
Suppose we assume $P_1^p \cong R$ and $\bar \theta$ is as above. 
Then $P_1^p  = r{R}$ and $\bar r = \bar \theta^p{\bar u}$ 
for some $\bar u \in (R/\eta^pR)^*$. Using \ref{fiber} we have:

\begin{lemma} 
$P(p)^G \cong R *_p R$ if and only if $P_1^p \cong R$ 
and if $\bar u \in (R/\eta^p)^*$ is the unit defined 
above, then $\bar u$ is in the image of $R^*$. 
\end{lemma}
 
The above condition arises because:

\begin{theorem}
Assume $S/R$ is $G$ Galois of degree $p$. 
Set $P(1) = S(1)$. Then $P(p)^G \cong 
R *_p R$ as an $R *_p R$ module. 
\end{theorem} 

\begin{proof} 
When $S(1) = P(1)$ then $S_1 = P_1 = \{s \in S | \sigma(s) = \rho{s} \}$ 
so $S_1 = S_{\rho}$ in our previous notation. 
We know $S_1^p \cong R$ via the multiplication map in 
$S$. Thus we need to know that we can identify 
$R \cong P_1^p \to R/\eta^p{R}$ with the 
canonical $R \to R/\eta^p{R}$ up to a unit of $R$. 
But this is just that fact that the multiplication 
map $S_1^p \to R$ is surjective.
\end{proof}

The purpose of the above discussion of $P(p)^G$ 
was that it appears in the construction of a 
general degree $p$ Galois 
extension in mixed characteristic as follows. Given a normalized 
rank projective $R[G](1)$ module $P(1)$, set $P_1 = 
P(1)/(\sigma - \rho)P(1)$. Summarizing the above discussion, 
we have $T'' = R[P_1] \subset R[P(1)] \subset T' = R'[P_1']$ 
and $R[P(1)]$ is generated over $T''$ by $x = (\theta - 1)/\eta$ 
for some $\theta \in P_1$. Also, $T''^G = R[P_1^p] \subset R[P(p)^G] 
\subset R'[P_1']$ and $R[P(p)^G]$ is generated 
over $T''^G$ by $\psi = (\theta^p - 1)/\eta^p$. 

\begin{theorem}\label{generalp} 
Let $P(1)$ be a normalized rank one projective over $R[G](1)$. 
Assume 
$P(p)^G \cong R *_p R$ and so $P(p)^G$ is generated by 
$t = (\psi - 1)/\eta^p$ for $\psi \in P_1^p$ a generator. 
Then $R[P(p)^G] = R[t]$. If $S = R[t](1/\psi) \subset T = 
R[P(1)](1/\psi)$ then $T/S$ is Galois with group $G$. 
\end{theorem}

\begin{proof}
All is clear except that $T/S$ is Galois. 
By descent, it suffices to prove this when $R$ is local 
which implies $P(1)$ is trivial and hence we can choose $x = (\theta - 1)/\eta$ 
which generates $P(1)$ such that $\theta$ generates $P_1 \cong R$. 
Then $x^p + g(x) = t$ where $\theta^p = (1 + x\theta)^p = 1 + t\theta^p$. 
Since $1 + t\theta^p$ clearly generates $P_1^p$ we have that $\theta^p$ 
and hence $\theta$ are invertible in $T$ and $S$. Clearly $T$ is generated 
over $S$ by $x$ and hence $T = S[Z]/(Z^p + g(Z) - t)$.
\end{proof} 

Note that, in the above construction, the ring $S$ defined above is 
independent of the choice of $\psi$. 
That is, if $t' = (\psi' - 1)/\eta^p$ also maps to 1 in $R/\eta^p$, 
and $\psi'$ also generates $P_1^p$, then $\psi' = u\psi$ 
for $u \in R^*$ and $u$ is congruent to 1 modulo $\eta^p$. 
Writing $u = 1  + v\eta^p$ we have $t' = ((1 + v\eta^p)\psi - 1)\eta^p = 
t + v\psi$ making it clear that $R[t] = R[t']$ (actually 
already clear since both are $R[P(1)]$), and, of course, 
$R[t](1/\psi) = R[t](1/(u\psi)) = R[t'](1/\psi')$. 

The above construction defines a degree $p$ cyclic 
over $S = R[t](1/1+t\eta^p)$. Thus $T$ is a module over 
$S[G]$ and we can form 
$T(1) = T \otimes_{S[G]} S[G](1)$. 

\begin{lemma}\label{degreepmodule} 
$T(1) \cong P(1) \otimes_{R[G](1)} S[G](1)$.
\end{lemma}

\begin{proof}
Since, by definition, 
$T(1) = \{z \in T | (\sigma - \rho)(\sigma - 1)(z) = 0\}$ 
it is clear that $P(1) \subset T(1)$ and hence there 
is a map $\phi: P(1) \otimes_{R[G](1)} S[G](1) \to T(1)$. 
Since the domain and range of $\phi$ are both rank one 
projectives over $S[G](1)$, it suffices to prove $\phi$ 
surjective. By faithful flatness, we map assume 
$R$ local and hence $P(1) \cong R[G](1)$. That is, 
there is an $x \in P(1)$ with $\theta = 1 + x\eta$ a 
generator of $P_1$. It follows that $P(p)^G$ is 
generated over $R *_p R$ by $a = x^p + g(x)$ where 
$1 + a\eta^p$ generates $P_1^p$. Now $1 + a\eta^p = 
(1 + x\eta)^p$ and so $1 + x\eta$ is invertible in $T$. 
Since it is in $T_1$, $T_1 = S(1 + x\eta)$. 
Then it follows that $x$ generates $T(1)$ over 
$S[G](1)$, and $\phi$ is surjective. 
\end{proof}

If $\phi: R \to R'$ 
is any homomorphism, then $\phi$ extends 
so $\phi: S \to R'$ exactly when $r' = \phi(t)$ 
satisfies $1 + r'\eta^p \in R'^*$. Of course, when $\phi$ 
exists, it defines 
a $G$ Galois $T'/R'$ by $T' = T \otimes_{\phi} S$. 
The above lemma shows that 
$T'(1) \cong R'[G](1) \otimes_{R[G](1)} P(1)$. 

\begin{theorem}\label{specializedegreep}
Suppose $\phi: R \to R'$ and $U'/R'$ is cyclic Galois of degree $p$. 
Let $T/S$ and $P(1)$ be as above. 
Then there is homomorphism 
$\phi: S \to R'$ such that $U' \cong T \otimes_{\phi} R'$ 
if and only if $U'(1) \cong P(1) \otimes_{R[G](1)} R'[G](1)$.
\end{theorem} 

\begin{proof} 
We have shown one direction. It is clear that if 
$P(1)' = P(1) \otimes_{R[G](1)} R'$ and we use $P(1)'$ 
to build $T'/S'$ as above, then there is an $\phi: S \to S'$ 
such that 
and $T' = T \otimes_S S'$. That is, we may assume $R = R'$, and replace 
$U'$ by $U$. 
We set $P(1) = U(1)$ and we know $P(p)^G \cong R *_p R$. 
Since $U(1) = P(1) \subset U$, we have an induced $G$ morphism 
$\phi: R[P_1] \to U$. Since $\phi$ is a $G$ morphism, 
$\phi(R[P(p)^G]) \subset R$. If $t \in P(p)^G$ generates 
$P(p)^G$ we know $R[P(p)^G] = R[t]$ and $t = (\psi - 1)/\eta^p$ 
where $\psi$ generates $P_1^p$. Since $\psi$ is a unit, 
$\phi$ extends to $R[P(1)](1/(1 + t\eta^p)) \to S$ 
and the result is clear.
\end{proof}

We next relate these questions about the $R[G]$ 
module structure of $S$ to the issues around the separable 
polynomial $Z^p + g(Z) - a$ that we started with. 
Recall that a root of such a polynomials satisfies 
$\sigma(\alpha) = \rho\alpha + 1$ with $1 + \alpha\eta$ 
invertible. 

\begin{theorem}\label{degreepnormalalmost} 
Let $S/R$ be Galois with group $G = <\sigma>$ of prime order 
$p$. 
The following are equivalent. 
\begin{enumerate}
\item There is an an $\alpha \in S$ with 
$\sigma(\alpha) = \rho\alpha + 1$ and 
$1 + \alpha\eta$ invertible. \label{degree:1}

\item $S(1) \cong 
R[G](1)$ as $R[G]$ modules. \label{degree:2}
\end{enumerate}
\end{theorem} 

\begin{proof} 
Assume \ref{degree:2}. 
Let $\alpha \in S(1)$ generate over 
$R[G]$. Since $S(1) \to S(0) = R$ 
is onto, we can assume $\alpha$ maps to 
$1$. That is, $(\sigma - \rho)(\alpha) = 
1$ or $\sigma(\alpha) = \rho\alpha + 1$. 
Since $(\sigma - 1): S(1) \to S_1 \cong R$ 
is onto, $(\sigma - 1)(\alpha) = 
\eta\alpha + 1$ must generate 
$S_1$ over $R$. By \ref{invertibleu}, 
this implies $1 + \eta\alpha$ is invertible. 

Conversely, if $(\sigma - \rho)\alpha 
= 1$ and $1 + \eta\alpha$ is invertible 
then $\alpha$ generates $S(1)$ 
which implies $S(1) \cong R[G](1)$. 
\end{proof} 

Suppose, as above, that $S(1) \cong R[G](1)$. 
That is, suppose there is an $\alpha \in S(1)$ 
such that $\sigma(\alpha) = \rho\alpha + 1$ and 
$1 + \eta\alpha \in S$ is invertible. We claim: 

\begin{proposition}\label{generate} 
For all $0 \leq i \leq p-1$,
$S(i)$ is the span over $R$ of $1,\ldots,\alpha^i$.  Also, 
$S(i) \cong R[G](i)$. 
\end{proposition}

\begin{proof} 
We prove the first statement by induction. We know 
$S(i) = \{x \in S | (\sigma - \rho^i)(x) \in S(i-1)\}$. 
We have $(\sigma - \rho^i)(\alpha^i) = 
(\rho\alpha + 1)^i - (\rho^i)\alpha^i$ and this is 
in the span of $1,\ldots,\alpha^{i-1}$ and thus is in 
$S(i-1)$ by induction. It follows that $\alpha^i \in S(i)$. 

We claim $(\sigma - \rho^i)(R\alpha^i + S(i-1)) = 
S(i-1)$. Since $S(i-1)/(\sigma - \rho^i)S(i-1) 
\cong (R/\eta^i)^{(i)}$, it suffices to show that 
the image of $(\sigma - \rho^i)\alpha^i$ is not in 
$\eta{R}/\eta^iR$. Since $S(i-2) = (\sigma - \rho^{i-1})S(i-1)$, 
and if $z = (\sigma - \rho^{i-1})z'$, then the 
image of $z$ is the image of $(\rho^i - \rho^{i-1})z'$ 
which is in $\eta{R}/\eta^{i}R$. That is, 
all the $\alpha^j$ for $j < i-1$ map to $\eta{R}/\eta^{i}R$. 
However, the $\alpha^i$ coefficient of 
$(\sigma - \rho^i)(\alpha^i)$ is $i\rho^{i-1}$ which 
does not map to $\eta{R}/\eta^i{R}$ because $\alpha^{i-1}$ 
does not. This proves the claim. 

The kernel of the map 
$$S(i) \buildrel{\sigma - \rho^i}\over\longrightarrow S(i-1)$$ 
is $J = \{x \in S | \sigma(x) = \rho^ix\}$. Since 
$(1 + \alpha\eta)^i$ is a unit and is an element of $J$ 
it generates $J$. But $(1 + \alpha\eta)^i \in R\alpha^i + S(i-1)$. 
This proves $R\alpha^i + S(i-1) = S(i)$. 

We also prove the second statement by induction but 
employing a more   
extended argument. 
We assume $S(i-1) \cong R[G](i-1)$ via a choice of 
generator $\gamma_{i-1} \in \alpha^{i-1} + S(i-2)$. 
Also, $S(i)/(\sigma - \rho^i)S(i) \cong R(i)$ 
and we saw above that $(\sigma - \rho^i)S(i) = S(i-1)$
but we need to be explicit about the choices of 
generators. 

\begin{lemma}
We have that $(\sigma - 1)\ldots(\sigma - \rho^{i-1})(\alpha^i)
= u(1 + \eta\alpha)^i$ where $u \in \Z[\rho]$ is a unit. 
\end{lemma}

\begin{proof}
Since $(\sigma - 1)\ldots(\sigma - \rho^{i-1})(S(i-1)) = 0$, 
we can focus on the leading term of 
$(\sigma - 1)\ldots(\sigma - \rho^{i-1})(\alpha^i)$ 
which is $(\rho^i - 1)(\rho^i - \rho)\ldots(\rho^i - \rho^{i-1}) = 
u\eta^i$ for $u \in \Z[\rho]$ a unit. Comparing this to the 
leading coefficient of $(1 + \eta\alpha)^i$ yields the result.
\end{proof} 

To continue with the proof of the second statement 
of \ref{generate}, we analyze the map $(\sigma- \rho^i): S(i) 
\to S(i-1)$ a bit further. Let $\alpha{S(i-1)} \subset S(i)$ 
be the $R$ span of $\alpha^i,\ldots,\alpha$. Note that 
$R(1 + \eta\alpha)^i \, \cap \, \alpha{S(i-1)} = 0$ 
and $R(1 + \alpha\eta)^i + \alpha{S(i-1)} = S(i)$. 
It follows that if $T_i$ is the restriction of 
$(\sigma - \rho^i)$ to $\alpha{S(i-1)}$ then 
$T_i$ is an isomorphism. If we use the basis 
$\alpha^i,\ldots,\alpha$ for $\alpha{S(i-1)}$ 
and $\alpha^{i-1},\ldots,1$ for $S(i-1)$ then the matrix 
associated to $T_i$, call it $A_i$, has entries from 
$\Z[\rho]$ and unit determinant. 

We define $\gamma_i$ inductively as 
$T_i^{-1}(\gamma_{i-1})$. Note that $\gamma_1 = \alpha \in 
S(1)$ and $\gamma_i \in \alpha{S(i-1)}$. 
That is, $\gamma_i$ is the unique element of 
$S(i)$ such that $(\sigma - \rho^i)\gamma_i = \gamma_{i-1}$ 
and $\gamma_i \in \alpha{S(i-1)}$. We claim the following, 
which will finish the proof of \ref{generate}.   

\begin{proposition}\label{gammagenerates}
$\gamma_i$ generates $S(i)$ over $R[G](i)$. 
\end{proposition}

\begin{proof} We have assumed by induction that 
$\gamma_{i-1}$
generates $S(i-1)$. It suffices to show that 
the coefficient of $\alpha^i$ in $\gamma_i$ is a unit. 
Now by definition $\gamma_i$ is in the $\Z[\rho]$ 
span of $\alpha^i,\ldots,\alpha$.  Since the inverse 
of all the matrices $A_j$ have entries in $\Z[\rho]$, 
we will show this coefficient is a unit in $\Z[\rho]$. 
Note this is a question purely about these 
matrices over $\Z[\rho]$. Thus we need only prove this 
fact for a specific Galois extension of $\Z[\rho]$. Also, if 
$\gamma \in \alpha{S(i-1)}$ generates $S(i)$ the coefficient 
of $\alpha^i$ must, conversely, be a unit because 
this coefficient is the image of $\alpha^i$ in 
$S(i)/S(i-1)$. 

The obviously easiest extension to understand 
is the split extension of $\Z[\rho]$. 
That is, we consider $T = \Z[\rho][Z]/(Z^p + g(Z))$. 
If $\alpha$ is a root of $Z^p + g(Z)$, then 
$(1 + \eta\alpha)^p = 1$ and so 
$\alpha = (1,1+\rho,\ldots,1 + \cdots + \rho^{p-2},0)$ 
as an element of $T = \Z[\rho] \oplus \cdots \oplus \Z[\rho]$.

Now $\beta = (1,\ldots,0) \in T$ clearly 
generates $T$ over $R[G]$. Set
$\beta_i = (\sigma - \rho^{p-1})(\sigma - \rho^{p-2})\ldots(\sigma - \rho^{i+1})(\beta)$ so  $\beta = \beta_{p-1}$ and 
$(\sigma - \rho^i)\beta_i = \beta_{i-1}$. 
Since $(t - \rho^{p-1})\ldots(t - \rho^{i+1})$ has degree 
$p - 1 - i$, $\beta_i$ has its last $i$ entries equal to $0$.  
In particular, $\beta_i$, for $i > 0$, has a 0 
in the last entry. Also $\beta_0 = \tr{\beta} = 
(1,1,\ldots,1)$ which is the element $1 \in \Z[\rho]$. 

Let $\gamma_i \in \alpha{T(i-1)}$ be defined for this extension. 
Since $\alpha{T(i-1)}$ is the span of $\alpha^i,\ldots,\alpha$, 
for $i > 0$, $\gamma_i$ has last entry $0$. 

\begin{lemma} 
For this extension, $\gamma_i = \beta_i$.
\end{lemma}

\begin{proof}
It is easy to check that $\gamma_0 = 1 = \beta_0$ and 
$\gamma_1 = \alpha = \beta_1$. Assuming the result for 
$i-1$, we have $(\sigma - \rho^i)(\beta_i - \gamma_i) = 0$ 
so $\beta_i - \gamma_i = x(1 + \alpha\eta)^i$ for some 
$x \in \Z[\rho]$. We see that the last entry of 
$x(1 + \alpha\eta)^i$ is $x$, and so $x = 0$. 
\end{proof}

Since $\beta_i$ generates $T(i)$, its $\alpha^i$ coefficient
must be a unit of $\Z[\rho]$. 
\end{proof}
We have proven \ref{gammagenerates} and \ref{generate}
\end{proof} 

We now can extend \ref{degreepnormalalmost}. 

\begin{theorem}\label{degreepnormal} 
Let $S/R$ be Galois with group $G = <\sigma>$ of prime order 
$p$. 
The following are equivalent. 
\begin{enumerate}
\item There is an an $\alpha \in S$ with 
$\sigma(\alpha) = \rho\alpha + 1$ and 
$1 + \alpha\eta$\linebreak invertible. \label{pnormal:1}
\item $S(1) \cong R[G](1)$ as $R[G]$ modules. \label{pnormal:2}
\item $S \cong R[G]$ as $R[G]$ modules. \label{pnormal:3}
\end{enumerate}
\end{theorem}

With this tool we consider the specialization process 
of \ref{specializedegreep} above a bit more. 
Assume we have a $P(1) \in \Pic(R[G](1))$ with 
$P(p)^G \cong R *_p R$ and hence a $G$ Galois 
extension $T/S$ where $S = R[t](1/(1 + t\eta^p)$. 
This defines $T_0/R$ using $\phi_0: S \to R$ 
defined by $\phi(t) = 0$. Note that this is not split if 
$P(1)$ is not trivial. If $1 + a\eta^p \in R^*$ 
then there is also a $\phi_a: S \to R$ defined by 
$\phi_a(t) = a$ and hence a $T_a/R$ which is also 
$G$ Galois. The question we need to ask is the 
relationship between $T_0$ and $T_a$ or what can we say about 
$D = T_aT_0^{-1}$ which is also $G$ Galois over $R$.

\begin{corollary}\label{differencedegreep}
$D \cong R[Z]/(Z^p + g(Z) - a)$ with Galois 
group generated by $\sigma$ where $\sigma(x) = \rho{x} + 1$ 
and $x$ is the image of $Z$. That is, $D \cong R[G]$ as an 
$R[G]$ module. 
\end{corollary} 

The summary of the above result is that if $T$ and $T_a$ 
have the same Galois module structure they differ by an extension 
with trivial Galois module structure. The proof is 
trivial since ``Galois module structure" also means 
``element of $\Pic(R[G])$". If $T_0D = T_a$ then by 
\ref{moduleproduct} we have $T_0D = T_a$ in the Picard 
group $\Pic(R[G])$ and so $D$ is the unit in this Picard group 
and hence $D \cong R[G]$ as $R[G]$ modules. 

Note that the above result is also a special case of \ref{extenddifference}, 
which has an independent proof. 

From Theorem~\ref{fiber}, we can extract the exact sequence 
$R^* \oplus R^* \to (R/\eta{R})^* \to \Pic(R[C](1)) 
\to \Pic(R) \oplus \Pic(R)$.  
In this paper we will frequently specialize to the case 
$R = \Z[\rho][x_1,\ldots,x_t]_Q$ where $Q \subset  1 + \eta^pM$ and 
$M \subset \Z[x_1,\ldots,x_t]$ is generated 
by the $x_i$'s. 
Then $R/\eta{R} = F_p[x_1,\ldots,x_n]$ so $(R/\eta{R})^* = 
F_p^*$ and $F_p^*$ is in the image of $\Z[\rho]^* \subset R^*$.  
Thus in this special case $\Pic(R[C](1)) \to \Pic(R) \oplus 
\Pic(R)$ is injective. Moreover the element $S(1)$ we care about 
automatically has $S(1)/(\sigma - 1)S(1) \cong R$ and so we see that 
$P(1) \in \Pic(R[C](1))$ is trivial if and only if $P_1 \in 
\Pic(R)$ is trivial. However, in this special case $\Pic(R) = 
\Pic(\Z[\rho])$ and this is not trivial but is concentrated 
in the constants. From this we can see why our generic degree 
$p$ cyclic Galois $S/R$ turns out to be simple. 

\begin{lemma}
Suppose $R = \Z[\rho][x_1,\ldots,x_n](1/s)$ is as above 
and $S/R$ is cyclic degree $p$ Galois. Let $M = (x_1,\ldots,x_n)R$. 
Then $S/R$ is simple (i.e. $S \cong R[G])$) if $S/MS$ is the 
split Galois extension of $R/M$.
\end{lemma}

\begin{proof}
The composition 
$\Pic(\Z[\rho]) \to \Pic(R) \to \Pic(R/M)$ is an isomorphism. 
\end{proof}

\section{Albert's Criterion}

We have now developed some theory about degree $p$ extensions, 
and we want to use it to construct extensions of degree $p^n$, 
as before in a mixed characteristic setting. 
Like in the previous section, we only care about the order 
$p$ cyclic subgroup of $R^*$ for our algebras $R$. 
That is, we confine ourselves to cases A and C, meaning that 
$\rho$ is a primitive $p$ root of one including $\rho = -1$ 
when $p = 2$ and then $\Z[\rho] = \Z$. 

Our approach 
will be to assume the base has $\rho$, 
and work one degree $p$ cyclic at a time. 
Of course to accomplish this we need some understanding 
of the so called extension problem - when can one extend 
a degree $p^n$ cyclic to one of degree $p^{n+1}$? 
When the base is a field, one can use the classic Albert 
criterion. We require such a criterion for commutative rings.

We next review the result of Albert. 
Suppose $F$ has characteristic unequal to a prime $p$ and contains 
a primitive $p$ root of one $\rho$. Let 
$K/F$ be a cyclic Galois extension of degree $p^r$. Then 
$K/F$ extends to $L/F$ cyclic of degree $p^{r+1}$ if and only if 
$\rho$ is a norm from $K$, which is equivalent to the cyclic crossed 
product $\Delta(K/F,\rho)$ being trivial in the Brauer group. 

Our more general result is a bit more complicated. 
To derive it, we start with some general well known observations about Azumaya algebras. 
If $B'/R'$ is an Azumaya algebra 
we say an $B'$ module $Q'$ 
is a {\bf splitting module} if the module action of 
$B'$ on $Q'$ induces $B' \cong \End_{R'}(Q')$. 

\begin{lemma}\label{splittingmodule}
Suppose $B'/R'$ is a split Azumaya algebra. 
Then there is a bijection between splitting 
modules for $B'$ and elements of $\Pic(R')$. 
\end{lemma}

\begin{proof} 
If $B' \cong \End_{R'}(P')$ then by the 
Morita theorems (e.g., \cite[p.~16--23]{DI}) the map 
$M \to P' \otimes_R M$ is an equivalence of 
the category of $R$ modules and left $B'$ modules. 
Checking ranks yields the lemma, 
\end{proof} 

We require a few more standard facts about splitting 
modules of Azumaya algebras. The key point is the 
appearance of elements of $\Pic$. 

\begin{lemma}
Suppose $A/R$ is a degree $n$ Azumaya algebra with 
$R \subset S \subset A$ such that $S$ is commutative, 
$A$ is projective as an $S$ module using the left action, 
and $S/R$ has degree $n$. If $P$ is a splitting module 
for $A$, then $P$ is a rank one projective module 
over $S$. 
\end{lemma} 

\begin{proof} 
The inclusion $S \subset A$ makes $P$ an $S$ module. 
As an $R$ module, $\End_R(P) \cong P \otimes_R P^*$ where 
$P^* = \Hom_R(P,R)$. In fact, if $x \in P$ and $f \in P^*$ 
the image of $x \otimes f$ is the map $a(y) = xf(y)$. 
If $s \in S$, then $sa(y) = s(xf(y)) = s(x)f(y)$. 
That is, $S$ acts on $A = P \otimes_R P^*$ via its action 
on $P$. Since $A$ is projective over $S$, and $P^*$ 
is faithfully flat over $R$, it follows that 
$P$ is projective over $S$. Checking degrees shows 
that $P$ is rank one.
\end{proof}

To begin our study of the commutative ring 
case, we fix some notation. $S/R$ will be a 
cyclic Galois extension of degree $p^r$ 
and, when it exists, $T/S/R$ will be such that $T/R$ is cyclic 
Galois of degree $p^{r+1}$. 
We have $1 \subset C \subset G$ where 
$G$ is the Galois group of $T/R$, 
$C$ that of $T/S$ and $G/C$ the Galois group 
of $S/R$, so $C$ has order $p$, $G/C$ has order 
$p^r$ and $G$ has order $p^{r+1}$. 
We fix $\sigma \in G$ to be a generator and 
then $\tau = \sigma^{p^r}$ is a generator 
of $C$. 

Since $R \subset S$ 
we can view the group algebras $R[C] \subset S[C]$. Since 
$R$ contains $\rho$, we have the description of $R[C]$ 
and $S[C]$ as iterated pull backs. In particular, 
we have $S[C](1) = ((\tau - \rho)(\tau - 1))^*S[C] 
\cong S[C]/(\tau - \rho)(\tau - 1)S[C]$ and similarly for 
$R[C](1)$. The second description of $S[C](1)$ makes it a ring. 
We view $S[C](1)$ as a ring extension of $ R[C](1)$ 
as usual. Let $\tilde\rho \in R[C](1) \subset S[C](1)$ be the image of $\tau$. 

In our generalization of Albert's result 
it will be useful to 
define the Azumaya algebra $A = \Delta(S[C]/R[C],\tau)$, where 
$\tau \in C \subset R[C]$. Then define 
$A(1) = A \otimes_{R[C]} R[C](1) = \Delta(S[C](1)/R[C](1),\tilde\rho)$. 
Since $A$ contains an element $\sigma$ such that $\sigma$ gives the Galois 
action on $S$ via conjugation and $\sigma^{p^r} = \tau$, 
we can also describe $A$ as a 
twisted group ring 
$S^*[G]$ where $G$ has an action on $S$ given by $G \to G/C = \Gal(S/R)$. 

It is useful to explicitly state the above 
splitting modules results as they apply 
to $A(1)$. 

\begin{lemma}
Suppose $P(1)$ is a splitting module for 
$A(1)$. Then $P(1)$ is a 
projective $A(1)$ module and a rank one 
projective $S[C](1)$ module. 
Conversely, if $P(1)$ is a rank one 
projective $S[C](1)$ module which extends 
to an $A(1)$ module, then $A(1) \cong 
\End_{R[C](1)}(P(1))$. 
\end{lemma} 

The above makes it clear that the modules 
$P(1)$ we are interested in are just rank 
one $S[C](1)$ modules with a ``compatible 
action" by $G$. 

We look at $A(1)$ in a different way for a moment. 
Define $A_0 = A(1)/(\tau - 1)A(1)$, 
$A_1 = A(1)/(\tau - \rho)A(1)$ and 
$A_2 = A(1)/(\tau-1,\tau-\rho)A(1)$.  
Let $\bar S = S/\eta{S}$ and $\bar R = R/\eta{R}$. Then 
$A_0 \cong \Delta(S/R,1)$.  Since $S/R$ 
is Galois, $\End(S/R,1) \cong \End_R(S)$. 
Moreover,  
$A_1 = \Delta(S/R,\tilde \rho)$ and $A_2 = \Delta(\bar S/\bar R,1) \cong 
\End_{\bar R}(\bar S)$. Of course, 
there is a fiber diagram as below. 
$$\begin{matrix}
A(1)&\longrightarrow&A_0\cr
\downarrow&&\downarrow\cr
A_1&\longrightarrow&A_2\cr
\end{matrix}$$

If $P(1)$ is a splitting module for $A(1)$, 
we set 
$P_1 = P(1)/(\tau- \rho)P(1)$, $P_0 = P(1)/(\tau - 1)P(1)$, 
and $P_2 = P_1/\eta{P_1} = P_0/\eta{P_0}$. 
Then there is a fiber diagram 
$$\begin{matrix}
P(1)&\longrightarrow&P_0\cr
\downarrow&&\downarrow\cr
P_1&\longrightarrow&P_2\cr
\end{matrix}$$ 
where $P_i$ is a splitting module for $A_i$
and the morphisms in the above diagram 
are $A(1)$ module morphisms. 
Note that $C$ acts trivially on $P_0$ (and 
$P_2$). On the other hand, $\sigma^{p^r} = \tau$ acts on $P_1$ 
by multiplication by $\rho$.

Conversely, suppose $P_1$ and $P_0$ 
are splitting modules for $A_1$ and $A_0$ 
which makes them $A(1)$ modules. 
Suppose $P_1/\eta{P_1} \cong P_0/\eta{P_0}$ 
as $A(1)$ and hence $A_2$ modules. If 
$P(1)$ is defined by the above fiber diagram, 
then clearly $P(1)$ is a splitting module 
for $A(1)$. Note that if $P(1)$ is a splitting module 
for $A(1)$, then $P_0$ and $P_1$ are projective rank one 
$S$ modules. 

By \ref{splittingmodule} we are free 
to alter $P(1)$ by a rank one module 
from $R[C](1)$. We use this to ``normalize" 
$P(1)$ as follows. 

We have $A_0 = A(1)/(\tau - 1)A(1) = 
\Delta(S/R,1) \cong \End_R(P_0) \cong \End_R(S)$. 
Thus there is an $I \in \Pic(R)$ such that 
$P_0 \cong I \otimes_R S$, the isomorphism 
preserving the $G/C$ action or equivalently the isomorphism 
being an $A(1)$ module map. $P_0/(\tau - \rho){P_0} = 
I/\eta{I} \otimes S/\eta{S}$. If we define $R[C]_1 = 
R[C](1)/(\tau - \rho)R[C](1) \cong R$, then 
$I$ is an $R[C]_1$ module.  Define $J$ by the fiber diagram: 
$$\begin{matrix}
J&\longrightarrow&I\cr
\downarrow&&\downarrow\cr
I&\longrightarrow&I/\eta{I}.\cr
\end{matrix}$$ 
It is clear that $J \in \Pic(R[C](1))$ and we can 
replace $P(1)$ by 
$$P(1) \otimes_{R[C](1)} J^{\circ}$$ 
and still have $A(1) \cong \End_{R[C](1)}(P(1))$ 
but {\bf normalized} so that $P_0 \to P_2$ 
is $S \to S/\eta{S}$. 

A normalized $P(1)$ is determined by $\psi: P_1 \to S/\eta{S}$ 
preserving the $G$ action and inducing $\psi: P_1/\eta{P_1} \cong S/\eta{S}$. 
However, $P_1/\eta{P_1}$ has a trivial action by $C$ and so the 
isomorphism $\psi$ is given by $\phi: (P_1/\eta{P_1})^{G/C} \cong R/\eta{R}$. 

\begin{proposition}
Suppose $S/R$, $G \subset C$, and $A(1)$ are as above. 
If $A(1) \cong \End_{R[C](1)}(P(1))$, then we can 
choose $P(1)$ to be normalized by which we mean that $P_0 \to P_2$ is $S \to S/\eta{S}$. 
A normalized 
$P(1)$ is determined by the isomorphism 
$\phi: (P_1/\eta{P_1})^{G/C} \cong R/\eta{R}$. 
If $P(1)'$ is another normalized $A(1)$ 
module with $P'_1 = P(1)'/(\tau - \rho)P(1)' = P_1$, 
and $P(1)'$ is defined by $\phi': (P_1/\eta{P_1})^{G/C}  \cong R/\eta{R}$, 
then $\phi' = u\phi$ for a unit $u \in (R/\eta{R})^*$. 
We have $P(1)' \cong P(1)$ as $A(1)$ 
modules, if and only if 
$u$ is in the image of $R^*$. 
\end{proposition} 

\begin{proof} 
Most of this is clear or noted above. 
Let $\psi: S \to S/\eta{S}$ be the standard map. 
If $P(1) \cong P(1)'$ as $A(1)$ modules, 
then it is given by $\psi_0: S \cong S$ and 
$\psi_1: P_1 \cong P_1$ such that $\psi \circ \psi_0 = \psi$ 
and $\phi' = \phi \circ \psi_1$. Now each $\psi_i$ is defined 
by multiplication by a unit $u_i \in S^*$. 
The fact that the $\psi_i$ are $A(1)$ morphisms means they 
preserve the $G$ action and so $u_i \in R^*$. 
It follows that $u_1$ is a preimage of $u$. 
\end{proof} 

\begin{corollary}
There is a one to one correspondence between normalized 
splitting modules $P(1)$ for $A(1)$ and the pairs 
$(P_1,\phi)$ where $P_1$ is a rank one projective $S$ module with 
compatible $G$ action (i.e. an $A_1$ module) and isomorphism 
$\phi: (P_1/\eta{P_1})^{G/C} \cong R/\eta{R}$. 
Two isomorphisms $\phi$ and $\phi'$ for the same $P_1$ always have relation 
$\phi' = u\phi$ for $u \in (R/\eta)^*$ and yield 
isomorphic $P(1)$'s if and only if $u$ is in the image of 
$R^*$.  
\end{corollary}  

Because it appears in the diagram, we defined 
``normalized" as a property of the morphism 
$P_0 \to P_2$. But it is really a property 
of $P_0$. 

\begin{lemma}\label{shortnormalized}
If $P(1)$ is a splitting module for $A(1)$ 
then $P(1)$ is normalized if and only if 
$P_0 \cong S$ as $A(1)$ modules or equivalently 
that $P_0^{G/C} \cong R$. 
\end{lemma}

\begin{proof} 
Since $C$ acts trivially on $P_0$ the $A(1)$ 
module action amounts to the $S$ action and the 
action of $G/C$. Thus the second and third 
statement are equivalent. Obviously if $P(1)$ is normalized 
we have $P_0 \cong S$ as $A(1)$ modules. 
Conversely, the morphism $P_0 \to P_2$ is just the 
natural map $P_0 \to P_0/(\tau - \rho)P_0 = P_0/\eta{P_0}$ 
and if $P_0 \cong S$ preserving the $G/C$ action, 
$P_0 \to P_2$ is $S \to S/\eta{S}$.
\end{proof}

Let us note now that all the above makes sense even when 
$\eta$ is a unit of $R$ (recall that $R$ is just a flat algebra 
over $\Z[\rho]$ and could contain $1/\eta$). 
Of course if $\eta$ is invertible then $R/\eta{R} = 0$ etc., 
and so all fiber diagrams are trivial. That is, 
$S[C](1) = S \oplus S$, $R[C](1) = R \oplus R$, and $P(1) = 
P_1 \oplus P_0$. This pattern will continue. 

With $S$ as the base ring we review and augment the definition of the 
general $C$ Galois extension from \ref{generalp}. 
Let $P(1)$ be a projective module over $A(1)$ which is a rank one 
projective module over 
$S[C](1)$ as above. We may assume 
$P(1)$ is normalized and so $P_0 \to P_2$ 
is $S \to S/\eta{S}$. Thus $P(1)$ is defined 
by $\phi: (P_1/\eta{P_1})^{G/C} \cong R/{\eta}R$. 
 
We first form $U' = S'[P_1']$ 
where $S' = S(1/\eta)$ and $P_1' = P_1 \otimes_S S'$. 
Recall that $S \oplus P_1 \subset P(1) \subset (1/\eta)S \oplus (1/\eta)P_1 \subset
S'[P_1'] = U'$. Of course $P(p) \subset U'$ is generated by $p$ fold products 
of elements of $P(1)$ and so $P(p)$ inherits a $\sigma$ action making 
it a $G$ module. Thus $P(p)^C$ is a 
$S *_p S$ module and a $G/C$ module. 
As we observed in \ref{ppowerhomo}, 
$\phi$ induces the $\phi^p: P(p)^C/\eta^pP(p)^C \cong 
S/{\eta^p}S$ that defines $P(p)^C$. 
Checking the definition makes it clear that 
$\phi^p$ is also a $G/C$ morphism. 

We note: 

\begin{lemma}
$S *_p S$ is a $G/C$ Galois extension of $R *_p R$.
\end{lemma}

\begin{proof}
We express $S *_p S$ in a different way. 
Consider the polynomial ring $R[t]$ with ideals 
$I = tR[t]$ and $J = (t - \eta^p)R[t]$. 
Then $I + J = \eta^pR[t] + tR[t]$ while 
$R/I = R$ and $R/J = R$. It follows from 
\ref{idealfiber} that $R *_p R = R[t]/(I \cap J)$. 
Tensoring $R[t]/(I \cap J)$ by $S$ proves the result.
\end{proof} 

In the previous section we used $R[G](1)$ rank one 
projectives to construct cyclic degree $p$ extensions. 
Here we use $A(1)$ modules $P(1)$, which are rank one projectives 
over $S[C](1)$, to construct 
extensions of $S$ cyclic over $R$. 
Since $P(p)$ has a $G$ action, $P(p)^C$ has a 
a $G/C$ action which is easily seen to be semilinear 
and hence a descent action. That is, 
$P(p)^{G}$ is a rank one projective as a module 
over $R *_p R$. For the moment assume that $P(p)^G \cong R *_p R$ 
in order construct our extension of $S/R$. 
Note that $P(p)^G \cong R *_p R$
means that $P(p)^G$  is generated by an element 
$t = (1/\eta^p)(\psi - 1)$ where $\psi$ generates $(P_1^p)^G$ over $R$. 
We set $V = S[P(p)^C](1/\psi) \subset U = S[P(1)](1/\psi)$ so 
$U/V$ is $C$ Galois and both rings have an inherited $\sigma$ action. 

Since $t$ is $G$ fixed, and $V = S[t](1/(1 + t\eta^p))$ 
has a $G/C$ action extending the one on $S$, and we have 
$V^{G/C} = R[t](1/(1 + t\eta^p)) = R'$. That is, $V \cong 
S \otimes_R R[t](1/(1 + t\eta^p))$. Since $U/V$ 
is $G$ Galois and $V/R'$ is $G/C$ Galois we have 
by \ref{galoistower} that $U/R'$ is $G$ Galois. 
Now we set $T_0 = U/tU$ so  
$T_0/S/R$ is $G$ Galois. 
We have shown: 

\begin{theorem}\label{albert1}
Suppose $S/R$ is $G/C$ Galois, where $C \subset G$ are cyclic groups 
of order $p$ and $p^{r+1}$ respectively. Suppose $P(1)$ 
is a normalized projective splitting $A(1)$ module which is rank one 
over $S[C](1)$, and assume 
$P(p)^G \cong R *_p R$. Then $S/R$ extends to a 
$G$ Galois extension $T/R$. 
\end{theorem}

Of course the above construction gives other 
extensions, since we can define $f: R[t](1/1 + t\eta^p) \to R$ 
with $f(t) = a$ whenever $1 + a\eta^p$ is invertible. 
Let the resulting extension be $T_a/S/R$. After we prove a 
converse to the above result, we will describe all 
the extensions that $U$ realizes. 

We require a useful criterion for the 
triviality of $P(p)^G$. 
Suppose $P(1)$ is a normalized splitting module 
for $A(1)$. Then $P_1 \in \Pic(S)$. The module 
$P_1^p$ has a trivial $C$ action and so $(P_1^p)^{G/C} 
\in \Pic(R)$. Also, $P_1/\eta{P_1}$ has a trivial 
$C$ action and so $(P_1/\eta{P_1})^{G/C} \in \Pic(R/\eta{R})$. 
If $P(1)$ is a normalized splitting module, 
$(P_1/{\eta}P_1)^{G/C} \cong R/\eta{R}$ and so there is a 
given generator $\bar \theta \in P_1/\eta{P_1}$ fixed by 
$G/C$. If 
$(P_1^p)^{G/C} \cong R$, this isomorphism defines a 
generator $\psi \in P_1^p$ fixed by $G/C$. Now 
$\bar \theta$ induces a generator $(\bar \theta)^p$ 
for $P_1^p/\eta^pP_1^p$ and $\psi$ maps to a generator 
$\bar \psi \in P_1^p/\eta{P_1^p}$ so $\bar \psi = 
(\bar \theta)^p\bar u$ where $\bar u \in (R/\eta^pR)^*$. 

\begin{proposition}\label{pictrivial}
Suppose $P(1)$ is a normalized splitting module for 
$A(1)$, inducing a $G/C$ fixed generator 
$\bar \theta \in P_1/\eta{P_1}$.  
$P(p)^G \cong R *_p R$ if and only if $(P_1^p)^G \cong 
R$ and for the associated generator $\psi$ we have 
$\bar \psi = (\bar \theta)^p\bar u$ where $\bar u$ is in the image of 
$R^*$.  
\end{proposition}

To prove the converse of \ref{albert1} we start with:  

\begin{proposition}
Suppose $S/R$ extends to a $G$ Galois extension $T/R$. 
Then $T$ is a projective $R[C]$, $S[C]$, and $A$ module. 
\end{proposition}

\begin{proof}
Since $S[C]$ and $A$ are separable over $R[C]$, it suffices 
to prove $T$ is projective over $R[C]$. But this is clear 
as it is projective over $R[G]$. 
\end{proof} 

Suppose $T \supset S \supset R$ is such that 
$T/R$ is $G$ Galois and $T^C = S$. Let 
$T(1) = T/(\tau - \rho)(\tau - 1)T \cong 
((\tau - \rho)(\tau - 1))^*T \subset T$ 
which is a projective rank one 
$S[C](1)$ module and so we can form $T(p)^C$ as above. 
Let $T_0 = T/(\tau - 1)T$ and $T_1 
= T/(\tau - \rho)T$ as usual. We need to investigate 
the properties of $T(1)$. 

\begin{lemma}
$T_0 \cong S$. $T(p)^C$ has a $G/C$ action and 
$T(p)^G \cong R *_p R$. 
\end{lemma} 

\begin{proof} 
The process of forming $T(p)^C$ preserves the 
$G$ actions and so $G/C$ acts on $T(p)^C$. 
We know $T(p)^C \cong S *_p S$ and since the 
map $(T_1)^p \to S$ is just multiplication the above 
isomorphism preserves $G/C$ actions. 
\end{proof} 

Note that in the above argument 
$T_1^p = S\psi$ and the generator of $T(p)^G$, 
$x = (\psi - 1)/\eta^p \in S^G = R$. 

We can now give our extension of Albert's criterion. 
The proof of all the pieces are above. 

\begin{theorem}\label{albert2}
Suppose $G \supset C$ are cyclic groups of order $p^{r+1}$, $p$ 
respectively. Let 
$S/R$ be a $G/C$ Galois extension. Then the following are equivalent. 
\begin{enumerate}
\item $S/R$ extends to a $G$ Galois extension $T/R$. \label{albert2:1}

\item $A(1) = \Delta(S[C](1)/R[C](1),\tilde\rho)$ is a split Azumaya 
algebra and there is a choice of normalized splitting module $P(1)$ with \linebreak
$P(p)^G \cong R *_p R$. \label{albert2:2}
\end{enumerate}
\end{theorem}

Because of the above result and \ref{pictrivial}, we think of the ``generalized" 
Albert criterion as having three obstructions. The first is the 
Brauer group criterion that $\Delta(S[C](1)/R[C](1),\tilde \rho)$ 
be split. The second is that one can choose the normalized 
splitting module $P(1)$ such that $(P_1^p)^{G/C} \cong R$. 
The third part is the requirement that a certain unit of $(R/\eta^pR)^*$ 
lifts to $R^*$. 

Our proof of \ref{albert1} involved constructing 
the $G$ Galois $U/V/R'$ where $R' = R[t](1/(1 + t\eta^p))$ 
and then specializing it. We next want to characterize all 
the specializations of $U$. That is, we want to characterize 
all the $T_a = U \otimes_{\phi} R''$ where $\phi: R' 
\to R''$ is defined by $\phi(t) = a$. To that end, 
let $Q(1) = P(1) \otimes_R R'$ which is an $A(1)' = A(1) \otimes_R R'$ 
splitting module. Form $U(1) = U/(\tau - \rho)(\tau- 1)U$. 
Note that by convention we extend scalars by tensoring on the right, but $A(1)' = R' \otimes_R A(1)$ also and since 
$P(1)$ is a left $A(1)$ module this order works a bit better. 
Similarly $Q(1) = R' \otimes_R P(1)$. 

\begin{lemma}
As a $A(1)'$ module, $U(1) \cong Q(1)$. 
\end{lemma} 

\begin{proof}
In particular, $U/V$ is the $C$ Galois extension 
defined by $P(1)$, and so by \ref{degreepmodule}, 
$V[C](1)P(1) = U(1)$. Since $P(1)$ is an $S[C](1)$ 
module $S[C](1)P(1) = U(1)$ and we have $VP(1) = U(1)$. 
Since $V \subset A(1)'$ we have $A(1)'P(1) = U(1)$ 
so the induced $A(1)'$ morphism $\phi: R' \otimes_R P(1) = 
A(1)' \otimes_{A(1)} P(1) \to Q(1)$ is a surjection. 
Since the domain and range are splitting modules 
for $A(1)'$ we have that $\phi$ is an isomorphism. 
\end{proof}

We can state our theorem. We give the result for a fixed 
base ring $R$ for simplicity, but the generalization 
to an arbitrary $R$ algebra is trivial. 

\begin{theorem}\label{specializeanydegree}
Let $S/R$ be $G/C$ Galois as above and let $P(1)$ 
be a splitting module for $A(1) = 
\Delta(S[C](1)/R[C](1),\tilde\rho)$ such that $P(p)^G 
\cong R *_p R$. Define $U/V/R'$ as above. 
If $T/S/R$ is $G$ Galois there is a $\phi: R' \to R$ 
such that $T \cong U \otimes_{\phi} R$ if and only if 
$T(1) \cong P(1)$ as left $A(1)$ modules. 
\end{theorem}

\begin{proof} 
The lemma shows that the specializations of $U(1)$ 
are all isomorphic to $P(1)$ as $A(1)$ modules. 
For the converse, an isomorphism $\phi: T(1) \cong P(1)$ 
extends to an $S$ algebra morphism $\phi: S[P(1)] \to T$ 
which maps $P(1) \to T(1) \subset T$ just as in \ref{specializedegreep}. 
It is also clear 
that $\phi$ preserves the $G$ action and 
$\phi(R[P(p)^G]) = R$. If $t \in P(p)^G$ generates 
$P(p)^G$ we know $R[P(p)^G] = R[t]$ and $t = (\psi - 1)/\eta^p$ 
where $\psi$ generates $P_1^p$. Since $\psi$ is a unit, 
$\phi$ extends to $R[P(1)](1/(1 + t\eta^p)) \to S$ 
and the result is clear. 
\end{proof} 

If $T/S/R$ and $T'/S/R$ are two $G$ Galois extensions 
extending $S/R$ and are isomorphic as $A$ modules, the above result shows that they 
are both specializations of the $U/V/R'$ defined 
above. They are also related in the following way. 

\begin{theorem}\label{extenddifference} 
Suppose $T/S/R$ and $T'/S/R$ are both $G$ Galois 
and $T \cong T'$ as modules over $A = \Delta(S[C]/R[C],\tau)$. 
Then $T' = TD$ where $D/R$ is cyclic of degree $p$ and 
$D \cong R[C]$ as $R[C]$ modules. That is, 
$D = R[Z]/(Z^p + g(Z) - a)$ for some $a$. 
\end{theorem} 

Of course by \ref{difference} $T' = TD$ for 
$D/R$ cyclic of degree $p$. The point is the 
trivial Galois structure of $R$. 

\begin{proof} 
We recall the proof of \ref{difference}. 
We start with $T \otimes_R T'$ which is Galois 
with group $G \oplus G$ over $R$. In this extension 
we have $S \otimes_R S$ and this later ring has an idempotent 
$e$ such that $e(S \otimes_R S) \cong S$ and 
$(s \otimes 1 - 1 \otimes s)e = 0$. 
The subgroup $H \subset G \oplus G$ fixing $e$ is the subset 
$\{(g,g') \in G \oplus G | gC = g'C\}$. By \ref{induced}, 
$W = e(T \otimes_R T')$ is $H$ Galois over $R$. 
Checking ranks we have $W \cong T \otimes_S T'$ as $A$ 
modules. 
Inside 
$H$ is the diagonal subgroup $N = \{(g,g)\}$ 
and $D = W^N$. 

Now we use the assumption that $T \cong T'$ as $A$ modules. 
Then $T \otimes_S T' \cong T \otimes_S T$ as $A$ modules, 
and $T \otimes_S T \cong \sum_C T$ where the right hand side 
is spanned over $T$ by the unique idempotents $e_i$ where 
$e_i(T \otimes_S T) \cong T$ and 
$(t \otimes 1 - 1 \otimes \tau^i)e_i = 0$. It is now 
clear that $(g,g)$ fixes the $e_i$ and so $W^N \cong 
\oplus_i R \cong R[C]$ as $R[C]$ modules.
\end{proof}

For completeness sake, we record the converse of the above 
result. 

\begin{lemma} 
Suppose $D \cong R[C]$ as $R[C]$ modules. 
Then $DT$ and $T$ are isomorphic as 
$A = \Delta(S[C]/R[C],\tau)$ modules. 
\end{lemma}

\begin{proof} 
We saw in \ref{moduleproduct} that 
$DT \cong R[C] \otimes_{R[C]} T \cong T$ as $R[C]$ 
modules. Let $\psi: D \cong R[C]$ be a choice of module 
isomorphism. Then the module isomorphism above 
is given by $\phi(\psi \otimes t) = \phi(\psi)(t)$. 
The proof of \ref{moduleproduct} 
relies on the fact that $DT = (D \otimes_R T)^{\gamma}$ 
where $\gamma = \tau \otimes \sigma^{-p^r}$. 
It follows that $S$ and $\sigma$ act on $DT$ by acting 
on $T$. Moreover, $S$ and $R[C]$ commute and so we have that 
$\phi$ is an $A$ module isomorphism. 
\end{proof} 

The above results are partly just explanatory. 
In a future section, 
we will have two extensions of a $G/C$ Galois $S/R$ 
and we will make them equal by modifying by $C$ Galois 
$D/R$ which have $D \cong R[C]$ as $R[C]$ modules. 
Because of the above results 
we will know that $D \cong R[C]$ must be true. 

The above machinery becomes more useful to us when the 
ground ring $R$ is special. In fact, we will be able to 
focus on a specific ground ring which is a localization 
of $\Z[\rho][x]$. But first, we turn to the case where 
$\eta$ is invertible in $R$ and show how to construct cyclic 
extensions in that case.

\section{Inverting \protect{\boldmath{$\eta$}}} 

In this section, $R$ will always be a flat algebra over 
$\Z[\mu]' = \Z[\mu](1/\eta)$ or $\Z[i]' = \Z[i](1/\eta)$. 
That is, we will confine ourselves to case A and B, 
with the additional assumption that $\eta$ is invertible. 
Our goal will be to understand cyclic Galois 
extensions with this simplifying assumption. 
To make parallel arguments easier, we will resubscript 
our rings. Recall that when $p$ is odd (case A) $\mu_m$ was defined 
to be the root of unity with $\mu_m^m = \mu$. That is, 
$\mu_m$ is a primitive $p^{m+1}$ root of unity. 
On the other hand, when $p = 2$ (case B), $\mu_m^m = i$ and so 
$\mu_m$ is a $2^{m+2}$ root of unity. In this 
section we define $V_m = R \otimes_{\Z[\mu]} \Z[\mu_{m-1}]$ 
for $p > 2$ and $V_m = R \otimes_{\Z[i]} \Z[\mu_{m-2}]$ 
when $p = 2$. The point is that for any $p$ $V_m$ is generated 
by a primitive $p^m$ root of one. 
 
Let $\Q(\mu_{m-1})/\Q(\mu)$ have Galois group $C$ in case A but 
let $C$ be the Galois group of $\Q(\mu_{m-2})/\Q(i)$ in case B.  
Thus $C$ is generated by $\tau$ where $\tau(\mu_s) = \mu_s^r$ 
where $r = p+1$ in case A and $r = 5$ in case B. 
Let $t$ be the order of $r$ and so the order $\tau$ in both cases. 
That is, $t = p^{m-1}$ for $p > 2$ 
but $t = 2^{m-2}$ when $p = 2$. 

In \cite{S1981} we defined 
$$M_{\tau}(z) = z^{r^{t - 1}}\tau(z)^{r^{t - 2}}\ldots\tau^{t - 1}(z)$$ 
so $\tau(M_{\tau}(z)) = (M_{\tau}(z))^r/z^{r^{t} - 1}$. 
Here $r^{t} - 1 = kp^{m}$ and   
$k$ is prime to $p$. 

Let $R$ and $V_m$ be as above. Note that $\tau$ 
defines an action on $V_m$ since it acts on $\Z[\mu_{m-1}](1/\eta)$ 
or $\Z[\mu_{m-2}](1/\eta)$ in cases A or B respectively. 
It is clear that $V^{\tau} = R$. 
Let $z \in V^*$ and set $T' = V(M_{\tau}(z)^{1/p^m}) = 
V[Z]/(Z^{p^m} - M_{\tau}(z))$. 
If $\theta = M_{\tau}(z)^{1/p^{m}}$ is the image of $Z$, define 
$\tau(\theta) = \theta^r/z^k$ and note that 
$\tau^{t}(\theta) = \theta^{r^t}/M_{\tau}(z^k) = 
\theta(\theta^{k^{p^{m}}}/M_{\tau}(z)^k) = \theta$. 
It follows that $\tau$ has order $t$ on $T'$. 
Let $\delta$ be the generator of the Galois group 
of $T'/V$ so that $\delta(\theta) = \mu_{m-1}\theta$ or 
$\mu_{m-2}\theta$ respectively. 

\begin{lemma}\label{taualone}
$T'/R$ is Galois with group $<\delta> \oplus <\tau>$ 
where $<\delta>$ is the Galois group of $T'/V$. 
If $T = T'^{\tau}$, then $T/R$ is Galois with group 
$<\delta>$. 
\end{lemma} 

\begin{proof} 
Since $\eta$ is invertible in $R$, $T'/V_m$ is cyclic Galois 
of degree $p^{m}$. When $p$ is odd, 
$\tau(\delta(\theta)) = \tau(\mu_{m-1}\theta) = \mu_{m-1}^r\tau(\theta) = 
\mu_{m-1}^r\theta^r/z^k = \delta(\tau(\theta))$. Thus $\tau$ and 
$\delta$ commute. When $p = 2$ we get the same result for basically the 
same reason. 
Since $\delta$ acts trivially on $V_m \subset T'$, 
it is clear that $\delta$ and $\tau$ generate the group 
$<\delta> \oplus <\tau>$ of automorphisms of $T'$. 
By \ref{galoistower} it suffices to prove that $T'/V$ and $V/R$ 
are Galois. Since we have inverted $\eta$, the second is 
easy and the first is just Kummer theory. 
\end{proof} 

In \cite{S1981} we showed that if $F$ is a field 
and $\rho \in F$ or $i \in F$, then all 
cyclic $L/F$ look like the extension constructed in 
\ref{taualone}. In fact, this holds whenever 
$T \otimes_R V_m/V_m$ is Kummer as the 
as we are about to see.

\begin{proposition}\label{universal}
Suppose $R$ is as above, a faithfully flat $Z[\mu](1/\eta)$ 
or $\Z[i](1/\eta)$ algebra. 
Again set $V_m = R \otimes_{\Z[\mu](1/\eta)} \Z[\mu_{m-1}]$ or 
$R \otimes_{\Z[i](1/\eta)} \Z[\mu_{m-2}]$ 
in cases A and B respectively. 
Assume $T/R$ is a cyclic Galois extension of degree 
$p^m$ with group $G = <\delta>$ and $T' = T \otimes_R V_m$ is Kummer 
over $V_m$. Then there is an $x \in V_m^*$ 
such that $T' = R_m[Z]/(Z^{p^m} - M_{\tau}(x))$ and 
$T = T'^{\tau}$ as in $\ref{taualone}$. 
\end{proposition}

\begin{proof} 
This proof is really the same as the proof in \cite{S1981}, so we just quickly 
go through it when $p$ is odd to demonstrate it works in this more general 
setting. By assumption, there is a $\theta \in T'^*$ 
such that $\delta(\theta) = \mu_{m-1}\theta$. Clearly the actions 
of $\tau$ and $\delta$ commute, so $\delta(\tau(\theta)) = 
\tau(\rho_{m-1}\theta) = \rho_{m-1}^r\tau(\theta)$. Thus $\tau(\theta) = 
\theta^ry$ for $y \in V_m^*$. If $t$ is the order of $\tau$, 
we compute that $\theta = \tau^t(\theta) = 
\theta^{r^t}y^{r^{t-1}}\tau(y)^{r^{t-2}}\ldots\tau^{t-1}(y)$. 
Now $r^t - 1 = p^mk$ where $k$ is prime to $p$, so there is a 
$k'$ with $kk' = 1 + sp^m$. If $x = \theta^{p^m}$ we have   
$x^k = M_{\tau}(y)^{-1}$ or $x = M_{\tau}(y^{-k})x^{-sp^m}$ 
and modifying $x$ by a $p^m$ power we are done.
\end{proof} 

There is a version of the above construction in Case C, where 
the cyclotomic Galois group is not cyclic, but it only defines 
some of the cyclic extensions and we do not need it. 

It turns out that our $M_{\tau}(z)$  
construction on cyclics are related to the 
corestriction. 

\begin{theorem}
The cyclic constructed in \ref{taualone} 
is also $\Cor_{V_m/R}(S/V_m)$ where $S = V_m[Z]/(Z^{p^m} - z^s)$ 
for $s$ prime to $p$. 
\end{theorem} 

\begin{proof}
With $S$ as given let $\theta \in S^*$ be the 
image of $Z$. We follow the construction of the corestriction 
from \ref{corestriction}. Let $G$  be the Galois group of $S/V_m$ 
and $\delta \in G$ such that $\delta(\theta) = \mu_{m-1}\theta$ or 
$\mu_{m-2}\theta$ in cases A and B respectively. We stick to case 
A below as B is pretty much identical. 
To construct the corestriction of $S$ we first consider 
the twist $g(S)$ for any $g \in C$. Pretty clearly, as an algebra,  
$g(S) = V_m[Z']/(Z'^{p^m} - g(z))$ and we can set $g(\theta)$ 
to be the image of $Z'$. The canonical generator of 
the Galois group of $g(S)/V_m$ is $\delta_g$ where 
$\delta_gg = g\delta$, so $\delta_{\tau}(\tau(\theta)) = 
\tau(\delta_{\tau}(\theta)) = \tau(\mu_{m-1}\theta) = 
\mu_{m-1}^r\tau(\theta)$. 

Now the prescription for defining the corestriction requires 
we form $W = \otimes_{g \in C} g(S)$ which has Galois group 
$(G \oplus \cdots \oplus G) \rtimes C$. Note that we are identifying 
the Galois group of $g(S)/V_m$ with $G$ by identifying $\delta$ 
and $\delta_g$, so that the action of $C$ is just the shift. 
Next we form $W^N$, where $N$ is the kernel of the product map 
$G \oplus \cdots \oplus G \to G$. Note that 
$N$ is generated by all $\delta\delta_g^{-1}$ for all $g \in C$ 
and is preserved by $\tau$. 

Consider $\Theta = M_{\tau}(\theta)$. Note that 
as elements of $G \oplus \cdots \oplus G$, $\delta$ only 
acts on $\theta$ and not on any other $C$ conjugates, and 
$\delta_g$ only acts on $g(\theta)$ and not on any other 
$C$ conjugates. Since $\theta$ appears in $M_{\tau}(\theta)$ 
as $\theta^{r^{t-1}}$, it follows that $\delta(\Theta) = \rho_{m-1}^s\Theta$ 
where $s = r^{t-1}$. Since $\tau(\theta)$ appears in 
$M_{\tau}(\theta))$ as $\tau(\theta)^{r^{t-2}}$, 
$\delta_{\tau}(\Theta) = (\rho_{m-1}^r)^{r^{t-2}}\Theta = \rho_{m-1}^s\Theta$.  
It follows that $\delta\delta_{\tau}^{-1}(\Theta) = \Theta$.  Arguing similarly 
all of $N$ fixes $\Theta$. Since 
$s$ is prime to $p$ this is a Kummer element generating 
$W^N$ and we are done.
\end{proof}

There is an easy consequence of the above theorem 
when we combine it with \ref{functorial}. 
The proof is merely the observation that the corestriction 
of a split extension is split. 

\begin{corollary}\label{splitprime}
Let $M \subset R$ be an ideal. 
In $\ref{taualone}$, if $z$ maps to $1 \in R_m/NR_m$ 
then $T/R$ constructed has the property that 
$T/NT$ is split over $R/N$. 
\end{corollary}

The above description of cyclics allows us to say something 
about when degree $p$ extensions can be enlarged to degree 
$p^m$ extensions. That is, we have another version of Albert's 
criterion only valid when $\eta$ is invertible. 
Let $N_{\tau}(z) = 
\prod_{g \in C} g(z)$ when $C = <\tau>$. Recall 
that $R$ is a $\Z[\mu]$ algebra. 

\begin{theorem}\label{invert}
Suppose $T/R$ is Galois with cyclic group of order $p^{m}$. 
If $T/S/R$ is such that $S/R$ is Galois of degree $p$, 
then $S = R(N_{\tau}(z)^{1/p})$ for some $z \in T^*$.  
Conversely, if $p$ is odd or $R$ contains $i$, then 
$S = R(N_{\tau}(z)^{1/p})$ extends to $T/S/R$ cyclic of degree 
$p^m$ constructed via \ref{taualone}.   
\end{theorem}

\begin{proof}
By \ref{universal}, $T/R$ is constructed 
as in \ref{taualone}. 
All powers of $r$ are congruent to 1 modulo $p$, 
so $M_{\tau}(z) = N_{\tau}(z)w^p$ for some $w$. 
Let $\beta = N_{\tau}(z)^{1/p}$. In this argument $\rho$ 
will be a primitive $p$ root of one so $\rho = -1$ when 
$p = 2$.  
We have $\delta(\beta) = \rho\beta$. Since $N_{\tau}(z)$ 
is $\tau$ fixed, we have $\tau(\beta) = \rho^e\beta$ for some $e$. Our goal is to show that $\rho^e = 1$. 
We defined $\tau$ 
by setting 
$$\tau(M_{\tau}(z)^{1/p^{m}}) = M_{\tau}(z)^{r/p^{m}}/z^k$$ 
where $r^{t} = 1 + kp^{m}$. Take the 
$p^{m-1}$ power of this equation to get 
$$\tau(M_{\tau}(z)^{1/p}) = 
\tau(N_{\tau}(z)^{1/p}w(z)) = \rho^eN_{\tau}(z)^{1/p}\tau(w(z)) = $$
$$M_{\tau}(z)^{r/p}/z^{kp^{m-1}} = (N_{\tau}(z)w(z)^p)^{(1+p)/p}z^{-kp^m} = $$
$$N_{\tau}(z)^{1/p}N_{\tau}(z)w(z)^{1+p}z^{-kp^{m-1}}$$ 
and cancelling 
and isolating we have 
$$\rho^e = N_{\tau}(z)w(z)^{1+p}z^{-kp^{m-1}}\tau(w(z))^{-1}.$$ 
Since $z$ is arbitrary the right side is really an expression in the multiplicative 
free abelian group with basis the $\tau$ conjugates of $z$, 
and so the right side must then be the trivial element of the free group. We have $\rho^e = 1$. 
 
To prove the converses we simple observe that if 
$S = R(N_{\tau}(z)^{1/p})$ 
then we use $z$ to construct $T/S/R$ as in \ref{taualone}. 
\end{proof}

Now we turn to applying the above machinery in a special case. 

\section{Special R} 

In the next section we will also be restricted to cases 
A and B. That is, $R$ will be an algebra (in this case 
a specific algebra) over $\Z[\mu]$ where $\mu$ is $i$, the square root of -1,  
in case B and the primitive $p$ root of one in case A when 
$p > 2$. Recall that $\eta = \mu - 1$ is a prime element 
of $\Z[\mu]$  
totally ramified over $p$. On the other hand we will be applying 
the generalized Albert criterion from \ref{albert2} and in this case 
we need $\rho = \mu$ when $p > 2$ but $\rho = -1$ 
when $p = 2$. That is, $\rho$ is always a primitive $p$ 
root of one. 
The specific class of $R$'s we will consider are  
$R = \Z[\mu][x]_Q$ where $Q$ is a multiplicatively 
closed subset of $1 + \eta{M}$ and $M$ is the ideal generated 
by $x$. Note that this $R$ satisfies the assumptions 
of the first section and so the results 
there apply. 
Let $\bar R = R/M$ and $\hat R = R/\eta{R}$. 
The next five results will describe some consequences of this choice.

\begin{lemma}\label{details}
$\hat R = F_p[x]$. 

$\bar R = \Z[\mu]$ 

$\Pic(\hat R) = 0$

$\Pic(R) = \Pic(\Z[\mu])$ 

$R^* \to (\hat R)^*$ is surjective. 
\end{lemma}

\begin{proof} 
This is all obvious for $\Z[\mu][x]$. 
Since $Q \subset 1 + \eta{M}$, the localization does not 
change the quotients by $M$ and $\eta$. 
$\Pic(\Z[\mu]) \to \Pic(\Z[\mu][x]) \to \Pic(R)$ is 
a composition of surjections   
and following by the quotient to $\Z[\mu] = \bar R$ shows the 
natural map is also injective. The last fact relies on the 
fact that $\Z[\mu]^* \to F_p^*$ is surjective \ref{total}.
\end{proof} 

The surjectivity of $R^* \to (\hat R)^*$ 
shows: 

\begin{proposition}\label{specialdegreep}
If $C$ is cyclic of degree $p$,
$\Pic(R[C](1)) \cong \Pic(R) \oplus \Pic(R)$. 
If $D/R$ is $C$ Galois, and $D/MD$ is split over $R/M$, 
then $D \cong R[C]$ as $R[C]$ modules. 
\end{proposition}

\begin{proof}
From \ref{fiber} there is an exact sequence 
$R^* \to (\hat R)^* \to \Pic(R[C](1)) \to \Pic(R) \oplus \Pic(R) 
\to \Pic(\hat R)$ and so the first statement is clear. 

To prove the last statement, we saw in \ref{degreepnormal} 
that it suffices to show $D(1) \cong R[C](1)$ or by the first statement, 
$D_i \cong R$ for $i = 0,1$. 
$D_0 \cong R$ is automatic. 
$D_1$ is the image of some $I \in \Pic(\Z[\mu])$ 
and $D_1/MD_1 \cong I$. 
Since $(D/M)(1) \cong \bar R[C](1)$, $D_1/MD_1 \cong \Z[\mu]$. 
\end{proof} 

We need to relate the properties of $R$ and $\bar R$. 
Almost identical arguments show: 

\begin{lemma}\label{picbar}
$\Pic(\bar R[C](1)) \cong \Pic(\Z[\mu]) \oplus \Pic(\Z[\mu])$. 
The natural map $\Pic(R[C](1)) \to \Pic(\bar R[C](1))$ 
is an isomorphism. 
\end{lemma} 

Similarly, 

\begin{lemma}
$\Pic(R *_p R) \to \Pic(R) \oplus \Pic(R)$ 
is surjective, as is $\Pic(\Z[\mu] *_p \Z[\mu]) 
\to \Pic(\Z[\mu]) \oplus \Pic(\Z[\mu])$. 
\end{lemma} 

\begin{proof} 
In the first case, the cokernel is a subgroup 
of $\Pic(R/\eta^pR) \cong \Pic(R/\eta{R}) = 0$. 
The next statement has a parallel proof. 
\end{proof} 

Finally we need some results relating the 
Brauer groups of $R$ and $R'$ etc. 

\begin{lemma}\label{brauer}
$\Br(R[C](1)) \to \Br(R) \oplus \Br(R)$ is injective. 
$\Br(R) \to \Br(R')$ is injective. 
$\Br(R[C](1)) \to \Br(R'[C](1))$ is injective. 
\end{lemma}

\begin{proof} The first statement follows from 
\ref{fiber} and the fact that $\Pic(R/\hat R) = 0$. 
Since $R$ is regular, $\Br(R) \to \Br(F)$ is 
injective where $F$ is the field of fractions 
of $R$ and $R'$. The second statement is clear. 
Since $R'[C](1) = R' \oplus R'$, the third statement 
is obvious. 
\end{proof}

The structure of the arguments in the rest of this section 
will be as follows. Our ultimate goal is to construct, 
for any $p^m$, a cyclic $T/R$ of degree $p^m$ 
so that $\hat T = T/\eta{T}$ is a very general extension of 
$\hat R = F_p[x]$. In fact, all that matters 
is that a small degree piece of $\hat T/\hat R$ 
is very general. Let $T/U/R$ 
be such that $U/R$ has degree $p$. 
Then we require that $U/\eta{U} = \hat U = \hat R[Z]/(Z^p - Z - x)$. 

Once this first $U$ is constructed, we will construct $T$ 
one cyclic at a time. That is, we will assume $S/R$ has been 
constructed of degree $p^{r-1}$ (with degree $p$ part as needed) 
and solve the embedding problem 
to construct $T$ of degree $p^r$. In fact it is convenient 
to make the argument in the opposite order. 
We first assume $U/R$ has been constructed, and turn 
to constructing $U$ at the end of the section. 
Given this assumption, we will assume 
by induction that $S/R$ has been constructed of degree $p^{r-1}$ 
and Galois group $G/C$ and proceed to construct the $G$ 
Galois $T/R$ of degree $p^r$. 
Of course, we will use \ref{albert2} to solve this 
embedding problem. Recall that the criterion in \ref{albert2} 
has Brauer group, Picard group, and unit components. 
More specifically, let us describe these obstructions: 
\begin{enumerate}[I.]
\item $A(1) = \Delta(S[C](1)/R[C](1),G,\tilde \rho)$ needs to split. \label{itm:I}

\item For some $P(1)$, a normalized splitting module for $A(1)$, 
we need $(P_1^p)^{G/C} \cong R$. \label{itm:II}

\item Given I) and II), there is a unit $\tilde u \in (R/\eta{R})^*$ 
that needs to be in the image of $R^*$. \label{itm:III}
\end{enumerate}

Starting with \ref{itm:I}, \ref{brauer} shows that it suffices 
to check that $A(1)$ is split by $R'[C](1) = R' \oplus R'$. 
Since $A_0 \cong \End_R(S)$ is automatically split, 
it suffices to check that $A_1$ is split by $R'$. 
We will achieve this by prearranging that if $S' = S \otimes_R R'$, 
then $S'/R'$ extends to a bigger cyclic. 

Assuming that $A(1)$ is split, we have a splitting module and 
hence a normalized splitting module $P(1)$. 
To deal with \ref{itm:II}, suppose 
$S/MS$ is a split extension 
of $R/M = \Z[\mu]$. Thus 
$P_1/MP_1 \cong I_1 \oplus \cdots \oplus I_{p^{r-1}}$ 
where $I_i \in \Pic(\Z[\mu])$. Moreover, 
$P_1$ and hence $P/MP_1$ has an action by $\sigma$ such that 
$\sigma^{p^{r-1}} = \rho$ and $\sigma$ cyclically permutes the 
direct summands of $S/M$. It follows that we can take all the 
$I_i = I$ and $\sigma$ acts by $\sigma(a_1,\ldots,a_{p^{r-1}}) = 
(\rho{a_{p^{r-1}}},a_1,a_2,\ldots,a_{p^{r-1}-1})$. 

Since $C$ acts trivially on $P_1^p$, we have $P_1^p \cong J \otimes_R S$ 
where $J \in \Pic(R) = \Pic(\Z[\mu])$. Tensoring by $\bar R$ 
we have $J \cong I^p$ in $\Pic(\Z[\mu])$. We can also view 
$I$ as an element of $\Pic(R)$. 

\begin{proposition}
Suppose $A(1)$ is split and $\bar S$ is a split extension 
of $\bar R$. Then there is a choice of normalized splitting 
module $P(1)$ such that $(P_1^p)^{G/C} \cong R$. 
\end{proposition}

\begin{proof}
Since $\Pic(Z[\mu]/\eta{\Z[\mu]}) = \Pic(F_p) = 0$, we know $I/\eta{I} 
\cong \Z[\mu]/\eta{Z[\mu]}$. Thus the fiber diagram 
$$\begin{matrix}
I(1)&\longrightarrow&R\cr
\downarrow&&\downarrow\cr
I&\longrightarrow&R/\eta{R}\cr
\end{matrix}$$ 
defines an element $I(1) \in \Pic(\Z[\mu][C](1))$ 
which we can extend to $I(1) \in \Pic(R[C](1))$. 
If $P(1)' = P(1) \otimes_{R[C](1)} I(1)^{\circ}$, 
then $P(1)'$ is our desired normalized splitting module. 
\end{proof} 

Thus assuming $\bar S$ split allows us to handle II). 

We have shown that we can remove two of the three obstructions 
to extending $S/R$ in certain circumstances. We finally must deal 
with the unit $\bar u \in (R/\eta^pR)^*$. 
Here we need to remove the obstruction by localizing. 
That is, we invert a preimage of $\bar u$. It seems 
difficult to derive detailed information about the 
$\bar u$ that arise, so we invert everything relevant. 
That is, from now on $R$ is {\bf very special}, meaning 
that is has the form 
$\Z[\mu][x]_Q$ where $Q$ is the full 
multiplicatively closed set $1 + \eta{x}\Z[\mu][x]$. 
Let $M = xR$ as usual. To deal with III), we assemble 
some properties of these very special $R$. 

\begin{lemma}\label{unitsspecial}
Suppose $R$ is very special, as above. Then, 
\begin{enumerate}
\item $1 + \eta{M} \subset R^*$. \label{units:1}

\item Every element of $R^*$ has the form $u(1 + ax\eta)/(1 + bx\eta)$ 
where $a,b \in \Z[\mu][x]$ and $u \in \Z[\mu]^*$. \label{units:2}

\item The cokernel of $R^* \to (R/\eta^p(R))^*$ is the 
cokernel of 
$\Z[\mu]^* \to (\Z[\mu]/\eta^p\Z[\mu])^*$. \label{units:3}
\end{enumerate}
\end{lemma}

\begin{proof} 
Any element of the form
$u = 1 + \eta{f}$, for $f \in M$, has the form 
$1 + \eta{f'/s}$ where $s \in Q$ and $f' \in x\Z[\mu][x]$. 
This equals $(1/s)(s + \eta{f'})$ 
and $s = 1 + \eta{g}$ for $g \in x\Z[\mu][x]$ so 
$u = (1/s)(1 + \eta({f'+g})) \in R^*$. This shows 
\ref{units:1}. 

As for \ref{units:2}, let $u \in R^*$ map to $\hat u \in F_p[x]^* = F_p^*$ 
and $\bar u \in \Z[\mu]^*$. Clearly $\bar u$ maps to 
$\hat u \in F_p^*$. Thus $\bar u^{-1}u$ maps to $1$ in 
$\bar R$ and $\hat R$. Since $\eta{R}$ and $x{R}$ are distinct prime 
ideals, $xR \cap \eta{R} = x\eta{R}$ and so $\bar u^{-1}u \in 1 + x\eta{R}$. 
Thus we can assume $u \in 1 + x\eta{b}$ for $b \in R$.  
Write $b = f/z$ where $z = 1 + x\eta{c}$, $f,c \in \Z[\mu][x]$ 
and just as above 
we have $u = (1 + x\eta(c + f))/(1 + x\eta{c})$.

To prove \ref{units:3} we first prove the cokernel of $\Z[\mu]^* \to (\Z[\mu]/\eta^p\Z[\mu])^*$ 
maps onto the cokernel of $R^* \to (R/\eta^p(R))^*$. 
Suppose $z \in (R/\eta^p{R})^*$. Then the image of 
$z$ in $R/\eta{R}$ lies in $F_p^*$. Choose $u \in \Z[\mu]^*$ 
mapping onto this element of $F_p^*$, and reduce to the 
case that $z = 1 + \eta{y}$ for $y \in R$. Write 
$y = \bar y + f$ for $\bar y \in \Z[\mu]$ and 
$f \in M_Q$. Thus modulo $\eta^2$ we have $z \equiv (1 + \bar y\eta)(1 + f\eta)$ 
and the second factor is in $R^*$. Thus we have reduced to the case 
$z = 1 + y\eta^2$ and we proceed by induction. 
This proves the surjectivity. 
If $1 + \eta{c}$ is in the image of $R^*$, we can go modulo 
$M$ and conclude that it is in the image of $\Z[\mu]^*$.
\end{proof} 

From \ref{fiber} we have exact sequences 
$\Z[\mu]^* \to (\Z[\mu]/\eta^p\Z[\mu])^* \to \Pic(\Z[\mu] *_p \Z[\mu]) \to 
\Pic(\Z[\mu]) \oplus \Pic(\Z[\mu])$ and similarly 
for $R *_p R$. 
We now have from \ref{details} that : 

\begin{corollary}\label{veryspecial}
With $R$ very special as above, 
the natural map $\Pic(\Z[\mu] *_p \Z[\mu]) 
\to \Pic(R *_p R)$ is an isomorphism. 
With $\bar R = R/M = \Z[\mu]$, the natural map 
$\Pic(R *_p R) \to \Pic(\bar R *_p \bar R)$ 
is an isomorphism. 
\end{corollary} 

The above \ref{veryspecial} allows us to deduce 
properties of splitting modules over $R$ 
from properties of splitting modules over $\bar R$. 

Our ultimate argument will involve splitting 
modules over $R$, $\bar R$, and $R'$. For this 
reason we return to investigating the relationship 
between constructions and functors over $R$ and 
$R' = R(1/\eta)$. We begin with a very general fact.  

\begin{proposition}
Suppose $V$ is a regular domain and $\eta \in V$ 
is a prime element. Set $V = V(1/\eta)$. Then the natural map 
$\Pic(V) \to \Pic(V')$ is an isomorphism. 
\end{proposition}

\begin{proof} 
The first statement follows immediately from the well known fact, 
in \cite[p.~167]{M} for example, that for $V$ the Picard group 
is the class group $C(V)$, and since $\eta{V}$ is principal, 
the surjection $C(V) \to C(V')$ has trivial kernel. 
\end{proof} 

We will need a description of the units of $R'$. 

\begin{lemma}\label{primeunits}
Every unit of $R'$ is of the form $u(1 + ax\eta)/(1 + bx\eta)$ 
where $a,b \in \Z[\mu][x]$ and $u \in (\Z[\mu](1/\eta))^*$. 
\end{lemma} 

\begin{proof}
By \ref{unitsspecial} the units of $R$ have this 
form where $u \in (\Z[\mu])^*$. Since $\eta$ is a prime element 
of $R$, every unit of $R'$ has the form $\eta^mv$ 
where $v$ is a unit of $R$. 
\end{proof}

As another corollary, we have: 

\begin{corollary}\label{picprime}
The natural map $\Pic(R *_p R) \to \Pic(R' *_p R')$ 
has kernel $K$ which is the cokernel of $\Z[\mu]^* 
\to (\Z[\mu]/\eta^p\Z[\mu])^*$. 
The natural map $\Pic(R[C](1)) \to \Pic(R'[C](1))$ 
is an isomorphism. 
\end{corollary}

\begin{proof} 
First of all, $R' = R(1/\eta) = \Z[\mu]'[x]_Q$ 
where $\Z[\mu]' = \Z[\mu](1/\eta)$. Thus $\Pic(R') = 
\Pic(\Z[\mu]') = \Pic(\Z[\mu]) = \Pic(R)$. 
The kernel we want is the kernel of 
$\Pic(\Z[\mu] *_p \Z[\mu]) \to \Pic(\Z[\mu]) \oplus \Pic(\Z[\mu)$ 
which is exactly $K$ described above. 
The second statement is obvious as $\Pic(R[C](1)) \cong 
\Pic(R) \oplus \Pic(R)$. 
\end{proof}

Combining the above observations with the assumption that 
$R$ is very special does, in fact,  
allow us to deal with the obstruction III). 
In detail, assume $A(1)$ and hence $A(1)' = A(1) \otimes_R R'$ and 
$\bar A(1) = A(1)/MA(1)$ 
are split. Let ${\cal S}$, ${\cal S}'$ and $\bar{\cal S}$ be the set of 
isomorphism classes of splitting modules of 
$A(1)$, $A(1)'$ and $\bar A(1)$ respectively. 
Inside these sets we define ${\cal N}$, ${\cal N}'$, 
and $\bar {\cal N}$ to be the subset of ${\cal S}$, 
${\cal S}'$ and $\bar {\cal S}$ respectively 
of normalized splitting modules. Note that with 
\ref{shortnormalized} in mind, $P(1)' \in {\cal S}'$ is 
normalized if and only if $P_0' = P(1)'/(\tau - 1)P(1)'$ 
satisfies $(P_0')^{G/C} \cong R'$ and $\bar P(1) \in 
\bar {\cal S}$ is in $\bar {\cal N}$ if and only if 
$\bar P(1)_0 = \bar P(1)/(\tau - 1)\bar P(1)$ satisfies 
$(\bar P(1)_0)^{G/C} \cong \bar R = \Z[\mu]$. 

With \ref{pictrivial} in mind, we define 
${\cal R} \subset {\cal N}$ to be the subset of normalized 
splitting modules $P(1)$ such that $(P_1^p)^{G/C} 
\cong R$ and similarly for ${\cal R}' \subset {\cal N}'$ 
and $\bar {\cal R} \subset \bar {\cal N}$. 
Finally we define ${\cal Q} \subset {\cal R}$ 
to be the $P(1)$ with $P(p)^G \cong R *_p R$
and $\bar {\cal Q} \subset \bar {\cal R}$, ${\cal Q}' \subset 
{\cal R}'$ similarly. 

\begin{theorem}\label{nonemptymodM}
Let $R$ be very special and $S$, $A(1)$ as above. 
Assume $A(1)$ is split and $S/MS$ is split 
over $R/M$. Then $P(1) \to \bar P(1) = 
P(1)/MP(1)$ defines a one to one correspondence 
between ${\cal Q}$ and $\bar {\cal Q}$. 
In particular, ${\cal Q}$ is non empty.
\end{theorem} 

Note that \ref{nonemptymodM} (which is as yet unproven) says 
that an extension 
$T/S/R$ exists but not that $T/MT$ is split. This we handle 
next. 

\begin{corollary}\label{Texists}
In the situation of \ref{nonemptymodM}, there is a 
$T$ such that $T/MT$ is split. 
\end{corollary}

\begin{proof}
Since $S/MS$ is split, $T/MT \cong \Ind_H^G(\bar D/\bar R)$ 
where $\bar D/\bar R$ is cyclic of degree $p$. 
Now we can view $\bar R = \Z[\mu] \subset R$ and so 
we can set $D = \bar D \otimes_{\bar R} R$. Now replace 
$T$ with $TD^{-1}$ which is split modulo $M$. 
\end{proof}

The astute reader will notice that we have not claimed 
that ${\cal Q} \to {\cal Q}'$ is bijective and that 
is because it is not bijective. Transferring to a question 
about Picard groups, the kernel of $\Pic(R *_p R) \to 
\Pic(R' *_p R') = \Pic(R' \oplus R')$ is the 
kernel of $\Pic(R *_p R) \to \Pic(R) \oplus \Pic(R)$ 
which is the cokernel of $R^* \to (R/\eta^p)^*$ 
which is nontrivial. We will solve this issue by only 
looking at extensions which are split modulo $M$. 
We return to the proof of \ref{nonemptymodM}. 

\begin{proof} 
By \ref{splittingmodule}, 
\ref{picprime}, and \ref{picbar} 
the natural maps ${\cal S} \to {\cal S}'$ and 
${\cal S} \to \bar {\cal S}$ are bijections. 
Since $\Pic(R) \to \Pic(R')$ and $\Pic(R) \to 
\Pic(\bar R)$ are injections, this implies that 
the induced maps ${\cal N} \to {\cal N}'$, 
${\cal R} \to {\cal R}'$, ${\cal N} \to \bar {\cal N}$ 
and ${\cal R} \to \bar {\cal R}$ are all bijections.  
Since $\Pic(R *_p R) \cong 
\Pic(R/M *_p R/M)$, ${\cal Q}\to \bar {\cal Q}$ is also a 
bijection. 

As for the last statement, since $(S/M)/(R/M)$ is split 
it extends to a split cyclic of degree $p^{s+1}$ 
and so $\bar {\cal Q}$ is non empty.
\end{proof} 

We should remind the reader that, by \ref{albert1}, 
different choices of $P(1) \in {\cal Q}$ yield 
different solutions $T/S/R$ because $T(1) \cong 
P(1)$. Moreover, even if we fix $P(1)$, different 
specializations of our versal extension yield non-isomorphic 
solutions. However, in our situation, we next show there is only one 
possible $P(1)$. 

Assume  
$\bar T = T/MT$ is split over $\bar R = R/M = \Z[\mu]$. 
Of course, for such a $T$ we must have $\bar T \cong 
\bar S[C]$ and so $\bar T(1) \cong \bar S[C](1)$. 
Since $\bar S$ is split over $\bar R$, 
the map $\bar R[C](1) \to \bar S[C](1)$ is injective 
and there is a unique $\bar P(1)$ which yields the 
split $\bar T$. 

\begin{corollary}\label{unique} 
Suppose $S/M$ is split as above. 
There is one and only one $P(1) \in {\cal Q}$ 
which yields a $T/S/R$ where $T/MT$ is split. 
\end{corollary}

Of course, by (\ref{albert1}), $P(1)$ yields other 
$T/S/R$ some of which may not split modulo $M$. 

Demanding that $T$ splits fixes $P(1)$. 
The situation over $R'$ is similar, as follows. 
Consider ${\cal Q}'$ defined above. 
Just as for ${\cal Q}$ we have: 

\begin{proposition} 
Suppose $A(1)$ above is split. 
There is a one and only one $P(1)' \in {\cal Q}'$ which 
yields a $T'/S'/R'$ with $T'/MT'$ split. 
\end{proposition}

\begin{proof} 
That $T'$ exists follows by almost the same argument as we used in 
\ref{Texists}. 

As for the uniqueness of $P(1)'$, of course $R' *_p R' = R' \oplus R'$. 
Let $\bar R' = R'/MR' = \Z[\mu]'$. 
$\Pic(R' \oplus R') = \Pic(R' *_p R') \cong 
\Pic(\bar R' *_p \bar R') = \Pic(\bar R' \oplus \bar R')$ and 
$\Pic(R' \oplus R') = \Pic(R'[C](1)) \cong \Pic(\bar R'[C](1)) 
= \Pic(\bar R' \oplus \bar R')$.  Finally, $\Pic(R'[C](1)) \to 
\Pic(S'[C](1))$ is injective since $\Pic(R'[C](1)) = \Pic(Z[\mu]') \oplus 
\Pic(\Z[\mu]')$ and $S'[C](1)/MS'[C](1)$ is a split extension 
of $\bar R'[C](1)$. 
\end{proof} 

In summary, 
we have: 

\begin{theorem}
Suppose $S/R$ is cyclic Galois of degree $p^{r-1}$ 
and $S' = S \otimes_R R'$ extends to a degree $p^r$ 
cyclic extension $T'/R'$ with $T'/MT'$ split over 
$R'/MR'$. Further assume 
$S/MS$ is the split extension of $\bar R$. 
Then $S/R$ extends to a cyclic degree 
$p^r$ extension $T/R$ with $T/MT$ split over $\bar R$. 
 
\end{theorem}

\begin{proof}
Since $A(1)' = A(1) \otimes_R R'$ is an obstruction 
to extending $S'$, it is split and so $A(1)$ 
is split and by \ref{nonemptymodM} there is a unique $P(1) \in {\cal Q}$ 
yielding $T$ with $T/MT$ split. Since $T'$ is split modulo $M$ it defines 
a $P(1)' \in {\cal Q}'$ and all the uniqueness implies 
$T'(1) \cong T(1) \otimes_R R'$ as $S'[C]$ modules. 
\end{proof}
 
The above theorem is only the first step in our induction argument. 
By that theorem, we have that $T$ exists, $T/MT$ is split, 
and $T(1) \otimes_R R' \cong T'(1)$ as $S'[C]$ modules. However, 
$T \otimes_R R'$ may not be isomorphic to 
$T'$ as Galois extensions. Of course, since they share $S'$, 
$T \otimes_R R' = D'T'$ where $D'/R'$ is cyclic of degree $p$. By \ref{extenddifference}, 
$D' \cong R'[C]$ as $R'[C]$ modules. 

To remove the above difficulty, we start by revising our set up a bit. We call this our induction assumption: 

\bigskip

INDUCTION ASSUMPTION: 

Let $S/R$ be degree $p^{r-1}$ cyclic, and $T/S/R$ cyclic of degree $p^{r}$ 
as constructed above. Set $S' = S \otimes_R R'$ and assume $S'/R'$ 
extends to a bigger cyclic extension $T''/S'/R'$ of degree $p^m$ 
where $m \geq r$. We define $T''/T'/S'/R'$ 
such that $T'/S'/R'$ has degree $p^{r}$
\bigskip

The point about $p^m$ is that this is the 
ultimate size of the extension of $R$ we are constructing after 
multiple degree $p$ steps. 

\begin{proposition}\label{inductionstep}
We can modify $T/S/R$ and $T''/T'/S'/R'$ 
so that the induction assumption still holds 
but $T(1/\eta) = T'$. 
\end{proposition}

\begin{proof}
Our proof strategy will be the following. 
We can modify $T/S$ by any degree $p$ Galois 
$D/R$ which splits modulo $M$ because by \ref{topproduct} 
$TD/R$ is part of the tower $T(D \otimes_R S)/S/R$. 
We can modify $T'/S'$ by any $D'/R'$ of degree $p$ 
which extends to a degree $p^{m-r}$ cyclic Galois that splits modulo 
$M$. This follows because if $V'/D'/R'$ has degree 
$p^{m-r}$ and $U'/T'/S'/R$ has degree $p^m$ 
then by \ref{topproduct} $V'T''/R$ has degree $p^m$ and we have 
the tower 
$T''V'/T'D'/S'/R'$ of degree $p^m$. 

By \ref{difference} there is a unique 
$D'/R'$ such that $T(1/\eta)D' = T'$ as extensions of $R'$. 
By \ref{extenddifference} $D' \cong R'[C]$ as $R'[C]$ 
modules. Thus $D'$ contains a unit $\alpha$ such that 
$\tau(\alpha) = \rho\alpha$ and $v = \alpha^p$ 
is a unit of $R'$. By \ref{primeunits} $v$ 
has the form $u(1 + ax\eta)/(1 + bx\eta)$ 
for $u \in (\Z[\mu])^*$ and $a,b \in \Z[\mu][x]$. After altering 
by a $p$ power we can assume $v = u(1 + ax\eta)$. 
By \ref{splitdifference} $D'/MD'$ is split, so 
$u$ is a $p$ power.  That is, we can write 
$D' = R'((1 + ax\eta^s)^{1/p})$ where $a \in R$ is not divisible 
by $\eta$. 
If $s \geq p$ then we are done, because we can write 
$1 + ax\eta^s = 1 + v\eta^p$ and $D' = D \otimes_R R'$ 
where $D/R$ is cyclic Galois of degree $p$ defined 
by $Z^p + g(Z) - v$. Since $v \in M$, $D/R$ splits modulo $M$. 
We can modify $T$ by $D^{-1}$ to get $T(1/\eta) = T'$.  

To deal with the $s \leq p-1$ cases, 
we require some more information about $1 + ax\eta^s$. 
Assume $s \leq p-1$ and $a$ is not divisible by $\eta$. 
We derive our additional information by using 
completions. Let $R_{\eta}$ be the completion of the 
discrete valuation ring which is $R$  localized at the prime 
$\eta{R}$. Set $T_{\eta}/S_{\eta}/R_{\eta}$ to be 
the cyclic extension induced by tensoring with $R_{\eta}$. 
Of course $T_{\eta}$ and $S_{\eta}$ are direct sums of complete 
discrete valuation rings. In fact, they are themselves complete 
valuation rings. 

\begin{lemma}\label{completions}
Let $T/S/U/R$ be the above cyclic Galois extension 
where $U/R$ has degree $p$. Let $T_{\eta}/S_{\eta}/U_{\eta}/R_{\eta}$ be the result of tensoring by 
$R_{\eta}$. Then $U_{\eta}$ 
is a complete discrete valuation domain, 
as is $T_{\eta}$ and $S_{\eta}$. 
\end{lemma}

\begin{proof} 
Let 
$\hat U = U/\eta{U}$. 
By assumption $\hat U = \hat R[Z]/(Z^p - Z - x) = (F_p[x])[Z]/(Z^p - Z - x)$. 
Thus $U_{\eta}/\eta{U_{\eta}} = F_p(x)[Z]/(Z^p - Z - x)$ which is 
clearly a field. That is, $\eta$ stays prime in $U_{\eta}$ and 
$U_{\eta}$ is a complete discrete valuation domain. In general, 
$T_{\eta}/R_{\eta}$ would have the form $\Ind_H^G(T''/R_{\eta})$ 
for $T''$ a domain 
but since $U_{\eta}$ is a domain, by \ref{induced} we have 
$H = G$ and $T_{\eta}$ is a complete discrete valuation domain.
\end{proof}

The advantage of going to the completion is that we have: 

\begin{proposition}\label{canextend}
$T_{\eta}/R_{\eta}$ extends to a cyclic Galois extension 
$W/T_{\eta}/R_{\eta}$ of degree $p^m$. 
\end{proposition}

\begin{proof}
The standard theory of Galois extensions of discrete valuation 
rings (e.g., \cite[p.~54]{Se}) shows that there is a one to one 
correspondence between Galois extensions of the residue field 
of $R_{\eta}$, namely, $R_{\eta}/\eta{R_{\eta}} = F_p(x)$, and unramified 
Galois extensions of the field of fractions of $R_{\eta}$, 
which is the same as 
Galois extensions of $R_{\eta}$. In characteristic $p$ 
all cyclics extend, so we lift an extension of 
$\hat T_{\eta}/\hat R_{\eta}$ and we are done.
\end{proof} 

If we form $R'_{\eta} = R' \otimes_R R_{\eta} = R_{\eta}(1/\eta)$, 
this is the field of fractions of $R_{\eta}$ and we write it as 
$K$. For any $n$ we can set $K_{n} = 
R' \otimes_{\Z[\mu]'} \Z[\mu_{n-1}]'$ when $p > 2$ 
and $R' \otimes_{\Z[\mu]'} \Z[\mu_{n-2}]'$ when $p = 2$. 
The point is, again, that $K_n$ is generated over $K$ 
by adjoining a primitive $p^n$ root of unity. 
$K_n$ is the field of fractions 
of the complete discrete valuation ring 
$W_n = R_{\eta} \otimes_{\Z[\mu]} \Z[\mu_{n-1}]$ 
or $R_{\eta} \otimes_{\Z[\mu]} \Z[\mu_{n-2}]$ 
when $p = 2$. 
Since the prime $\eta_{n} = 
\mu_n - 1$ is totally ramified over $\eta$, 
$W_n$ is a complete discrete valuation domain as claimed 
and $K_n$ is a field. 
Let $N_{n}: K_{n}^* \to K^*$ be the norm map. 

Now continuing with the proof of \ref{inductionstep}, 
we notice that if $W/R_{\eta}$ is as in \ref{canextend}, 
then $W'_{\eta} = W \otimes_{R} K$ and $T_{\eta}'' = T'' \otimes_{R'} K$ 
have a common subextension $S'_{\eta}$ of degree 
$p^{r-1}$ and so $W'_{\eta}D_1' = T''_{\eta}$ where 
$D_1'/K$ is cyclic of degree $p^{m-r+1}$. 
Looking at degree $p^r$ subextensions over $K$, 
we have that $D'_{\eta} = D' \otimes_{R'} K$ 
is the degree $p$ subextension of $D_1'/K$. 
That is, $D'_{\eta}$ extends to a cyclic of degree 
$p^{m-r+1}$. By \ref{invert}, in $K^*$, 
we have $(1 + a\eta^s) = z^pN_{m-r+1}(u)$ for 
$z \in K^*$ and $u \in K_{m-r+1}^*$. 
Since the left side is a unit, and using the total ramification 
again, $u$ must have $\eta_{m-r+1}$ valuation a multiple 
of $p$. Since $N_{m-r+1}(\eta_{m-r+1}) = \eta$, 
we can alter $z$ by a power of $\eta$ and assume 
$u$ and $z$ are both units in $W_{m-r+1}$ and $R_{\eta}$ 
respectively. These rings have the 
same residue fields (total ramification again), 
and the norm map is the $p^{m-r+1}$ power map in the residue 
field. Thus there is a $y \in R_{\eta}^*$ such that 
$N_{m-r+1}(u)$ is congruent to $y^{p^{m-r+1}}$ 
modulo $\eta$ and replacing $u$ by $u/y$ 
we can assume $u$ maps to $1$ module $\eta_{m-r+1}$. 
It follows that $z^p$ and hence $z$ maps to 1 
modulo $\eta$. All together, we have $z = 1 + h\eta$ and 
$u = 1 + b\eta_{m-r+1}^k$ for $b \in W_{m-r+1}$ for some 
$k \geq 1$. Finally $z^p = 1 + h\eta^p$. 
Thus $N_{m-r-1}(u)$ and $1 + a\eta^s$ are congruent 
modulo $\eta^p$ and we are assuming $\eta$ does not divide 
$a$ and $s \leq p - 1$. By \ref{normcomputations}
we have $k = s$.  

Quoting \ref{normcomputations}, $a$ is congruent to 
$b^{p^{m-r+1}}$ or $b^{p^{m-r+1}} - b^{p^{m-r}}$ 
Let $\hat b$ be the image of $b$ is the residue 
field $F_p(x)$. 
Since $a \in R$, the image $\hat a \in F_p[x]$. 
Since $F_p[x]$ is integrally closed, $\hat b \in F_p[x]$. 
Choose $b' \in R$ a preimage of $\hat b$. 
Let $E'/R'$ be defined as in \ref{taualone}.  
That is, $E'/R'$ is the $\tau$ invariant subring of $R'((M_{\tau}(1 + b'\eta_m^s)^{1/p^{m-r+1}})$. Let $E_1'/R'$ be the degree 
$p$ subextension, so $E_1' = R'(N_{m-r+1}(1 + b'\eta_m^s)^{1/p})$ 
and $N_m(1 + b'\eta_m^s)$ is congruent to $1 + b\eta^s$ 
modulo $\eta^{s+1}$. 
Replace $T'/R'$ by $T'E'/R$, and note that $T'E' \supset S'$ 
and the degree $p$ subextension of $T'E/S'$ 
is $T_1'E_1' = T(1/\eta)D'E_1'$. $(D'E_1')/S'$ is 
defined by $S'((1 + b''\eta^{s'})^{1/p})$ where $s' > s$ and we proceed 
by induction until we can assume $s \geq p$ which finishes 
\ref{inductionstep}
\end{proof}

With \ref{inductionstep} proven, we have shown by induction that: 

\begin{theorem}\label{allothersteps}
Let $R$ be as in Case A or B and very special. 
Suppose $U/R$ is cyclic of degree $p$ 
with $\hat U = F_p[x][Z]/(Z^p - Z - x)$. 
Assume further that $U' = U(1/\eta)$ extends to a
degree $p^m$ cyclic extension of $R' = R(1/\eta)$. 
Then $U/R$ extends to a degree $p^m$ extension 
of $R$. 
\end{theorem} 

The above is the second step of our ultimate theorem. 
The first step is to construct the $U/R$ assumed there, 
which follows. In order 
to assume as little as possible, we are NOT assuming 
$R$ is very special but only localizing as needed. 

\begin{proposition}\label{firststep} 
Let $\hat U = \hat R[Z]/(Z^p - Z - x)$. Suppose  $m \geq 2$. 
Then there is an $s \in \Z[\mu][x]$ and cyclic Galois degree $p$ extension $U/R$
where $R = \Z[\mu][x](1/s)$ is such that 

1) $s \in 1 + \eta{x}R$.  

2) $U/\eta{U} \cong \hat U$ as a Galois extension of $\hat R$. 

3) $U'/R' = U(1/\eta)/R(1/\eta)$ extends to a cyclic Galois extension $T''/U'/R'$ of degree 
$p^m$ and $T''/MT''$ is split where $M$ is generated by $x$. 

\end{proposition} 

Note that the extension we are creating has trivial Galois 
module structure, and this is not surprising by \ref{specialdegreep}. 

\begin{proof} 
Let $V_m'' = \Z[\mu_{m-1}][x]$ when $p > 2$ and 
$V_m'' = \Z[\mu_{m-2}][x]$ when $p = 2$. As before, 
the point is that $V_m''$ is generated by a primitive 
$p^m$ root of one. Let $\tau$ be such that $\tau(\rho_{s}) = \rho_{s}^{r}$ 
where $r = p + 1$ when $p > 2$ and $r = 5$ when $p = 2$. 
Let $t$ be the order of $r$ modulo $p^m$ which we write as 
$p^n$ in order to cover both cases. 
Clearly $\tau$ acts on $V_m''$ and $(V_m'')^{\tau} = \Z[\mu][x]$. 
Let $N_{\tau}:V_m'' \to \Z[\mu][x]$ be the norm map. 
If $p > 2$ set $z = 1 + x\eta_m^{p+1}$. If $p = 2$ 
set $z = 1 + x\eta_m^2$. If $s = N_{\tau}(z)$ then by  \ref{normcomputations} 
we have $s = 1 + b\eta^p$ where $b$ is congruent to 
$x^{p^{n-1}}$ modulo $\eta$ if $p > 2$ and $x^{2^{n}}$ 
if $p = 2$. To cover both cases let $n' = n - 1$ or $n' = n$ 
respectively. Set 
$R = \Z[\mu][x](1/s)$, $V_m = V_m(1/s)$, $R' = 
R(1/\eta)$ and $V_m' = V_m(1/\eta)$. 
We defined 

$$M_{\tau}(z) = \tau^{t - 1}(z)(\tau^{t - 2}(z))^r\ldots
z^{r^{t-1}}$$ 
and set $W = V'_{m-1}(M_{\tau}(z)^{1/p^m})$.  
Then by \ref{invert} $\tau$ extends to $W$ and we set 
$T' = W^{\tau}$ is which is cyclic 
over $R'$ of degree $p^m$. By \ref{splitprime} $T'/MT'$ is split. 
Let $T'/U'/R'$ be such that 
$U'/R'$ is cyclic of degree $p$. Then by \ref{invert} $U' =  
R'(N_{\tau}(z)^{1/p}) = R'(s^{1/p})$. 

Let $U = R[Z]/(Z^p + g(Z) - b)$ which is cyclic 
Galois over $R$ by \ref{prop:1.1}. Also $U(1/\eta) = U'$ 
by the same proposition. 
Then $\hat U = F_p[x][Z]/(Z^p - Z - x^{p^{n'}})$ which is 
isomorphic to $F_p[x][Z]/(Z^p - Z - x)$. 
\end{proof}

We have both pieces \ref{firststep} and 
\ref{allothersteps} and so we have proven (remember $\Z[\mu] = 
\Z[\rho]$ if $p > 2$ but $\Z[\mu] = \Z[i]$ when $p = 2$): 

\begin{theorem}\label{cyclicpartwithmu} 
Let $R = \Z[\mu][x]_Q$ be very special as above, 
$p$ a prime, and $m \geq 1$. There is a 
cyclic Galois extension $T/U/R$ of degree $p^m$ 
such that $U/R$ has degree $p$ and $\hat U = 
U/\eta{U} = \hat R[Z]/(Z^p - Z - x)$. 
\end{theorem} 

It seems time to switch from cyclic Galois extensions 
to cyclic algebras, which we do in the next section. 

\section{Almost Cyclic Algebras}

It is very important to note that in this section, $R$ is not assumed 
to be a $\Z[\rho]$ or $\Z[\mu]$ algebra. 
That is, we are not assuming any of the cases A,B,C. 

If $S/R$ is $G$ Galois then $\phi: S \otimes_R S \cong \oplus_{g \in G} S$ 
where the $g$ component map is $\phi_g(s \otimes s') = sg(s')$. 
A module over $S \otimes_R S$ can also be thought of as 
an $S - S$ bimodule. The special structure of $S \otimes_R S$ 
allows us to make a few observations about these bimodules. 

\begin{proposition}\label{bimodule}
Let $S/R$ be $G$ Galois and $M$ a module over $S \otimes_R S$. 
Then $S \otimes_R S$ has idempotents $e_g$, for $g \in G$, 
such that $\sum_g e_g = 1$, and $e_ge_{g'} = 0$ 
when $g \not= g' \in G$. $(S \otimes_R S)(1 - e_g)$ 
is generated as an $S \otimes_R S$ ideal by all 
$g(s) \otimes 1 - 1 \otimes s$ for all $s \in S$. 
If $M$ is an $S \otimes_R S$ module, then 
$e_gM = \{x \in M | (g(s) \otimes 1)m = (1 \otimes s)m $ all $ s \in S\}$. 
\end{proposition} 

Let $M_g = e_gM$ above. Viewing $M$ as a bimodule, 
$M_g = \{m \in M | ms = g(s)m $ all $ s \in S \}$, and, 
of course, $e_{g'}M_g = 0$ if $g' \not= g$. 

\begin{proof} 
That $\phi$ is an isomorphism can also 
be expressed by saying that 
$S \otimes_R S = \sum_{g \in G}e_gS$ where 
$e_g$ is an idempotent, $\phi_g(e_g) = 1$ 
and $\phi_{g'}(e_g) = 0$ if $g' \not= g$. 
Note that $e_1$ is just the so called separating 
idempotent since $\phi_1 = \mu$ is the multiplication 
map $S \otimes_R S \to S$. It is straightforward 
that the kernel of $\mu$ is generated as an $S \otimes_R S$ 
module by the set $\{s \otimes 1 - 1 \otimes s | s \in S\}$. 
Of course, $1 - e_1$ generates the kernel of $\phi_1$. 
If $e_1 = \sum_i s_i \otimes t_i$, then 
$e_1$ is uniquely defined by the relations $\sum_i s_it_i = 1$ 
and if $g \not= 1$, $\sum_i s_ig(t_i) = 0$. 
Since a similar characterization holds for $e_g$, 
we have $(g \otimes 1)(e_1) = e_g$ and 
$(S \otimes_R S)(1 - e_g)$ is generated by 
$\{g(s) \otimes 1 - 1 \otimes s | s \in S \}$. 

If $M$ is a bimodule, $M_g = \{m \in M | ms = g(s)m $ all $ s \in S\}$ 
is just the annihilator of all the $g(s) \otimes 1 - 1 \otimes s$ 
which generates the ideal $(S \otimes S)(1 - e_g)$. 
Thus $M_g$ is the annihilator of $1 - e_g$ which is just $e_gM$. 
\end{proof} 

If $M$ is an $S - S$ bimodule, and $M = M_g$, we say 
that $M$ is a $1,g$ bimodule. When $M$ is a $1,g$ bimodule, 
it has the same rank as a left $S$ module and as a right 
$S$ module and so we can talk about its rank unambiguously.

\begin{proposition} \label{injective} 
Let $R$ be a commutative Noetherian domain with fraction field $F$. 
Suppose $A/R$ is an algebra containing the commutative 
subalgebra $S \supset R$ such that $A$, $S$ are finitely generated 
as modules over $R$. We view $A$ as a $S - S$ bimodule 
via left and right multiplication. Suppose $I$, $I'$ are 
sub-bimodules with $II' = S$ and $I'I \subset S$. 

a) Then $I$ is projective 
as a right $S$ module and $I'$ is projective 
as a left $S$ module. 

b) Suppose further that $S/R$ is $G$ Galois 
of degree $n$ 
and $I$, $I'$ are $1,g$ and $1,g^{-1}$ 
bimodules respectively of rank $n$ over $R$. 
Then the multiplication map 
$I \otimes_S I' \to S$ is an isomorphism. 
If $J \supset I$ and $J' \supset I'$ are also 
rank one 
$1,g$ and $1,g^{-1}$ sub-bimodules with 
$JJ' = S$ and $J'J \subset S$ then $I = J$ and $I' = J'$. 
\end{proposition}

\begin{proof} 
Let $a_i \in I$ and $b_i \in I'$ be such that 
$\sum_i a_ib_i = 1$. Define $f_i: I \to S$ by setting 
$f(a) = b_ia$ which is clearly a right $S$ module morphism. 
For any $a \in I$ we have $a = (\sum_i a_ib_i)a = \sum_i a_if_i(a)$ 
which shows, by the dual basis lemma, that $I$ is projective 
as a right $S$ module. Similarly, $I'$ is projective as a 
left $S$ module. 

As for part b), we have both $I$ and $I'$ are faithful 
over $S$ (since $II' = S$) and hence rank one. 
The surjection $I \otimes_S I' \to S$ must split 
with kernel of rank $0$ and hence the kernel must be $0$. 
We have natural maps $I \otimes_S I' \to I \otimes_S J' \to 
J \otimes_S J' \to S$ which are all isomorphisms and so 
$I = J$ and $I' = J'$ by faithful flatness. 
\end{proof}

If $A/R$ is Azumaya and $S \subset A$ is such that 
$S/R$ is $G$ Galois, then $A$ is an $S - S$ bimodule 
by left and right multiplication. We use \ref{bimodule} 
to explore this bimodule structure. 

\begin{proposition}\label{azumayabimodule}
If $R \subset S \subset A$ are as above, then 
$A$ is a projective left $S$ module and projective 
$S \otimes_R S$ module. If $g \in G$, then 
$A_g = \{a \in A | as = g(s)a$ all $s \in S\}$ 
is projective over $S$ and $S \otimes_R S$. 
If $A$ has degree $n$ and $S/R$ has degree $n$, then 
$A$ is a rank one $S \otimes_R S$ projective and 
$A_g$ is projective rank one over $S$. 
\end{proposition}

\begin{proof} 
Since $S$ and $S \otimes_R S$ are separable over $R$, 
and $A$ is projective over $R$, 
the first statement follows from \cite[p.~48]{DI}. 
Since $A_g = e_gA$ is a direct summand of $A$, 
$A_g$ is projective over $S$ and $S \otimes_R S$. 
Checking ranks over $R$ 
yields the rest of the proposition.
\end{proof}

We next take a second look at a familiar object. 
Let $R$ be a commutative ring of characteristic $p$, and $a,b \in R$. 
Let $A = (a,b)$ be the algebra generated over $R$  
by $x,y$ subject to the relations 
$x^p = a$, $y^p = b$, and $[x,y] = xy - yx = 1$. 
Note that $R$ is the center of $A$. 
The algebra $(a,b)$ has an $R$ basis the monomials $x^iy^j$ for $0 \leq i,j \leq p-1$. 
Moreover, as shown in \cite[p.~44]{KOS}, $(a,b)$ is always 
Azumaya over $R$ of degree $p$. In addition, 
the exponent $p$ part of the Brauer group 
is generated by Brauer classes of the 
form $(a,b)$. 

Let us look at these ``differential crossed 
product" algebras a bit more. Consider 
the element $\alpha = xy$. Then 
$y\alpha = (yx)y = (\alpha - 1)y$. 
Note that $x\alpha = x(yx + 1) = (xy)x + x = 
(\alpha + 1)x$. If $a$, or $b$ were invertible 
then $x$ or $y$ respectively would be 
invertible and  
$(a,b)$ would be a cyclic 
algebra. However, in any case, 
$y(\alpha^p - \alpha) = ((\alpha - 1)^p - (\alpha - 1))y = (\alpha^p - \alpha)y $. In a similar 
way $x$ commutes with $\alpha^p - \alpha$. 
Thus $\alpha^p - \alpha \in R$. We compute that 
$\alpha^p - \alpha = n(\alpha) = n(x)n(y) = ab$ 
where $n:A \to R$ is the reduced norm. Let 
$S = R[\alpha] \subset A$. Then 
$S$ is the image of $S' = R[t]/(t^p - t - ab)$ 
and we have the map $\phi: S' \to S \subset A$. 
It is clear that $1, \alpha, \ldots, \alpha^{p-1}$ is a basis for $S$ over $R$ 
since it is a subset of a basis for $A/R$. 
Thus $\phi$ is injective and $S/R$ 
is Galois. 

\begin{lemma} 
The centralizer of $S$ in $A$ is $S$. 
\end{lemma}

\begin{proof} 
Let $G = <\sigma>$ be the Galois group of $S/R$, 
where $\sigma(\alpha) = \alpha + 1$. 
In the notation of \ref{bimodule}, the centralizer 
of $S$ in $A$ is $A_1 = e_1A$. Since 
$x^iy^j \in A_{\sigma^{i-j}}$, 
$e_1x^iy^j = 0$ if $i$ is not congruent to $j$ 
mod $p$. Since $e_1S = S$, $S = e_1A$ is the centralizer of $S$. 
\end{proof} 

We will show that the above algebra $(a,b)$ is ``almost cyclic", a concept we will 
work to define. To justify the definition, we continue to 
study the above algebras. 
Let $J_{\sigma} \subset A$ be the 
$\{a \in A | as = \sigma(s)a$ all $s \in S\}$, 
which we wrote as $A_{\sigma}$ above. 
We note that $J_{\sigma} \supset J_{\sigma}' = 
Sx + Sy^{p-1}$. In a similar way, $J_{\sigma^{-1}} 
\supset J_{\sigma^{-1}}' = Sx^{p-1} + Sy$. We claim: 

\begin{proposition}\label{differential}
$J_{\sigma}$ and $J_{\sigma^{-1}}$ are rank 
one projectives over $S$. 
$J_{\sigma}'J_{\sigma^{-1}}' = J_{\sigma}J_{\sigma^{-1}} = 
S$. $J_{\sigma}^p = S$. Finally $J_{\sigma} = J_{\sigma}'$ and $J_{\sigma^{-1}}' = 
J_{\sigma^{-1}}$. 
\end{proposition}

\begin{proof}  
By \ref{azumayabimodule} $J_{\sigma}$ and $J_{\sigma^{-1}}$ 
are rank one projective over $S$.  
$J_{\sigma} \supset Sx + Sy^{p-1} = J_{\sigma}'$ 
and $J_{\sigma^{-1}} \supset Sx^{p-1} + Sy = J_{\sigma^{-1}}'$. 
It is convenient to work generically and set $R$ to be the 
polynomial ring $F_p[a,b]$ and view $A \subset \Delta(K/F,a)$ 
where $F = F_p(a,b)$ and $K = S \otimes _R F$. 
Proving the equalities for this generic $A$ will imply 
them for all. 

We can write $J_{\sigma}' = I_{\sigma}x$ 
where $I_{\sigma} = S + S(\alpha^{-1}bx) \subset K$ 
since $\alpha{y^{p-1}} = xb$. Similarly $J_{\sigma^{-1}}' = 
I_{\sigma^{-1}}x^{-1}$ where $I_{\sigma^{-1}} = 
Sa + S\sigma^{-1}(\alpha)$ since $y = x^{-1}\alpha = 
\sigma^{-1}(\alpha)x^{-1}$ and we note $I_{\sigma^{-1}} \subset S$. 
Now $J_{\sigma}J_{\sigma^{-1}} = I_{\sigma}xI_{\sigma^{-1}}x^{-1} = 
I_{\sigma}\sigma(I_{\sigma^{-1}}) = 
(S + S\alpha^{-1}b)(Sa + S\alpha) = 
Sa + Sb + S(ab/\alpha) + S\alpha$. 
It will be useful to note that since $ab$ is the norm of $\alpha$, 
$ab/\alpha = \adj(\alpha) = \prod_{i=1}^{p-1} \sigma^i(\alpha)$ 
is the adjoint of $\alpha$. 
Set $L$ to be the apparently smaller ideal 
$Sa + S\alpha + S\adj(\alpha)$. 

\begin{lemma}
We have $L = S$. 
$J_{\sigma}'J_{\sigma^{-1}}' = J_{\sigma}J_{\sigma^{-1}} = S$.  
In addition, $J_{\sigma} = J_{\sigma}'$ and 
$J_{\sigma^{-1}}' = J_{\sigma^{-1}}$. 
\end{lemma}

\begin{proof}
Let $M \subset R$ be a maximal ideal. If $a \notin M$ 
then, of course, $L + M = S$. If $a \in M$, then in 
$A/MA$ we have $\alpha^p - \alpha = 0$ and so 
$S/MS = R/M \oplus \cdots \oplus R/M$ is split and 
$\alpha$ can be identified with $(0,1,\ldots,p-1)$. 
Of course, $\adj(\alpha) = ((p-1)!,0,\ldots,0)$ 
and $\alpha + \adj(\alpha)$ maps to a unit of $S/M$ implying 
that $L + M = S$ again. This shows $L = S$. 
Since $L \subset J_{\sigma}'J_{\sigma^{-1}}' \subset 
J_{\sigma}J_{\sigma^{-1}} \subset S$, all these inclusions 
are equalities. Then Lemma~\ref{injective} shows  
the rest. 
\end{proof} 

We have shown \ref{differential}. 
\end{proof} 

It will be quite important that $L = S$ and that we need 
not use $b$. Also, one should note that all we needed about 
$\alpha$ was that in $A/M$ its rank was either $p$ or $p-1$, 
and this meant that $\alpha$ and $\adj(\alpha)$ generate 
all of $S$. 

There is one more point of view that is useful. 
Since $xy = \alpha$, we can conclude that in any $A/M$, 
it is also true that $x$ and $y$ have rank at least $p-1$. 
If we split $A/M \subset \End_F(V)$, 
then $\bar S = (S/M)F$ can be identified with diagonal 
matrices. If $v_i \in V$ , $0 \leq i \leq p-1$, are eigenvectors for $\bar S$ and thus a 
basis for $V$, the rank restrictions imply that we can assume 
$x(v_i) = v_{i+1}$ but $x(v_{p-1})$ might be $0$. 
Since $xy$ has rank no less than $p - 1$, it follows that 
$y(v_i) = f_iv_{i-1}$ with $f_i \not= 0$ for $i > 0$ 
but $y(v_0)$ might be $0$. In matrix language, $x$ 
has nonzero entries in the superdiagonal and nowhere else 
except possibly the lower left corner. $y$ is subdiagonal 
plus upper right corner possibly. The point is that 
independent of the corners, the diagonal matrices 
plus such super and sub diagonal matrices generate all matrices.  

The picture of the above description is: 

\[
\begin{pmatrix}
\star & \diamond & & & && b \\
\bullet & \star & \diamond & & & &   \\
 & \bullet &   & & & &  \\
&   & \raisebox{.2in}{\rotatebox{-19}{$\ddots$}}&\raisebox{.2in}{\rotatebox{-19}{$\ddots$}}  &  \raisebox{.2in}{\rotatebox{-19}{$\ddots$}}& \\
&  &  &&&  \diamond &  \\
& & & & \bullet & \star & \diamond \\
a & & & && \bullet & \star
\end{pmatrix}
\] 
where $\diamond$ plus $a$ is $x$ and $\bullet$ plus $b$ is 
$y$.  

Let's see that, in general, $J_{\sigma}$ above 
is not principal. We assume $R = F_p[a,b]$ 
as above. Suppose $sx + ty^{p-1}$ generated $J_{\sigma}$ 
over $S$ for fixed $s,t \in S$. Let $f(a,b)$ be the reduced norm  
$N_{A/R}(sx + ty^{p-1})$. Since this is not contained 
in any maximal ideal, it is a unit and thus must be an element 
$d \in F_p$.  
Now $y^{p-1} = y^{-1}b = x\alpha^{-1}{b} = \sigma(\alpha)^{-1}bx$ 
and so $sx + ty^{p-1} = (s + t\sigma(\alpha)^{-1}b)x = 
[(s(\alpha + 1) + tb)/(\alpha + 1)]x$ and taking reduced norms 
we have $N_{A/R}(s(\alpha + 1) + tb)/b = d$. 
This implies $\Delta(S/R,bd)$ is split, clearly false. 

With the above discussion of $(a,b)$, we can move toward 
defining a general almost cyclic algebra. 
We begin with a result showing how we can use $J_{\sigma}$ to 
prove an algebra is Azumaya. Suppose 
$R$ is a domain with field of fractions $F$ and 
$A/R$ is an order such that $A \otimes_R F$ is central simple over 
$F$. Assume $R \subset S \subset A$ 
are such that $S/R$ is $G = <\sigma>$ cyclic Galois of degree $n$. Further assume 
the centralizer, $C(S)$ of $S$ in $A$ is $S$ and $J_{\sigma}$ 
is as above. 

\begin{proposition} \label{azumaya}
The following are equivalent. 
\begin{enumerate}[a.]
\item $A/R$ is Azumaya.\label{itm:a}

\item In $A$, $J_{\sigma}J_{\sigma^{-1}} = S$. \label{itm:b}

\item In $A$, $J_{\sigma}^p = S$.\label{itm:c}
\end{enumerate}
\end{proposition} 

\begin{proof} 
All three of the statements are equivalent to the same statements 
after localizing at all maximal ideals of $R$. Thus we may assume 
$R$ is local and $S$ is semilocal. Hence we may assume 
$\Pic(S) = 1$. 

Assume $A/R$ is Azumaya. It must have degree $n$. 
We saw above that $J_{\sigma}$ 
has rank one over $S$ and so $J_{\sigma} = Su$ and 
$J_{\sigma^{-1}} = Sv$. If 
$M \subset R$ is the maximal ideal then $A/MA$ 
is projective over $S/MS \otimes S/MS$ and so 
if $\bar J_{\sigma} = 
\{\bar a \in A/MA | \bar a\bar s = \sigma(\bar s)\bar a\ $ all $ \bar s \in S/MS\}$ then $J_{\sigma} 
\to \bar J_{\sigma}$ 
is surjective. $A/MA$ must be a classical 
cyclic algebra and so $u$ and $v$ have invertible 
images in $A/MA$. It follows that $u,v$ are both invertible, 
implying 
$Suv = SuSv = S$ and $uv$ is invertible in $S$. 
Similarly, $(Su)^p = Su^p$ and $u^p$ is invertible 
in $S$. This proves \ref{itm:b} and~\ref{itm:c}. 

Next assume \ref{itm:b} or \ref{itm:c}. If $I = J_{\sigma}$, and 
$I' = J_{\sigma^{-1}}$ or $I' = J_{\sigma}^{p-1}$ we can conclude 
from lemma~\ref{injective}  
that $J_{\sigma}$ is rank one projective and so 
$J_{\sigma} = Su$. Similarly, assuming b), $J_{\sigma^{-1}} = Sv$. 
Now $S = SuSv = Suv$ or 
$S = (Su)^p = Su^p$. Thus assuming \ref{itm:b} or \ref{itm:c} we can conclude that 
$u$ is invertible. 
We have that 
$A$ contains $A' = \Delta(S/R,\sigma,a)$ for some invertible $a$. 
Thus $A \cong A' \otimes_R B$ for $B$ the centralizer of $A'$ 
and since $S$ is maximal we have $B = R$. 
\end{proof}

The above motivates the definition of an almost cyclic algebra 
as follows. Let $S/R$ be $G = <\sigma>$ Galois of degree 
$n$. Let $J \in \Pic(S)$ which we identify with an invertible ideal 
in $S$. Then $\tau(J)$ is the $\tau$ twisted 
element of $\Pic(S)$. That is, 
$\tau(J) \cong J \otimes_{\tau^{-1}} S$. 

Let $S[t,\sigma]$ be the formal twisted 
polynomial ring such that $ts = \sigma(s)t$. 
Let $T = S \oplus Jt \oplus (Jt)^2 \oplus \cdots $ 
be the subring. Note that $(Jt)^m = 
J\sigma(J)\cdots\sigma^{m-1}(J)t^m$ and 
$J\sigma(J)\cdots\sigma^{m-1}(J) \cong J \otimes_S \sigma(J) \cdots 
\otimes_S \sigma^{m-1}(J)$. 
Note also that $(Jt)^m$ can be thought of as the $S - S$ bimodule 
or $S \otimes_R S$ module which is 
$J\sigma(J)\cdots\sigma^{m-1}(J)$ as a left $S$ module but the 
right $S$ action is twisted by $\sigma^m$. 
In particular, $(Jt)^n = 
SN(J)t^n$ where 
$$N(J) = (J \otimes \sigma(J) \otimes \cdots \otimes \sigma^{n-1}(J))^G$$ 
is the norm on the Picard group $N: \Pic(S) \to \Pic(R)$. 

We now add the assumption $N(J) \cong R$ which means that there is 
a $G$ preserving isomorphism $\phi: (Jt)^n \cong S$. 
Note that $\phi$ is not unique but is unique up to an element 
of $R^*$. 

Define $\Delta(S/R,\sigma,J,\phi) = T/I$ where 
$I$ is the ideal generated by all elements $x - \phi(x)$ 
for all $x \in (Jt)^n$. Clearly $\Delta(S/R,\sigma,J,\phi) = 
S \oplus J_1 \oplus \cdots \oplus J_{n-1}$ 
where $J_i \cong (Jt)^i$ as an $S$ bimodule as above. 
Using \ref{azumaya} we have:  

\begin{lemma}
$\Delta(S/R,\sigma,J,\phi)$ is Azumaya over $R$.
\end{lemma} 

The above is what we will call an {\bf almost cyclic} algebra, 
and it is very clear that our differential crossed 
product is such an algebra. This will be even clearer 
after the next lemma. However, we first give another degree $p$ 
almost cyclic algebra. This construction appeared 
in the preprint \cite{S2022}, but needs to be repeated here. 
For convenience, we make the construction 
generically and the general case is then just all 
specializations. 

Let $R$ be the polynomial $\Z[\rho]$ algebra $\Z[\rho][a,b](1/(1 + ab\eta^p)$. 
Let $F$ be the field of fractions of 
$R$. Let $B/F$ be the symbol algebra $(1 + ab\eta^p,a)$ 
generated by $\beta$, $x$ where $x^p = a$, $\beta^p  = 
1 + ab\eta^p$ and $x\beta = \rho\beta{x}$. 
Set $\alpha = (\beta - 1)/\eta$ so 
$\alpha^p + g(\alpha) = ab$ and $y = x^{-1}\alpha$ 
so $y^p = a^{-1}(ab) = b$ and $xy = \alpha$. Note that when 
$p = 2$, $x$ has norm $-a$ and $\alpha$ 
has norm $-ab$ so this is true in that case also. 

Further note that $x\alpha = 
(\rho\alpha + 1)x$ and $\alpha{y} = y(\rho\alpha + 1)$. 
Let $A \subset B$ 
be $R$ submodule spanned by $\{ x^iy^j | 0 \leq i,j \leq p-1\}$. 
Since $yx = x^{-1}\alpha{x} = \rho^{-1}\alpha - \rho^{-1}$, 
it is clear that $A$ is an order and $S = R[\alpha]$ 
is cyclic Galois over $R$. We denote this algebra 
$(a,b)_{\rho}$. 

\begin{lemma}\label{almostrho} 
$(a,b)$ and $(a,b)_{\rho}$ are almost cyclic with respect to $R[xy]/R$ 
and hence Azumaya over $R$. 
\end{lemma} 

\begin{proof} 
We perform the proof in the $(a,b)_{\rho}$ 
because the other case is basically done already.  
Let $J_{\sigma} = Sx + Sy^{p-1}$ and 
$J_{\sigma^{-1}} = Sx^{p-1} + Sy$. 
Note that $y^{p-1}x^{p-1} = by^{-1}ax^{-1} = 
ab(\alpha)^{-1} = \adj(\alpha)$. 
Then $J_{\sigma}J_{\sigma^{-1}} = 
(Sx + Sy^{p-1})(Sx^{p-1} + Sy) = 
Sa + Sb + S\alpha + S\adj(\alpha)$. 
Set $L = Sa + S\alpha + S\adj(\alpha)$. 
If $M \subset R$ is maximal, and $a \notin M$, 
then clearly $L + MS = S$. If $a \in M$ then modulo $M$ 
$\alpha$ is a root of $Z^p + g(Z) = 0$ 
and so $S/MS$ is split and $\alpha = 
(0,1,1 + \rho,\ldots,1 + \cdots + \rho^{p-2})$.  
We have $\adj{\alpha} = (\gamma,0,\ldots,0)$ 
where $\gamma$ is the product of the nonzero 
entries of $\alpha$. Clearly $S\alpha + S\adj(\alpha)  + 
MS = S$ in this case also so $L = S$.
\end{proof} 

In the above result we showed that the generic $(a,b)_{\rho}$ 
was almost cyclic. If $R'$ is an arbitrary $\Z[\rho]$ algebra 
and $a',b' \in R'$ and $1 + a'b'\eta^p \in {R'}^*$ 
then there is a $\phi: R \to R'$ defined by setting 
$\phi(a) = a'$ and $\phi(b) = b'$ and we define 
$(a',b')_{\rho} = (a,b)_{\rho} \otimes_{\phi} R'$. 
The above description shows: 

\begin{corollary} 
$(a',b')_{\rho}$ is almost cyclic. 
\end{corollary} 

Note that if $\phi(\eta) = 0$, $(a',b')_{\rho}$ is just 
$(a',b')$ and so the above corollary applies. 

For our applications, there will be special cases of 
almost cyclic algebras that will concern us. However, 
before we turn to that, we next give some properties of general almost cyclic 
algebras that will allow some 
perspective on the special cases. 
In particular, we will look at 
tensor products and powers of almost cyclic algebras. 
Complicating matters, there is an ambiguity here we will have to deal 
with. If $A/R$, and $B/R$ are Azumaya algebras with 
a maximal commutative subring $S/R$ (meaning $A$ and $B$ are projective 
over $S$ also), then $A \otimes_R B$ 
is split by $S$ and so there are Azumaya algebras $C$ 
Brauer equivalent to each $A \otimes B$ with 
$S/R$ as a maximal commutative subring. However, such a $C$  
is not, in general, unique and so to prove isomorphism results 
we need to specify $C$ more precisely. 
Similarly, if $m$ divides $n$ and $S/R$ is cyclic Galois of degree $n$, 
there is a $S'$ such that $S \supset S' \supset R$, and both $S/S'$ and $S'/R$ 
are cyclic Galois of degrees $m$ and $k = n/m$ respectively. 
Then if $A = \Delta(S/R,\sigma,I,\phi)$ we know that $A^m = A \otimes_R \cdots 
\otimes_R A$ is split by $S'$ but again there is no unique $C$ 
in this Brauer class with $S'$ maximal commutative. 
However, one can make a natural choice of $C$ 
in both cases as follows. 

Suppose 
$A = \Delta(S/R,\sigma,I,\phi)$ and $B = \Delta(S/R,\sigma,J,\psi)$ are almost cyclic 
algebras of degree $n$. 
The multiplication map $\psi: S \otimes_R S \to S$ defines an idempotent 
$e_S$ such that $\psi(e_S) = 1$ and $\psi$ restricts to an isomorphism 
$e_S(S \otimes_R S) \cong S$. 
Then $e_S(A \otimes_R B)e_S$ is an Azumaya algebra with maximal commutative 
subalgebra $S$. We will further describe $A$ below. 

In the second case, suppose $R \subset S' \subset S$ are as above. 
We need to specify an Azumaya algebra in the Brauer class of 
$A^m$ with $S'$ maximal commutative, and here the task is a bit 
more complicated. 
First there is 
a homomorphism $\phi': S^m = S \otimes_R \cdots \otimes_R S 
\to S \otimes_{S'} \cdots \otimes_{S'} S = S_{S'}^m$. 
To $\phi'$ there is an associated idempotent $e' \in S^m$ 
such that $\phi'$ induces $e'S^m \cong S_{S'}^m$. 
Note that $A' = e'A^me'$ contains $S^m_{S'} \supset S' \supset R$ 
as a maximal commutative subalgebra because $S^m_{S'} = e'S^m = e'S^me$.  
If $A'' \subset A'$ is the centralizer of $S'$, then 
$A'' = B \otimes_{S'} \cdots \otimes_{S'} B = B_{S'}^m$ 
where $B$ is the centralizer of $S'$ in $A$. In particular, 
$B/S'$ has degree $m$ and has $S$ as a maximal commutative 
subalgebra. 

We recall from \cite[p.~33]{S1999} that: 

\begin{proposition}\label{exterior}
Suppose $C/T$ is Azumaya of degree $m$. Then 
$C^m = C \otimes_T \cdots \otimes_T C$ has an idempotent 
$f$ such that $fC^Mf \cong C$ and if $C \otimes_T T' \cong \End_{T'}(V')$ 
then $f$ is the idempotent associated to $V'^m \to \bigwedge^n V'$.
\end{proposition}

Let $f \in A''$ be the idempotent 
guaranteed by \ref{exterior} so $fA''f = S'$. 
Set $e_{m,S'} = f = e'f = fe'$. 
Note that $e_{m,S'}A^me_{m,S'} = fA'f$. Clearly $S' \subset fA'f$ 
and the centralizer of $S'$ in $C = fA'f$ is $fA''f = S'$ and 
$C$ is projective over $S'$. 

\begin{theorem}\label{operationsoncyclics} 
Let $A = \Delta(S/R,\sigma,I,\phi)$ and 
$B = \Delta(S/R,\sigma,J,\psi)$ 
be almost cyclic algebras and let $mk = n$. Then: 
\begin{enumerate}
\item $e_S(A \otimes_R B))e_S \cong \Delta(S/R,\sigma,IJ,\phi \otimes \psi)$, \label{item:es}

\item $C = e_{m,S'}A^me_{m,S'} \cong \Delta(S'/R,\sigma,N_{S/S'}(J),\phi)$ 
where $\phi$ makes sense because $N_{S'/R}(N_{S/S'}(J)) = 
N_{S/R}(J)$.\label{item:c}
\end{enumerate} 
\end{theorem}

\begin{proof}
The proof of \ref{item:es} is straightforward, so we will only give 
the argument for \ref{item:c}. Let $G = <\sigma>$ be the Galois group 
of $S/R$ and $G/C$ be the Galois group of $S' = S^C$ over $R$ 
where $C = <\sigma^{m/n}>$. 
If $\eta \in G$ let $\bar \eta = \eta{C} \in G/C$. 
For $\bar \eta \in G/C$, let $K_{\bar \eta} \subset C$ 
be $\{ x \in C | xs' = \bar \eta{s'}x $ all $s' \in S' \}$. 

We claim: 

\begin{lemma} 
$e_{m,S'}J_{\eta}^me_{m,S'} = K_{\bar \eta}$. 
\end{lemma}

\begin{proof} 
First of all we only give the argument when 
$\eta = \sigma$ as it contains all the essentials. 
Furthermore, to prove this fact we can extend scalars and assume 
$S = R \oplus \cdots \oplus R$ is split, and  
$A \cong \End_R(V)$ where $V \cong S$. 
Then $V$ has an $R$ basis $v_0,\ldots,v_{n-1}$ 
of eigenvectors for $S$. If $e_i\in S$ is the idempotent 
associated to $v_i$, we can assume $\sigma(e_i) = e_{i+1}$ 
and $\sigma(e_{n-1}) = e_0$. If $u \in J_{\sigma}$ 
then $u(e_i) = u_{i+1}e_{i+1}$ for $u_{i+1} \in R$ 
and $u^n = \prod_i u_i$. 

Now $S' \subset S$ is the subring 
such that if $s' \in S'$ is written $s' = \sum r_ie_i$, 
then $r_i = r_j$ whenever $k = n/m$ divides $i-j$. 
We set $w_i = v_{ik} + \cdots + v_{ik + k-1}$ and note that 
the $w_0,\ldots,w_{m-1}$ form an $S'$ basis for $V$. 

The idempotent $e_{m,S'} \in A^m = \End_R(V^m)$ 
corresponds to an $R$ direct summand of 
$V^m = V \otimes_R \cdots \otimes_R V$. 
In fact applying $e'$ and then $f$, this direct summand 
arises from the natural map $V^m \to \bigwedge^m_{S'} V$. 
Now $\bigwedge^m_{S'} V$ is the free $S'$ module 
of rank one with generator $w_0 \wedge \cdots \wedge w_{m-1}$. 
Thus in this special case $C = e_{m,S'}A^me_{m,S'} = 
\End_R(S'x_0)$. 

For $i = 0,\ldots,k-1$, let $f_i = 
\sum_{j=0}^{m-1} e_{i + jk} \in S'$ be the idempotents, 
so that the $f_i$ span $S'$ over $R$. Note that 
$\sigma(f_i) = f_{i+1}$ and $\sigma(f_{k-1}) = f_0$.
The $x_i = f_i(w_0 \wedge \cdots \wedge w_{m-1})$ span 
$\bigwedge^m_{S'} V$ over $R$. 
We have $x_i = v_i \wedge v_{i+k} \wedge \cdots \wedge v_{i + (m-1)k}$. 
 
Suppose $z_0,\ldots z_{m-1} \in J_{\sigma}$ 
and define $z_{i,j}$ by $z_i(v_j) = z_{i,j+1}v_{j+1}$. 
Then direct computation shows that 
$e_{m,S'}(z_0 \otimes \cdots \otimes z_{m-1})e_{m,S'}(x_j) = 
(\prod_i\prod_l z_{i,j + 1 + l(m - 1)})x_{j+1}$ 
and $e_{m,S'}(z_0 \otimes \cdots \otimes z_{m-1})e_{m,S'}(x_{m-1}) = 
(\prod_i\prod_l z_{i,l(m - 1)})x_{0}$. 
This shows $e_{m,S'}J_{\eta}^me_{m,S'} \subset K_{\bar \eta}$ 
in this case. 

To prove 
equality we may take a faithfully flat extension 
and thereby assume $J_{\sigma} = Su$ for $u$ a unit.  
A further faithfully flat extension allows us to assume 
$u^n = 1$. Changing the $v_i$ if necessary, we can assume 
$u(v_i) = v_{i+1}$ and $u(v_{m-1}) = u_0$. 
The computations above show that if 
$v = e_{m,S'}(u\otimes\cdots\otimes{u})e_{m,S'}$ then 
$v(x_i) = x_{i+1}$ and $v(x_{k-1}) = x_0$. 
Clearly $K_{\sigma} = S'v$ and the lemma is proven. 
\end{proof}

We now return to the proof of part 2) of the theorem. 
We saw in the proof of the lemma above that 
after a faithfully flat extension $K_{\sigma} = 
S'v$ and the parallel argument for $\sigma^{-1}$ 
makes it clear that $K_{\sigma}K_{\sigma^{-1}} = 
S'$. We need $e_{m,S'}J_{\sigma}^me_{m,S'} \cong 
N_{S/S'}(J)$ as $S'$ modules which is the observation
that $\bigwedge_{S'}^m(J) \cong N_{S/S'}(J)$. 
\end{proof}

Picard group elements can be difficult to work 
with and difficult to ``lift". However, there is 
a special case of almost cyclic algebras that is much easier. 
Suppose $S/R$ is a cyclic Galois extension of degree 
$n$ with $R$ a domain and $a \not= 0$. 
Set $F = q(R)$ and $K = S \otimes_R F$. Form the 
central simple algebra $B = \Delta(K/F,\sigma,a)$ 
where $a \in R$. Suppose $B$ contains $x,y$ 
such that $x^n = a$, $y^n = b$, $xs = \sigma(s)x$ and 
$sy = y\sigma(s)$ for all $s \in S$. Further assume 
$\alpha = xy \in S$ is such that 
the ideal $Sa + S\alpha + S\adj(\alpha) + Sb = S$. 
Let $A = \Delta(S/R,a,\alpha,b)$ be the subalgebra 
generated by $S$, $x$ and $y$.

\begin{proposition} $A$  
is almost cyclic and Azumaya over $R$.
\end{proposition}  

\begin{proof} 
We have really already some of this argument in pieces, 
but we will gather here some key details. 

First, we verify specifically that $A$ is an order. 

\begin{lemma}
$A$ is spanned as a $R$ module by the $Sx^i$ and $Sy^j$ 
for $0 \leq i,j \leq n-1$. 
\end{lemma}

\begin{proof}
Note that $x^iS = Sx^i$ and $y^jS = Sy^j$. 
Also, 
$$x^iy^i = \sigma^{i-1}(\alpha)\cdots\alpha \in S.$$  
Thus $Sx^iy^j \subseteq Sx^{i-j}$ if $i \geq j$ and 
$x^iSy^j \subseteq Sy^{j-i}$ if $j \geq i$. 
Of course $(Sx^i)x^j = Sx^{i+j}$ if $i + j < n$ 
and $Sx^ix^j \subseteq Sx^{i+j-n}$ if $i + j \geq n$, 
and similarly for $Sy^iy^j$. Thus if $A'$ is the $R$ submodule 
generated by the $Sx^i$ and $Sy^j$ then $A'$ is finitely 
generated and closed under multiplication.
\end{proof}

Returning to the whole argument, 
let $J_{\sigma}' = Sx + Sy^{n-1}$ and $J_{\sigma^{-1}}' = 
Sx^{n-1} + Sy$ which are contained in $J_{\sigma} = 
\{z \in A | zs = \sigma(s)z\}$ and $J_{\sigma^{-1}} = 
\{z \in A | sz = z\sigma(s)\}$ respectively and are both bimodules. 
Then $J_{\sigma}' = I_{\sigma}x$ where $I_{\sigma} = 
S + S\alpha^{-1}b$ and $J_{\sigma^{-1}}' = I_{\sigma^{-1}}x^{-1}$ 
where $I_{\sigma^{-1}} = Sa + S\sigma^{-1}(\alpha)$. 
Of course $J_{\sigma}'J_{\sigma^{-1}}' = I_{\sigma}\sigma(I_{\sigma^{-1}}) = (S + S\alpha^{-1}b)(Sa + S\alpha) = 
Sa + Sb + \alpha + \adj(\alpha) = S$ which implies 
$J_{\sigma}J_{\sigma^{-1}} = S$. By Lemma~\ref{injective} 
and Proposition~\ref{azumaya} we are done. 
\end{proof} 

If $a$ is zero but $b \not= 0$ one can make the above 
construction but reversing the roles of $x$ and $y$. 
If $a = b = 0$, then the above argument does not work. 
Thus we need: 

\begin{lemma}
Suppose $S/R$ is cyclic Galois of degree $n$ with generator 
$\sigma$. Suppose $\alpha \in S$ has rank $n-1$. 
Let $B$ be the free $R$ algebra generated by 
$S$, $x$ and $y$ subject to $x^n = 0$, $y^n = 0$, 
$xy = \alpha$, and $yx = \sigma^{-1}(\alpha)$, 
$xs = \sigma(s)x$ and $sy = y\sigma(s)$ for all 
$s \in S$. 
Then $B$ is Azumaya of degree $n$, split, and almost cyclic. 
\end{lemma} 

\begin{proof}
We saw above that $B$ is spanned as an $S$ module by 
the $x^i$ and $y^j$. Since $\alpha$ has rank $n-1$, 
$S = R \oplus \cdots \oplus R$ is split. 
Write $\alpha = (a_1,\ldots,a_n)$ so $a_1\ldots{a_n} = 0$ 
and 
the max spec of $R$ is the disjoint union 
of the sets of maximal ideals ${\cal M}_i = \{M | a_i \in M\}$. Since $\alpha$ has rank $n-1$, 
${\cal M}_i \cap {\cal M}_j = \emptyset$ when 
$i \not= j$. 
It follows that $a_i$ and $a_{i+1}\ldots{a_n}$ generate the 
unit ideal and hence that $R \to \oplus R/a_iR$ 
is surjective. The kernel is $Ra_1 \cap \cdots \cap Ra_n$. 
However, set of $s_i = a_1\ldots{a_{i-1}}a_{i+1}\ldots{a_n}$ 
generate the unit ideal and hence this kernel is $0$. 

It suffices to prove the result for each $R/a_iR$. 
Thus we may assume $\alpha = (0,a_2,\ldots,a_n)$ where the 
$a_i$ are units. Since $\adj(\alpha) = (a_2\ldots{a_n},0,\ldots,0)$ 
it follows that $xy$ and $yx$ generate the unit ideal of $S$. 
Thus if $bx = 0$ and $by = 0$ then $b = 0$. 

Let $e_i = (0,\ldots,0,1,0 \ldots,0)$ with the $1$ 
is the $i$ place. Then $B$ is the direct sum of the $e_iBe_j$. 
Also, $e_i(Sx^k)e_i = Se_ie_{i+k}x^k = 0$ if $k > 0$. 
Similarly $e_iSy^ke_i = 0$. It follows that $e_iBe_i = e_iSe_i = 
Re_i$. 

Suppose $b$ centralizes $S$. Write $b = \sum_{i,j}b_{ij}$ 
where $b_{ij} \in e_iBe_j$. Then $e_kb = \sum_j b_{kj}$ 
and $be_k = \sum_j b_{jk}$ so it $b$ centralizes $S$ we have 
$b_{ij} = 0$ for $i \not= j$ and so $b \in \sum_i e_iBe_i = 
Re_i = S$. 

If $J_1 = Sx + Sy^{n-1}$ and $K = Sx^{n-1} + Sy$ then 
$JK = S$ and $KJ \subset S$ so $J$ is projective over $S$, 
and arguing similarly $J_i = Sx_i + Sy^{n-i}$ is projective over 
$S$. 

\end{proof} 

Note that the above construction is really a very special 
case of an almost cyclic algebra. The Picard group 
element is generated by 2 elements, which is quite special. 
In fact, if $I$ is generated by two elements then 
there is no reason to believe $I^r$ is generated 
by 2 elements, so if we change $\sigma$ the algebra 
is no longer special. Similarly if $S \supset S' \supset R$ 
is a tower of cyclic extensions, and $I \in \Pic(S)$, then 
$N_{S/S'}(I)$ need no longer be 2 generated.  Thus the power operation 
of \ref{operationsoncyclics} does not preserve this specialness. 
However, the special case does encompass $(a,b)$ and $(a,b)_{\rho}$.

Recall that in the cases of $(a,b)$ and $(a,b)_{\rho}$ 
we had the stronger property 
$Sa + S\alpha + S\adj(\alpha) = S$, 
and it pays to make this a definition. 
If $a \in R$ we say a $G = <\sigma>$ Galois extension is 
$a -$ {\bf split} if $S/aS \cong R/aR \oplus \cdots \oplus R/aR$ 
is split. 

\begin{lemma} 
Suppose $S/R$ is $G = <\sigma>$ Galois and $a -$
split. Then there is an $\alpha \in S$ and an almost cyclic algebra 
$\Delta(S/R,a,\alpha,b)$. If $R$ is a regular domain, 
the Brauer class of $\Delta(S/R,a,\alpha,b)$ only depends on 
$S/R$ and $a$. 
\end{lemma}

\begin{proof} 
If $S/aS = R/aR \oplus \cdots \oplus R/aR$ then we can 
choose $\bar \alpha = (0,r_2,\ldots,r_n)$ where all the $r_i$ 
are units of $R/aR$. Thus $\bar \alpha$ and $\adj(\bar \alpha)$ generate 
the unit ideal in $S/aS$. Let $\alpha \in S$ be a preimage 
of $\bar \alpha$. Since $n(\bar \alpha) = 0 \in R/aR$, 
$n(\alpha) = ab$ for some $b$. It is clear that 
$Sa + S\alpha + S\adj(\alpha) = S$ and so $\Delta(S/R,a,\alpha,b)$ 
is Azumaya. If $F$ is the field of fractions of $R$, 
then the image of this almost cyclic algebra is 
$\Delta(S/R,a)$ and $\Br(R) \to \Br(F)$ is injective.
\end{proof}  

If $a \in R$ and $\alpha \in S$ are such that $N_{S/R}(\alpha)$ 
is divisible by $a$ and $Sa + S\alpha + S\adj(\alpha) = S$ 
we say that $a,\alpha$ are {\bf suitable} in $S$. 
The algebra $\Delta(S/R,a,\alpha,b)$ we call an almost 
cyclic algebra of the {\bf special sort}. 
Of course, the differential crossed product $(a,b)$ 
has $a,\alpha$ suitable where $\alpha^p - \alpha = ab$.  

We want to use almost cyclic algebras to solve two questions. 
The first involves $p$ divisibility. In \cite[p.~37]{KOS} 
it was shown that 
for any $R$ of characteristic $p$, $\Br(R)$ is $p$ 
divisible and the $p$ torsion part of $\Br(R)$ 
is generated by the classes of algebras of the 
form $(a,b)$. We would, however, like to give an explicit 
algebra such that the class $[A]$ satisfies $p^n[A] = (a,b)$. 
To manage this we can restrict to the case of the polynomial 
ring $R = F_p[a,b]$. To find $A$ as above we first need to 
solve a slightly modified cyclic embedding problem. 

Suppose $S'/R'$ is cyclic Galois of degree $p^m$ and $a -$ split. 
Assume $R'$ is of characteristic $p$. 
We want to find $T'/R'$ such that $T' \supset S' \supset R'$ 
is cyclic Galois of degree $n+1$ and is also $a -$ split. 
This is not difficult using the approach and results 
of \cite{S1981} as follows. 

\begin{theorem}
Suppose $R$ is a commutative ring of characteristic $p$ 
and $S/R$ is cyclic Galois of degree $p^m$. Assume 
$a \in R$ is such that $S/R$ is $a -$ split. Then 
there is an cyclic Galois extension $T/R$ of degree $p^{m+1}$ 
such that $S \subset T$ is the induced subextension and 
$T$ is $a -$ split. 
\end{theorem}

\begin{proof} 
From \cite{S1981} we know that $T'/S/R$ exists extending $S/R$ 
and cyclic of degree $p^{m+1}$. Any two choices of $T'$ 
differ by a cyclic degree extension of $R$. 
That is, if we set $R' = R[x]$, $S' = R'[y]/(y^p - y - x)$, 
there is a degree $p^{m+1}$ cyclic extension $U/R'$ 
with $U \subset T'[x] \otimes_R' S'$ such that any 
extension $T'' \supset S \supset R$ which is 
cyclic degree $p^{m+1}$ can be written $U \otimes_{\phi} R$ 
for some $\phi: R' = R[x] \to R$. In particular, 
there is a $\phi':R[x] \to (R/aR)[x] \to R/aR$ 
where $U \otimes_{\phi'} R/aR$ is the split extension 
of $S/aS$. By taking a preimage of $\phi'(x)$, 
there is a $\phi: R[x] \to R$ such that 
the composition $R[x] \to R \to R/aR$ is $\phi'$. 
Set $T = U \otimes_{\phi} R$.
\end{proof}

Now we have: 

\begin{theorem}\label{pthroot}
Suppose $R$ is a commutative regular domain 
of characteristic $p$ and $S/R$ is cyclic of degree 
$p^r$. Let $A = \Delta(S/R,a,\alpha,b)$ be such that 
$a,\alpha$ are suitable in $S$. Suppose there is 
a degree $p^{m}$ $T/S/R$ which is $a -$ split. Then there is an $\alpha' \in T$ 
with $a,\alpha'$ suitable such that for $b' = n(\alpha')/a$, 
$B = \Delta(T/R,a,\alpha',b')$ satisfies 
$p^{m-r}[B] = [A]$ in $\Br(R)$. 
\end{theorem}

\begin{proof} 
Let $T \supset S \supset R$ be $a - $ split. 
Choose $\alpha' \in T$ a preimage of an element 
$\bar \alpha' \in T/aT$ of the form 
$(0,r_2,\ldots,r_{p^{m+1}})$ where all the $r_i$ are 
units. Then $p^m[B] = [A]$ because this is true 
in $\Br(F)$ where $F$ is the field of fractions of $R$.
\end{proof}

The assumptions above on $R$ may seem restrictive but when 
applied to $R = F_p[a,b]$ we see that any degree 
$p$ $(a,b)$ is Brauer equivalent to $p^m[B]$ 
for some almost cyclic degree $p^{m+1}$ degree algebra 
$B$. In detail, 

\begin{corollary}
Suppose $R'$ is a commutative ring of characteristic $p$, 
$a',b' \in R'$ and $n \geq 1$. Then there is an explicit 
almost cyclic Azumaya $B'/R'$ such that in $\Br(R')$ 
we have $p^n[B'] = [(a',b')]$. 
\end{corollary} 

\begin{proof}
Let $\phi: R = F_p[a,b] \to R'$ be defined by $\phi(a) = a'$ 
and $\phi(b) = b'$. By the above theorem, there is an 
Azumaya $B/R$ such that $p^n[B] = [(a,b)]$ in $\Br(R)$. 
Set $B' = B \otimes_{\phi} R'$. 
\end{proof}

We are finally in a position to prove our surjectivity 
results. 

\section{Cyclic Galois and Brauer Surjectivity}

In the end, our goal is to lift structures from characteristic $p$ 
to characteristic 0. The results, and hence the forthcoming argument, 
are a bit different when $p > 2$ and $p = 2$. 
In both cases the idea will be to build our structures 
over special rings using the corestriction and then  
observe that the general results follow.
For the moment we restrict to case A and B, so 
when $p > 2$ we defined $\rho = \mu$ to be a 
primitive $p$ root of one but when $p = 2$, 
$\mu = i$ a square root of $-1$. 
In both cases we perform our corestriction 
over the following rings.  
In a previous section we built cyclic extensions 
over $R_{\mu} = \Z[\mu][x]_Q$ where $Q = 1 + x\eta\Z[\mu][x]$. 
We use the corestriction argument 
to build them over $R = \Z[x]_{Q_0}$ where $Q_0 = 1 + px\Z[x]$. 
We first must observe: 

\begin{lemma}\label{finite}
$R_{\mu} = R \otimes_{\Z} \Z[\mu]$ 
\end{lemma} 

\begin{proof} 
What this lemma really says is that $R_{\mu} = 
\Z[\mu][x]_{Q_0}$ meaning that any element of  
$Q$ is invertible in $\Z[\mu][x]_{Q_0}$. 
Let $\sigma$ generate the Galois group of 
$\Q(\mu)/\Q$ so when $p > 2$, $\sigma(\mu) = \mu^r$ 
for $r$ mapping to a generator of $(\Z/p\Z)^*$. 
When $p = 2$, $\sigma(i) = -i$. In both cases we set 
$\sigma(x) = x$ so $\sigma$ acts on $\Z[\mu][x]$. 
Then $\sigma(\eta) = u\eta$ where $u \in \Z[\mu]^*$ 
so $\sigma(x\eta\Z[\mu][x]) = x\eta\Z[\mu][x]$. 
Thus if $1 + x\eta{f} \in Q$, $\sigma^i(1 + x\eta{f}) 
\in Q$ and and so the norm $N_{\sigma}(1 + x\eta{f}) 
\in Q$. Since $x\eta\Z[\mu][x] \cap \Z[x] = px\Z[x]$, 
we have $N_{\sigma}(1 + x\eta{f}) \in Q_0$ proving the lemma.
\end{proof}

The point of the above result is that we can use the 
corestriction defined in \ref{corrho} and \ref{degree2cor} 
for $R_{\mu}/R$. 

We now consider the $p > 2$ and $p = 2$ cases separately, 
starting with the first. 
Consider the $p^m$ extension $T_{\rho}/R_{\rho}$ 
extension constructed in \ref{cyclicpartwithmu}. 
Recall that $T_{\rho}$ is $x$ split, and if $U_{\rho} \subset T_{\rho}$ 
is the degree $p$ subextension, then $U_{\rho}/\eta{U}_{\rho} 
= F_p[Z]/(Z^p - Z - x)$. Also, $Q = 1 + \eta{x}\Z[\rho][x]$. 

Let $r(p-1) - 1$ be divisible by $p^m$ and let $V_{\rho}/R_{\rho}$ 
be the $r$ power of $T_{\rho}/R_{\rho}$. Let 
$T/R$ the the corestricted extension of $V_{\rho}/R_{\rho}$ 
from \ref{corrho}. By \ref{splitprime}, $T/R$ is $x$ split. 
By \ref{goodmodp} $T/pT$ is the $r(p-1)$ power of $T_{\rho}/R_{\rho}$ 
and so $T/pT \cong T_{\rho}/\eta{T}_{\rho}$ as Galois 
extensions of $\hat R = F_p[x]$ by \ref{torsion}. We have: 

\begin{proposition}\label{oddlift} 
Let $p > 2$ be prime. 
$T/R$ is a degree $p^n$ Galois extension such that 
$T/pT \cong T_{\rho}/\eta{T}_{\rho}$ as Galois extensions of 
$\hat R$. If $S/R \subset T/R$ is the degree $p$ subextension, 
then $S/pS \cong \hat R[Z]/(Z^p - Z - x)$ as Galois extensions 
of $\hat R$. 
\end{proposition} 

It turns out that using the special lift detailed 
above we get a general lift using a special property 
of the ground rings. This next result details this special 
property. 

\begin{lemma}\label{lifthomop}
Suppose $R = \Z[x_1,\ldots,x_n]_Q$ where 
$Q \subset 1 + p{M}$ and $M$ is the ideal of 
$\Z[x_1,\ldots,x_n]$ generated by the $x_i$'s. 
Suppose $S$ is a commutative ring with 
$1 + p{S} \subset S^*$. Let $\hat S = S/p{S}$ 
and $\psi: S \to \hat S$ be the canonical map. Assume  
$\hat \phi: R \to \hat S$ is a ring morphism. 
Then there is $\phi: R \to S$ such that $\psi \circ \phi = 
\hat \phi$. 
\end{lemma}

\begin{proof} 
Suppose $\hat a_i = \hat \phi(x_i)$. Let $a_i \in S$ be a preimage of 
$\hat a_i$. Define $\phi: \Z[x_1,\ldots,x_n] \to S$ 
by setting $\phi(x_i) = a_i$. If $q = 1 + \eta{f} \in Q'$, 
then $\phi(q) = 1 + \eta(\phi(f)) \in S^*$ and so 
$\phi$ extends uniquely to $R$.
\end{proof}

The construction in \ref{oddlift} was not versal but 
it is not difficult to use it and make a versal construction. 

Let $G$ be the cyclic group of order $p^n$. 
Form $F_p[x_0,\ldots,x_{n-1}]$. For each $x_i$, let 
$R_i = \Z[x_i]_{Q_i}$ where $Q_i = 1 + p{M_i}$ and $M_i$ 
is generated by $x_i$. By \ref{oddlift} there is a cyclic 
degree $p^{n-i}$ extension $T_i/R_i$ such that 
if $S_i/R_i$ is the degree $p$ subextension, 
$\hat S_i = S_i/M_iS_i = \hat R_i[Z]/(Z^p - Z - x_i)$. 
Let $R = R_0 \otimes_{\Z} \otimes R_1 \cdots \otimes_{\Z} 
R_{n-1}$ which is $\Z[x_0,\ldots,x_{n-1}]_Q$ 
for $Q = Q_0\ldots{Q_{n-1}}$. Note that 
$\hat R = R/p{R} = F_p[x_0,\ldots,x_{n-1}]$. 
Moreover if the $T_i$ are viewed as cyclic extensions of 
$R$, we can set $T = T_0\ldots{T_{n-1}}$. 
If $\hat T = T/\eta{T}$ and similarly for the $T_i$, 
we have $\hat T = (\hat T_0)\ldots(\hat T_{n-1})$. 

\begin{theorem}\label{versal}
$T/R$ is versal for characteristic $p$ Galois extensions 
of degree $p^n$. That is, if $pR' = 0$ and $T'/R'$ 
is cyclic Galois of degree $p^n$, there is a 
$\phi: R \to R'$ such that $T' \cong T \otimes_{\phi} R'$ 
as Galois extensions. 
\end{theorem}

\begin{proof}
Write $T'_n \supset T'_{n-1} \supset \ldots T'_1 \supset R'$ 
where $T'_i/R'$ is cyclic Galois of degree $p^i$. Then 
$T'_1 = R'[Z]/(Z^p - Z - a)$ so if we define 
$\phi_0(x_0) = a$ then $T'' = T \otimes_{\phi_0} R'$ has the same 
degree $p$ subextension as $T'$. That is, $T'' = T'S'$ where $S'/R'$ 
is cyclic of degree $p^{n-1}$. By induction there is 
a $\phi': R_1 \otimes \cdots R_{n-1} \to R'$ such that if 
$S = T_1\ldots{T_{n-1}}$ then $S \otimes_{\phi'} R' \cong 
S'$. Let $\phi = \phi_0 \otimes \phi'$. 
\end{proof}

With the above and \ref{lifthomop} we have: 

\begin{theorem}\label{liftwithrho} 
Let $p > 2$. 
Suppose $S$ is a commutative ring, $1 + pS \subset S^*$, 
and $\hat S = S/p{S}$. 
If $\hat T'/\hat S$ is a degree $p^n$ cyclic Galois extension, 
there is a degree $p^n$ cyclic Galois $T'/S$ such that $T' \otimes_S \hat S 
\cong \hat T$ as Galois extensions. 
\end{theorem} 

\begin{proof}
By the above there is a $\hat \phi: R \to \hat S$ such that $T \otimes_{\hat \phi} \hat S\cong \hat T$. Let $\phi: R \to S$ be as in \ref{lifthomop} and set $T' = T \otimes_{\phi} S$. 
\end{proof} 

We know turn to the $p = 2$ version of \ref{oddlift}. 
Here our results are not as a strong because 
arbitrary degree $2^m$ extensions cannot be lifted, 
this being the heart of Wang's counterexample to Grunwald's 
theorem. However, something is true that is not too different. 
Recall that $R_{\mu} = \Z[i][x]_Q$ where 
$Q = 1 + x\eta\Z[i][x]$. Let 
$T_{\mu}/R_{\mu}$ be the degree $2^m$ cyclic Galois extension 
constructed in \ref{cyclicpartwithmu} with group $G$. 
Thus $\hat T_{\mu}/\hat R_{\mu}$ is degree $2^m$ cyclic extension 
of $F_2[x]$ whose degree 2 part is $F_2[x][Z]/(Z^2 - Z - x)$. 
Let $R = \Z[x]_{Q_0}$ where $Q_0 = 1 + 2x\Z[x]$. 
By \ref{finite} $R_{\mu} \cong R \otimes_{\Z} \Z[i]$. 
Thus by \ref{degree2cor} 
we can define $T/R = \Cor_{R/R_0}(T_{\mu}/R_{\mu})$. 

\begin{theorem}\label{2lift}
$\hat T = T/2T = \Ind_H^G(\hat S/\hat R)$ 
where $H \subset G$ are cyclic groups of order 
$2^{m-1}$ and $2^m$ respectively, and $\hat S/\hat R$ 
is $H$ Galois. Furthermore, if $\hat S/\hat U/\hat R$ 
is such that $\hat U/\hat R$ has degree 2, 
$\hat U = \hat R[Z]/(Z^2 - Z - x)$ and $T_0/R_0$ 
is $x$ split. 
\end{theorem}

\begin{proof}
The ring homomorphism $R \to R/2R = F_2[x]$ 
factors as $R \to R_{\mu} \to R_{\mu}/\eta{R_{\mu}} = R/2R$ 
and so 
$T/2T \cong (T \otimes_{R} R_{\mu})/\eta{(T \otimes_{R} R_{\mu})}$ 
as extensions of $F_2[x]$. 
By the definition of the corestriction, 
$T \otimes_{R} R_{\mu} = (T_{\mu} \otimes_{R_{\mu}} \sigma(T_{\mu}))^N/R$ 
where $N = \{(g,g^{-1}) | g \in G \} \subset G \oplus G$. 
Now $\sigma(T_{\mu})/\eta\sigma(T_{\mu}) \cong T_{\mu}/\eta{T_{\mu}}$ 
so 
$$((T_{\mu} \otimes_{R_{\mu}} \sigma(T_{\mu}))^N)/\eta(T_{\mu} \otimes_{R_{\mu}} \sigma(T_{\mu}))^N 
\cong (\hat T_{\mu} \otimes_{\hat R_{\mu}} \hat T)^N$$ 
which is 
$\Ind_H^G(\hat S/\hat R)$ where $\hat S/\hat R$ is the degree 
$2^{m-1}$ part of $\hat T_{\mu}/\hat R_{\mu}$.
\end{proof} 

We are not providing a $p = 2$ version of \ref{versal} 
because using the above construction and the 
strategy from \ref{versal} we get  
an extension that is too big in our estimation. 

The next results we want are lifting results for the Brauer group. 
Here we start with the following facts from \cite{KOS}. 
Recall that if $p\hat S = 0$ and $a,b \in \hat S$ 
we defined the degree $p$ Azumaya algebra $(a,b)$ over $\hat S$ as generated 
by $x,y$ subject to $x^p = a$, $y^p = b$, and $xy - yx = 1$. The facts from \cite[p.~46 and~p.~37]{KOS} are: 

\begin{theorem}
The $p$ torsion part of the Brauer group $\Br(\hat S)$ is generated 
by Azumaya algebras of the form $(a,b)$. Furthermore, 
$\Br(\hat S)$ is $p$ divisible. 
\end{theorem} 

The theorem we are aiming for is: 

\begin{theorem}\label{braueronto}
Let $S$ be a commutative ring 
such that $1 + p{S} \subset S^*$. Set $\hat S = S/pS$. 
Then $\Br(S) \to \Br(\hat S)$ is surjective on the 
$p$ primary components. 
\end{theorem} 

Our strategy for proving this result is embodied in the next result. 
Let $R = \Z[x,y]_Q$ where $Q$ is the multiplicative set 
$1 + pxy\Z[x,y]$. 
Then $\hat R = R/p{R} = F_p[x,y]$. 

\begin{proposition}
Let $\hat R = F_p[x,y]$ and let $(x,y)$ be the degree $p$ 
algebra. Suppose that for any $r$, we have an Azumaya $B_r/R$ 
with image $\hat B_r/\hat R$ such that $p^r[\hat B_r] = [(x,y)]$. 
Then Theorem~\ref{braueronto} is proven. 
\end{proposition} 

\begin{proof}
Assume the givens in the proposition. Let $\hat \beta \in \Br(\hat S)$ 
have order $p^n$. Write $p^{n-1}\hat \beta$ as the product of 
classes of $[(a_i,b_i)]$ for $i = 1,\ldots,m$. 
Define $\hat \phi_i R \to \hat S$ by setting 
$\hat \phi_i(x) = a_i$ and $\hat \phi_i(y) = b_i$. Use \ref{lifthomop} 
to get $\phi_i R \to S$ so that the composition 

$$R \buildrel{\phi_i}\over\longrightarrow S \longrightarrow \hat S$$

is $\hat \phi_i$. 
Set $A_i= B_{n-1} \otimes_{\phi_i} S$. 
If $\beta'$ is the Brauer class $[A_1][A_2]\ldots[A_m]$, 
and $\hat \beta'$ is the image of $\beta'$ in $\Br(\hat S)$, 
then $p^{n-1}\hat \beta' = p^{n-1}\hat \beta$ or 
$p^{n-1}(\hat \beta - \hat \beta') = 0$. We are done by inducting on $n$. 
\end{proof}

\begin{proof}
We turn to the proof of the theorem. 
We have to separate the $p > 2$ and $p = 2$ cases. 
When $p > 2$, we construct the $B_r$'s needed in the proposition 
as follows. 
For $r = 0$ we set $B_0 = (x,y)_{\rho}$. We have an embedding 
$\iota: D' = \Z[z]_{Q'} \to \Z[x,y]_Q = R$ by setting $\iota(z) = xy$ 
where $Q'$ is $1 + pz\Z[z]$ and 
$\iota(Q') \subset Q = 1 + pxy\Z[x,y]$. 
Let $T'/D'$ be the $z$-split degree $p^{r+1}$ cyclic  extension 
constructed in \ref{oddlift} 
and $S'/D'$ its degree $p$ subextension, so 
if $S = S' \otimes_{D'} R$ satisfies $\hat S = S/pS = F_p[Z]/(Z^p - Z - xy)$. 
In $\hat A = (x,y)$ there is are $\alpha$, $\beta$ such that 
$\alpha\beta - \beta\alpha = 1$ and if $\gamma = \alpha\beta$ 
then $\gamma^p - \gamma = xy$. Thus we can view $\hat S \subset \hat A$ 
and $\hat A$ is the almost cyclic algebra 
$\Delta(\hat S/\hat R,x,\gamma,y)$.  
Since $T = T' \otimes_{D'} R$ is $z$ split it is $x$ split. By \ref{pthroot} 
there is a almost cyclic algebra $B_r = \Delta(T/R,x,\alpha',y')$  
with $p^r[\hat B_r] = [\hat A]$ as needed. 

So we assume $p = 2$. Let $S'/R'$ be $2^{m-1}$ cyclic 
Galois and $A' = \Delta(S'/R',x',\gamma',y')$ almost cyclic 
so $\alpha'^{2^{m-1}} = x'$, $\beta'^{2^{m-1}} = y'$, 
$\gamma' = \alpha'\beta' \in S'$ of rank at least $2^{m-1} - 1$. 
Now consider $A = M_2(A')$ and $T' = \Ind_H^G(S')$ 
where $G$ is cyclic of order $2^m$. We view 
$T' \subset A$ as being the diagonal matrices with 
entries from $S'$. 

\begin{proposition}
$A$ is an almost cyclic algebra of the 
form $$\Delta(T/R,x',\gamma,y').$$ 
\end{proposition}

\begin{proof}
Let $\gamma \in T$ be the element 
$$\begin{pmatrix}
\gamma'&0\\
0&1
\end{pmatrix}$$
so that $\gamma$ has the required rank. 
Suppose $\alpha'\beta' = \gamma'$ are the elements of 
$A'$ with $\alpha^2 = x'$ and $\beta'^2 = y'$, 
$\alpha's = \sigma(s)\alpha'$ and $s\beta' = \beta'\sigma(s)$.
Set 
$\alpha \in A$ to be
$$\begin{pmatrix}
0&\alpha'\\
1&0
\end{pmatrix}$$

and $\beta \in A$ to be 
$$\begin{pmatrix} 
0&1\\
\beta'&0
\end{pmatrix}.$$ 

We immediately check that $\alpha\beta = \gamma$ 
and $\alpha^4 = x'$, $\beta^4 = y'$. 
Moreover, for any $(s_1,s_2) \in T$, 
$\alpha(s_1,s_2) = (\sigma(s_2),s_1)\alpha$ and similarly 
for $\beta$. 
\end{proof}

In particular $M_2((x,y))$ is 
$\Delta(\Ind_H^G(\hat S/\hat R,	x,\alpha,y)$ where 
$\hat S = F_2[x][Z]/(Z^2 - Z - xy)$. Using the 
$x$ split lift from \ref{2lift} we proceed exactly as 
in the $p > 2$ case 
\end{proof}

\end{document}